\documentclass[10pt,reqno]{amsart}

\usepackage{hyperref}
\usepackage{amssymb}
\usepackage{amsfonts}
\usepackage{amsthm}
\usepackage{amsmath}
\usepackage{caption}
\DeclareCaptionLabelFormat{matrix}{Matrix #2}
\usepackage{nicematrix}
\usepackage{multirow}
\usepackage{bbm}
\usepackage{xcolor}
\usepackage{mathtools}

\usepackage{bigints}
\usepackage[normalem]{ulem} 
\newcommand\redout{\bgroup\markoverwith
{\textcolor{red}{\rule[0.5ex]{2pt}{0.8pt}}}\ULon}
\usepackage{tikz}
\usepackage{calc}
\usepackage{graphicx}

\usepackage{accents}
\newcommand{\dbtilde}[1]{\accentset{\approx}{#1}}

\usepackage{adjustbox}
\usepackage{stmaryrd}
\newcommand{\parallelsum}{\mathrel{\mkern-5mu\clipbox{0 0.75ex 0 0}{${\sslash}$}\!}}

\DeclareMathOperator{\diag}{diag}

\numberwithin{equation}{section}
\allowdisplaybreaks

\newtheorem{theorem}{Theorem}[section]
\newtheorem{proposition}[theorem]{Proposition}

\newtheorem{props}[theorem]{Properties}
\newtheorem{lemma}[theorem]{Lemma}

\newtheorem{corollary}[theorem]{Corollary}

\theoremstyle{definition}
\newtheorem{definition}[theorem]{Definition}
\newtheorem{exmp}[theorem]{Example}

\theoremstyle{remark}
\newtheorem*{remark}{Remark}
\newtheorem*{remarks}{Remarks}

\newcommand{\R}{\mathbb{R}}

\newcommand{\Z}{\mathbb{Z}}

\newcommand{\1}{\mathbbm{1}}
\newcommand{\HH}{\mathbb{H}}
\newcommand{\I}{\mathbb{I}}
\newcommand{\Q}{\mathbb{Q}}

\newcommand{\Id}{\text{Id}}
\newcommand{\Span}{\text{span}}

\renewcommand{\hat}{\widehat}

\newcommand{\scriptB}{\mathcal{B}}
\newcommand{\scriptC}{\mathcal{C}}

\newcommand{\scriptF}{\mathcal{F}}
\newcommand{\scriptG}{\mathcal{G}}

\newcommand{\scriptK}{\mathcal{K}}

\newcommand{\scriptM}{\mathcal{M}}

\newcommand{\scriptP}{\mathcal{P}}

\newcommand{\scriptX}{\mathcal{X}}

\DeclareMathOperator*{\dist}{dist}

\renewcommand{\eqref}[1]{(\ref{#1})}

\begin{document}

\title{Finner-like inequalities in the Heisenberg group}
\author{Kaiyi Huang}
\address{University of Wisconsin--Madison}
\email{khuang247@wisc.edu}

\begin{abstract}
We completely characterize the range of $L^p$-boundedness of certain multilinear Radon-like transforms involving vertical projections in the Heisenberg group.
\end{abstract}

\maketitle


\section{Introduction}\label{S:intro}
In \cite{finner1992generalization}, Finner characterized the range of $L^p$-boundedness of the multilinear form
\begin{equation}\label{E:main ml form} 
    \scriptM(f_1,\dots,f_M)\coloneqq\int_\Sigma\displaystyle\prod_{j=1}^Mf_j\circ\pi_j(x)dx,
\end{equation}
with $\Sigma=\R^n$ and $\pi_j$ being \textit{coordinate projections}, i.e., orthogonal projections onto the subspaces spanned by some subsets of the standard basis (such a subspace is called a \textit{coordinate subspace}). This has been generalized to the Brascamp--Lieb inequality, where the $\pi_j$ are arbitrary linear surjections \cite{BCCTlong,BCCTshort} (see also \cite{Lieb}).

Concurrently, there has been a great deal of interest in characterizing the range of $L^p$-boundedness of the \textit{generalized multilinear Radon-like transform} $\scriptM$ in which the $\Sigma$ is a Riemannian manifold and the $\pi_j$ are nonlinear. Of particular interest is the case that the ensemble $\{\pi_j\}_{j=1}^m$ exhibits ``curvature" in the sense of \cite{christ1999singular,TaoWright,stovall2011improving}, potentially leading to better estimates than where the $\pi_j$ are linear. Motivated by these questions, we sharpen certain results from \cite{huang2024inequalities} and prove an analogue of Finner's inequality in the context of the \textit{Heisenberg group}, completely characterizing the range of $L^p$-boundedness of $\scriptM$, where $\Sigma=\HH^n$ is the Heisenberg group and the $\pi_j$ are \textit{vertical projections} whose linear parts are coordinate projections. The \textit{Heisenberg group} $\HH^n=(\R^{2n+1},\bullet)$ is $(2n+1)$-dimensional Euclidean space endowed with a nonabelian group operation $\bullet$ defined by
\begin{equation}
    (x,y,t) \bullet (x',y',t') = (x+x',y+y',t+\tfrac12(x\cdot y' - y\cdot x')),
\end{equation}
where $(x,y,t),(x',y',t')\in\R^n\times\R^n\times\R$. The \textit{vertical projection} $\pi_V:\HH^n\longrightarrow V\times\R$ associated to any subspace $V$ of $\R^{2n}$ is defined by
\begin{equation}
    \pi_V(x,y,t)\coloneqq(x,y,t)\bullet((x,y)_{V^\perp},t)^{-1},
\end{equation}
where $(x,y)_{V^\perp}$ is the orthogonal component of $(x,y)\in\R^{2n}$ in $V^\perp$, the orthogonal complement of $V$. The unique (up to scaling) left Haar measure on $\HH^n$ is the Lebesgue measure on $\R^{2n+1}$.

While the bounds for generalized multilinear Radon-like transforms are well understood in the case that the $\pi_j$ are of corank-$1$ (i.e., fibers of the $\pi_j$ are $1$-dimensional) \cite{christ1998convolution,TaoWright,stovall2011improving,christ2020endpoint}, the case of arbitrary coranks largely eludes us in part because the fibers can interact in complicated ways. In this paper we introduce new tools to address this issue.
    
Let $n,m,M$ be positive integers with $m<M$. Suppose we are given coordinate subspaces $V_1,\dots,V_M$ of $\R^n$ and their respective orthogonal complements $K_j=V_j^\perp$. Let $L_j:\R^n\longrightarrow V_j$ be the orthogonal projections onto $V_j$ for each $j$. For each $x\in\R^n$, let $x_{K_j}$ denote its orthogonal component in $K_j$. In this article, we consider an analogue of \eqref{E:main ml form} with $\Sigma=\HH^n$ and the vertical projections
\begin{equation} \label{E:pi_j vp}
    \pi_j(x,y,t)\coloneqq
    \begin{cases}
        (L_jx,y,t+\frac{1}{2}x_{K_j}\cdot y_{K_j}), & 1\leq j\leq m, \\
        (x,L_jy,t-\frac{1}{2}x_{K_j}\cdot y_{K_j}), & m<j\leq M.
    \end{cases}
\end{equation}
Our main result is as follows:

\begin{theorem}\label{T:main}
For $\pi_j$ of the form \eqref{E:pi_j vp}, there exists $0<C<\infty$ depending only on $n,\pi_j,p_j$ such that
\begin{equation} \label{E:strong}
    \scriptM(f_1,\ldots,f_M) \leq C\prod_{j=1}^M \|f_j\|_{L^{p_j}}
\end{equation}
holds uniformly over continuous functions $f_1,\dots,f_M$ with compact support if and only if $p_j\in[1,\infty]$ for all $1\leq j\leq M$, and
\begin{gather}
\tag{$A$}\label{A} n+1=\sum_{j=1}^{m}\frac{\dim V_j+1}{p_j}+\sum_{j=m+1}^M\frac{n+1}{p_j},\\[8pt]
\tag{$B$}\label{B} \dim V+1\leq\displaystyle\sum_{j=1}^{m}\frac{\dim V_j\cap V+1}{p_j}+\displaystyle\sum_{j=m+1}^M\frac{\dim V+1}{p_j} \\[8pt]
\tag{$C$}\label{C}
\displaystyle\sum_{i=1}^m\frac{\dim K_i\cap V}{p_i}=\displaystyle\sum_{j=m+1}^M\frac{\dim K_j\cap V}{p_j},
\end{gather}
for arbitrary coordinate subspaces $V$ of $\R^n$.
\end{theorem}

\begin{remark}
    Though the conditions on $\{p_j\}_{j=1}^m$ and $\{p_j\}_{j=m+1}^M$ in \eqref{A} and \eqref{B} appear asymmetric, \eqref{C} further implies that for every coordinate subspace $V$ of $\R^n$,
    \[
    \displaystyle\sum_{j=1}^{m}\frac{\dim V_j\cap V+1}{p_j}+\displaystyle\sum_{j=m+1}^M\frac{\dim V+1}{p_j}=\displaystyle\sum_{j=1}^{m}\frac{\dim V+1}{p_j}+\displaystyle\sum_{j=m+1}^M\frac{\dim V_j\cap V+1}{p_j}.
    \]
\end{remark}

\section{Prior results and innovation}\label{S:prior result and innovation}
\subsection{A historical overview}\label{S:lit}
In the case of general $\pi_j$, there is a long history in the search of $p_j$ such that
\begin{equation}\label{E:general multilinear}
    \scriptM(f_1,\dots,f_M)\leq C\prod_{j=1}^M\|f_j\|_{p_j}
\end{equation}
holds uniformly over $f_j$ for some positive finite constant $C$ depending only on $\Sigma,\pi_j,p_j$. In this section, we review some of the prior literature most directly connected with our problem.

When the $\pi_j$ are linear surjections, \eqref{E:general multilinear} is well studied. Examples include H\"older's inequality\cite{rogers1888extension,holder1889ueber}, the Loomis--Whitney inequality\cite{loomis1949inequality}, and Finner's in-equality\cite{finner1992generalization}. This line of work culminates at the result of the Brascamp--Lieb inequality\cite{BrasLiebAdvances73,BCCTlong,BCCTshort,Lieb}, which concerns arbitrary number of linear surjections $\pi_j$ in Euclidean space and characterizes the $p_j$ such that \eqref{E:general multilinear} holds.

Now that the linear case is completely resolved, it is natural to study the nonlinear case. There are two lines of work in this direction. If one perturbs the linear surjections $\pi_j$, then the $L^p$ inequalities for $\scriptM$ continue to hold, as shown in \cite{BBBCF}. However, perturbation can introduce ``curvature" in the sense of \cite{christ1999singular,TaoWright,stovall2011improving}, which can in some instances lead to new inequalities, though a general theory is far from established.

When fibers of $\pi_j$ are curves, we say that $\pi_j$ have \textit{corank-$1$}. Boundedness of \eqref{E:strong} for the corank-$1$ case is proved up to the endpoints\cite{christ1998convolution,TaoWright,stovall2011improving}, with the endpoint case resolved when $X_j$ satisfy certain nilpotency conditions and have polynomial integral flows of bounded degrees\cite{christ2020endpoint}. In addition, Hickman in \cite{hickman2015topics} characterizes the $L^p$-boundedness of a class of dilated average operators along curves.

It is still unclear how to characterize $p_j$ when $\pi_j$ have arbitrary coranks except for some special cases (see \cite{christ2011quasiextremals,gressman2015,gressman2019generalized,gressman2021improving,gressman2023local,gressman2024generalized}). \cite{christ2011quasiextremals} introduces an inflation map to handle the dimension mismatch. Gressman converts a variant of multilinear Radon-like transforms to a sublevel estimate problem, and borrows ideas from Geometric Invariant Theory to prove a class of bilinear Radon-like transforms with mixed coranks\cite{loomis1949inequality,gressman2023local,gressman2024generalized}. In addition, Duncan in \cite{duncan2022global} proves a class of algebraic Brascamp--Lieb inequalities on a scaling line.

Since truly nonlinear inequalities of Brascamp--Lieb type arise from noncommutativity of vector fields whose integral flows are invariant under $\pi_j$, \cite{TaoWright} suggests that one should study model cases such as nilpotent Lie groups, wherein the Heisenberg group is the simplest nontrivial case. (See \cite{huang2024inequalities} for more connections between Radon-like transforms and Heisenberg groups.) \cite{FPmathz} and \cite{zhang2024loomis} formulate and prove the Loomis--Whitney inequality in the Heisenberg group and corank-$1$ Carnot groups, respectively, which are still in the corank-$1$ case. In \cite{huang2024inequalities}, the authors study a symmetric version of \eqref{E:strong}, called the Brascamp--Lieb inequality in $\HH^n$, where $\pi_j$ have arbitrary coranks. Sufficient conditions and necessary conditions are given, but the conditions are not the same in general. The proof of sufficiency is by adapting an induction argument relying on ``critical subspaces" from \cite{BCCTlong,BCCTshort}. Removing the symmetry assumption, we observe that the base case from the induction consists of a new class of multilinear Radon-like transforms. There is no straightforward adaption of known methods to study these forms. New ideas are needed.

\subsection{Innovation of the paper}
In the prior literature, a common scheme to prove the \textit{strong-type inequality} \eqref{E:general multilinear} starts with the \textit{restricted weak-type inequality}
\begin{equation}\label{E:general rwt}
    |\Omega|\lesssim\displaystyle\prod_{j=1}^M|\pi_j(\Omega)|^\frac{1}{p_j}
\end{equation}
Then we study the geometric structure of quasiextremizers of \eqref{E:general rwt} satisfying
\begin{equation}\label{E:general quasiextremal}
    |\Omega|\geq\varepsilon\displaystyle\prod_{j=1}^M|\pi_j(\Omega)|^\frac{1}{p_j}.
\end{equation}
Usually, such $\Omega$ can be approximated by Carnot--Carath\'eodory balls (that are called the paraballs in this article). These balls are analogues of ellipsoids in the sub-Riemannian setting. In the end, we may obtain \eqref{E:general multilinear} by approximating $\scriptM(f_1,\dots,f_M)$ with almost disjoint paraballs. (See \cite{christ1998convolution,TaoWright,stovall2011improving,christ2011quasiextremals,christ2020endpoint}.)

In our context, however, even proving \eqref{E:general rwt} is too hard. Mixed coranks (i.e., $\ker D\pi_j$ vary) often lead to nontrivial intersection of the $\pi_j$ fibers (i.e., $\ker D\pi_j\cap\ker D\pi_{j^\prime}\neq0$). Hence, it is hard to construct an ``almost" diffeomorphism from a ``product" set of $\pi_j$ fibers to a product set of $\Omega$, which is a common technique in the above literature.

In this paper, we first study a class of ``regular" sets (see Section \ref{S:srq}) by introducing some auxiliary maps $\widetilde\pi_j$ (see Section \ref{S:more notations}) with the following properties:
\begin{enumerate}
    \item The $\widetilde\pi_j$ have \textit{nested kernels}, that is, $\ker d\widetilde\pi_j$ is a proper subspace of $\ker d\widetilde\pi_{j+1}$ for each $j$.

    \item The restriction of the $\pi_j$ to the fibers of the $\widetilde\pi_j$ are coordinate projections.
\end{enumerate}

It turns out much easier to construct an almost diffeomorphism via the $\widetilde\pi_j$ fibers thanks to the nested kernels (see Section \ref{S:inflation}). Therefore, we first prove the restricted weak-type inequality for ``regular" sets involving $\widetilde\pi_j$:
\begin{equation}\label{E:rwt auxiliary}
    |\Omega|\lesssim\displaystyle\prod_{j=1}^M|\widetilde\pi_j(\Omega)|^\frac{1}{p_j}.
\end{equation}
Then we restrict our focus to the fibers of $\widetilde\pi_j$, where the $\pi_j$ are essentially coordinate projections. We generalize Finner's inequality\cite{finner1992generalization} concerning coordinate projections to deduce \eqref{E:general rwt} from \eqref{E:rwt auxiliary} (see Section \ref{S:finner}). We approximate by paraballs the quasiextremizers as in \eqref{E:general quasiextremal} within the class of ``regular" sets, and partition a general measurable set into quasiextremal ``regular" sets corresponding to different $\varepsilon$. Adapting the method in \cite{christ2011quasiextremals}, we are able to prove \eqref{E:general rwt} for general $\Omega$. Adapting the scheme one more time, we can eventaully obtain \eqref{E:general multilinear}.

It is hoped that the method can be applied to cases beyond the Heisenberg group.

\subsection*{Outline}
In Section \ref{S:proof overview}, we introduce Theorem \ref{T:base case}, a superficially weaker version of Theorem \ref{T:main}, and provide a proof overview illustrated by some examples. In Section \ref{S:main to weak} we prove by induction Theorem \ref{T:main} assuming Theorem \ref{T:base case}. To prove Theorem \ref{T:base case}, we start with the restricted weak-type inequality concerning characteristic functions of measurable sets. We first introduce a class of auxiliary maps $\widetilde\pi_j$ with nested kernels along with other notations in Section \ref{S:more notations}. To motivate this, we review Finner's inequality and prove a generalization thereof in Section \ref{S:finner}. In Section \ref{S:refinement}, we refine $\Omega$ by trimming $\Omega\cap\widetilde\pi_j^{-1}(z)$. An almost diffeomorphic inflation map is constructed in Section \ref{S:inflation}. A key ingredient to estimate the image of the inflation map is approximation of arbitrary sets by ellipsoids, which will be presented in Section \ref{S:appr}. In Section \ref{S:srq}, we introduce the notions of semiregular sets, regular sets, and quasiextremal sets. In Section \ref{S:rwt semireg}, we prove the restricted weak-type inequality for semiregular sets using the inflation map. In Section \ref{S:regular set paraball}, we introduce the notion of paraballs, an important property of which is that if two paraballs have drastically different $\pi_j$ images or $\widetilde\pi_j$ fibers, then they are almost disjoint. In the same section, we approximate quasiextremal regular sets with paraballs. Section \ref{S:summing disjoint sets} presents a dyadic scheme to sum over almost disjoint sets. In the first part of Section \ref{S:summing quasiextremals}, we obtain paraball-like structure of quasiextremal semiregular sets by summing over almost disjoint regular sets; in the second part of the section, we show the restricted weak-type inequality of arbitrary measurable sets by summing over almost disjoint quasiextremal semiregular sets, a byproduct being that all quasiextremal sets can be approximated by paraballs; the last part passes from the restricted weak-type inequality to the strong-type inequality concerning general functions by summing over almost disjoint quasiextremal sets.

\subsection*{Notations}
Let $V,W$ be two vector spaces. We write $V\leq W$ if $V$ is a subspace of $W$, and $V<W$ if $V$ is a proper subspace of $W$. If $V\leq\R^d$, then we denote by $V^\perp$ the orthogonal complement of $V$ in $\R^d$. For $x\in\R^d$, we denote by $x_V, x_{V^\perp}$ the unique elements in $V,V^\perp$, respectively, such that $x=x_V+x_{V^\perp}$.

For $1\leq j\leq M$, let $V_j\leq\R^n$, and let $L_j:\R^n\longrightarrow V_j$ be the orthogonal projections as in Section \ref{S:intro}. We write $n_j=\dim V_j$, $K_j=V_j^\perp=\ker L_j$, and $k_j=\dim K_j$.

$A\lesssim B$ ($A\gtrsim B$, respectively) means $A\leq CB$ ($A\geq CB$, respectively) for some finite constant $C>0$ depending only on $n,\pi_j,p_j$.

For $x\in\R^d$, $|x|$ denotes its Euclidean norm. Suppose $E\subset\R^d$ is a Lebesgue measurable subset of $\R^k$ naturally embedded into $\R^d$ for some $k\leq d$. Then $|E|_k$ denotes its $k$-dimensional Lebesgue measure for $k>0$. $|\cdot|_0$ denotes the counting measure. When the context is clear, we may safely drop the subscript $k$.

For any set $S$, $\# S$ denotes the cardinality of $S$.

\subsection*{Acknowledgement}
This work was supported in part by NSF DMS-2246906, the Wisconsin Alumni Research Foundation (WARF), a Bung-Fung Lee Torng Graduate Student Summer Fellowship Fund-112550007, and a UW-Madison Mathematics Hirschfelder Fellowship. The author would like to thank her advisor Betsy Stovall for introducing the topic, being generous with her time discussing the problem, carefully reading many versions of the manuscript, and giving many helpful suggestions for improving its readability. The author would also like to thank Philip Gressman for helpful feedback on an earlier version of the manuscript.

\section{Proof overview}\label{S:proof overview}
In this section, we outline the proof of Theorem \ref{T:main}.
Since the proof of necessity can be easily adapted from \cite{huang2024inequalities}, we omit the proof in this article. We will prove sufficiency by induction on $n,M$ following \cite{BCCTlong,BCCTshort,huang2024inequalities}. The base case is subsumed in the following superficially weaker result:

\begin{theorem}\label{T:base case}
Following the notations in Theorem \ref{T:main}, suppose $(p_1,\dots,p_M)\in([1,\infty]\cap\Q)^M$, and
\begin{gather}
\tag{$A$} n+1=\sum_{j=1}^{m}\frac{\dim V_j+1}{p_j}+\sum_{j=m+1}^M\frac{n+1}{p_j},\\[8pt]
\tag{$\widetilde B$}\label{B twiddle} \dim V+1<\displaystyle\sum_{j=1}^{m}\frac{\dim V_j\cap V+1}{p_j}+\displaystyle\sum_{j=m+1}^M\frac{\dim V+1}{p_j}, \\[8pt]
\tag{$C$}
\displaystyle\sum_{i=1}^m\frac{\dim K_i\cap V}{p_i}=\displaystyle\sum_{j=m+1}^M\frac{\dim K_j\cap V}{p_j},
\end{gather}
for arbitrary proper coordinate subspaces $V<\R^n$. Then
\begin{equation}\label{E:base case}
    \scriptM(f_1,\ldots,f_M) \leq C\prod_{j=1}^M \|f_j\|_{L^{p_j}}
\end{equation}
for all continuous, compactly supported functions $f_1,\ldots,f_M$ with some finite constant $C>0$ depending only on $n,\pi_j,p_j$.
\end{theorem}

Let's make a technical change in the definition \eqref{E:pi_j vp} by substituting $x\cdot y$ for $x_{K_j}\cdot y_{K_j}$ in the last coordinate of each $\pi_j$. From now on, $\pi_j$ refers to the following:
\begin{equation}\label{E:pi_j}
    \pi_j(x,y,t)\coloneqq
        \begin{cases}
            (L_jx,y,t+\frac{1}{2}x\cdot y), & 1\leq j\leq m, \\
            (x,L_jy,t-\frac{1}{2}x\cdot y), & m<j\leq M.
        \end{cases}
\end{equation}
The inequalities \eqref{E:base case} with the $\pi_j$ given by \eqref{E:pi_j vp} and \eqref{E:pi_j} are equivalent. This can be seen by a change of variables on the domain of $f_j$, under which $\|f_j\|_{p_j}$ are invariant.

Our first step is to prove the \textit{restricted weak-type inequality}
\begin{equation}\label{E:outline rwt}
    |\Omega|\lesssim\displaystyle\prod_{j=1}^M|\pi_j(\Omega)|^\frac{1}{p_j}
\end{equation}
for bounded Lebesgue measurable sets $\Omega\subset\HH^n$. This is a special case of \eqref{E:base case} where the $f_j$ are characteristic functions. Once the restricted weak-type inequality is obtained, the \textit{strong-type inequality} \eqref{E:base case} can be deduced from the geometric properties of quasiextremals of \eqref{E:outline rwt} and careful dyadic analysis (see \cite{christ2011quasiextremals}). We would like to remark that there are examples of multilinear Radon-like transforms with restricted weak-type inequalities only, but not strong-type inequalities (See \cite{christ2020endpoint}), but this is not the case in our setting.

\subsection{Examples}\label{S:ex}
The base case involves a broad family of multilinear Radon-like transforms that does not seem to have been systematically studied before. In each case, conditions \eqref{A}-\eqref{C} determine a unique $M$-tuple $(p_1,\dots,p_M)$ that also lies in \eqref{B twiddle}. We don't know any direct application of known methods to obtain \eqref{E:outline rwt}. For better illustration, we list four examples of the base case here. Then, we will outline the proof of two of them so that the readers can see the new ideas more clearly. Let's first introduce a notion of ``weights" of certain vector fields:

\begin{definition}\label{D:weight}
    Let $m,M,n$ be positive integers with $m<M$, and let $\pi_1,\dots,\pi_M$ be defined as in \eqref{E:pi_j}. Let $X$ be a vector field whose integral flows are invariant under some $\pi_j$; or equivalently, $X\in\ker D\pi_j$ everywhere. Let $\mathbf p\in[1,\infty]^M$. Then, the weight of $X$ with respect to $\mathbf p$ is defined as
    \begin{equation}\label{E:weight}
        w(X)\coloneqq\displaystyle\sum_{j:X\in\ker D\pi_j}\frac{1}{p_j}.
    \end{equation}
\end{definition}

\begin{exmp}\label{Ex:direct sum}
    Consider $\HH^2\simeq\R^5$, $m=2,M=3$. Let
    \[
    \pi_j(x,y,t)=
    \begin{cases}
        (x_2,y,t+\frac{1}{2}x\cdot y), & j=1, \\
        (x_1,y,t+\frac{1}{2}x\cdot y), & j=2, \\
        (x,t-\frac{1}{2}x\cdot y), & j=3.
    \end{cases}
    \]
    We observe that the fibers of $\pi_j$ are foliated by the vector field $X_j$ for each of $j=1,2$, while those of $\pi_3$ are foliated by the vector fields $Y_1,Y_2$, where
    \[
    X_j=\frac{\partial}{\partial x_j}-\frac{1}{2}y_j\frac{\partial}{\partial t},\quad Y_j=\frac{\partial}{\partial y_j}+\frac{1}{2}x_j\frac{\partial}{\partial t},\qquad j=1,2,
    \]
    generate the Lie algebra of left-invariant vector fields of $\HH^2$. By direct calculation, the only admissible exponents -- i.e., those satisfying conditions \eqref{A}, \eqref{B}, and \eqref{C} -- are $p_1=p_2=p_3=\frac{7}{3}$ (they also satisfy \eqref{B twiddle}). Note that $\ker D\pi_j\cap\ker D\pi_{j^\prime}=0$ for distinct $j,j^\prime$, and the weights of $X_j,Y_j$ are the same: $w(X_j)=w(Y_j)=\frac{3}{7}$ (see Table \ref{tab:direct sum}).
\end{exmp}

\begin{table}
    \centering
    \begin{tabular}{|c|c|c|c|}
        \hline
        $\pi_1$ & $X_1$ & & $p_1=\frac{7}{3}$ \\
        \hline
        $\pi_2$ & & $X_2$ & $p_2=\frac{7}{3}$ \\
        \hline
        $\pi_3$ & $Y_1$ & $Y_2$ & $p_3=\frac{7}{3}$ \\
        \hline
        $w(X_i)=w(Y_i)$ & $\frac{3}{7}$ & $\frac{3}{7}$ & \\
        \hline
        \end{tabular}
    \caption{This chart illustrates the weights from Example \ref{Ex:direct sum}. We list on the $j$-th row the vector fields that everywhere form a basis for $\ker D\pi_j$, and observe that $\ker D\pi_j\cap\ker D\pi_{j^\prime}=0$ for distinct $j,j^\prime$. The weights of the vector fields are the same.}
    \label{tab:direct sum}
\end{table}

\begin{exmp}\label{Ex:isoperimetric}
    Let $\HH^3\simeq\R^7$, $m=3,M=4$. Define
    \[
    \pi_j(x,y,t)=
    \begin{cases}
        (x_1,y,t+\frac{1}{2}x\cdot y), & j=1, \\
        (x_2,y,t+\frac{1}{2}x\cdot y), & j=2, \\
        (x_3,y,t+\frac{1}{2}x\cdot y), & j=3, \\
        (x,t-\frac{1}{2}x\cdot y), & j=4, \\
    \end{cases}
    \]
    In this case, for $j=1,2,3$, the fibers of $\pi_j$ are foliated by $\mathbf{\hat X}_j$, while the fibers of $\pi_4$ are foliated by $Y_1,Y_2,Y_3$, where
    \[
    X_j=\frac{\partial}{\partial x_j}-\frac{1}{2}y_j\frac{\partial}{\partial t},\quad Y_j=\frac{\partial}{\partial y_j}+\frac{1}{2}x_j\frac{\partial}{\partial t},\qquad j=1,2,3,
    \]
    generate the Lie algebra of left-invariant vector fields of $\HH^3$, and $\mathbf{\hat X}_j$ is the pair of vector fields from $X_1,X_2,X_3$ that excludes $X_j$. The only admissible exponents are $(p_1,p_2,p_3,p_4)=(\frac{7}{2},\frac{7}{2},\frac{7}{2},\frac{7}{4})$. Note that each $X_j$ has integral flows invariant under two $\pi_j$, so $\ker D\pi_j\cap\ker D\pi_{j^\prime}\neq0$ for some $j\neq j^\prime$ (see Table \ref{tab:even}), but the vector fields all have weights equal to $\frac{4}{7}$.
\end{exmp}

\begin{table}
    \centering
    \begin{tabular}{|c|c|c|c|c|}\hline
        $\pi_1$ & & $X_2$ & $X_3$ & $p_1=\frac{7}{2}$ \\ \hline
        $\pi_2$ & $X_1$ & & $X_3$ & $p_2=\frac{7}{2}$ \\ \hline
        $\pi_3$ & $X_1$ & $X_2$ & & $p_3=\frac{7}{2}$ \\ \hline
        $\pi_4$ & $Y_1$ & $Y_2$ & $Y_3$ & $p_4=\frac{7}{4}$ \\ \hline
        $w(X_i)=w(Y_i)$ & $\frac{4}{7}$ & $\frac{4}{7}$ & $\frac{4}{7}$ & \\
        \hline
    \end{tabular}
    \caption{This chart illustrates the weights from Example \ref{Ex:isoperimetric}. We list on the $j$-th row the vector fields that everywhere form a basis for $\ker D\pi_j$, and observe that $\ker D\pi_j\cap\ker D\pi_{j^\prime}\neq0$ for some $j\neq j^\prime$. The weights of all the vector fields are equal.}
    \label{tab:even}
\end{table}

\begin{exmp}\label{Ex:aniso}
    Let $\HH^4\simeq\R^9$, $m=3,M=5$. Define
    \[
    \pi_j(x,y,t)=
    \begin{cases}
        (x_2,y,t+\frac{1}{2}x\cdot y), & j=1, \\
        (x_3,y,t+\frac{1}{2}x\cdot y), & j=2, \\
        (x_4,y,t+\frac{1}{2}x\cdot y), & j=3, \\
        (x,y_1,t-\frac{1}{2}x\cdot y), & j=4, \\
        (x,y_2,y_3,y_4,t-\frac{1}{2}x\cdot y), & j=5.
    \end{cases}
    \]
    In this case, for each of $j=1,2,3$, the fibers of $\pi_j$ are foliated by $\mathbf{\hat X}_j$, while the fibers of $\pi_4$ are foliated by $Y_2,Y_3,Y_4$, those of $\pi_5$ foliated by $Y_1$, where
    \[
    X_j=\frac{\partial}{\partial x_j}-\frac{1}{2}y_j\frac{\partial}{\partial t},\quad Y_j=\frac{\partial}{\partial y_j}+\frac{1}{2}x_j\frac{\partial}{\partial t},\qquad j=1,2,3,4
    \]
    generate the Lie algebra of left-invariant vector fields of $\HH^4$, and $\mathbf{\hat X}_j$ is the triple of vector fields from $X_1,\dots,X_4$ that excludes $X_{j+1}$. The only admissible exponents are $(p_1,p_2,p_3,p_4,p_5)=(\frac{31}{5},\frac{31}{5},\frac{31}{5},\frac{31}{10},\frac{31}{15})$. Note that $w(X_j)=w(Y_j)=\frac{10}{31}$ for $j=2,3,4$, while $w(X_1)=w(Y_1)=\frac{15}{31}$ (see Table \ref{tab:uneven}).
\end{exmp}

\begin{table}
    \centering
    \begin{tabular}{|c|c|c|c|c|c|}\hline
        $\pi_1$ & $X_1$ & & $X_3$ & $X_4$ & $p_1=\frac{31}{5}$ \\ \hline
        $\pi_2$ & $X_1$ & $X_2$ & & $X_4$ & $p_2=\frac{31}{5}$ \\ \hline
        $\pi_3$ & $X_1$ & $X_2$ & $X_3$ & & $p_3=\frac{31}{5}$ \\ \hline
        $\pi_4$ & & $Y_2$ & $Y_3$ & $Y_4$ & $p_4=\frac{31}{10}$ \\ \hline
        $\pi_5$ & $Y_1$ & & & & $p_5=\frac{31}{15}$ \\ \hline
        $w(X_i)=w(Y_i)$ & $\frac{15}{31}$ & $\frac{10}{31}$ & $\frac{10}{31}$ & $\frac{10}{31}$ & \\
        \hline
    \end{tabular}
    \caption{This chart illustrates the weights from Example \ref{Ex:aniso}. We list on the $j$-th row the vector fields that everywhere form a basis for $\ker D\pi_j$, and observe that $\ker D\pi_j\cap\ker D\pi_{j^\prime}\neq0$ for some $j\neq j^\prime$. In addition, the weights of the vector fields are not always the same.}
    \label{tab:uneven}
\end{table}

\begin{exmp}\label{Ex: aniso 2}
    Let $\HH^5\simeq\R^{11}$, $m=4,M=6$. Define
    \[
    \pi_j(x,y,t)=
    \begin{cases}
        (x_3,y,t+\frac{1}{2}x\cdot y), & j=1, \\
        (x_2,x_4,y,t+\frac{1}{2}x\cdot y), & j=2, \\
        (x_3,x_4,x_5,y,t+\frac{1}{2}x\cdot y), & j=3, \\
        (x_1,x_2,x_5,y,t+\frac{1}{2}x\cdot y), & j=4, \\
        (x,y_1,t-\frac{1}{2}x\cdot y), & y=5, \\
        (x,y_2,y_3,y_4,y_5,t-\frac{1}{2}x\cdot y), & j=6.
    \end{cases}
    \]
    The only admissible exponents are $p_1=p_2=p_3=p_4=\frac{43}{6},p_5=\frac{43}{18},p_6=\frac{43}{12}$. Let
    \[
    X_j=\frac{\partial}{\partial x_j}-\frac{1}{2}y_j\frac{\partial}{\partial t},\quad Y_j=\frac{\partial}{\partial y_j}+\frac{1}{2}x_j\frac{\partial}{\partial t},\qquad j=1,2,3,4,5.
    \]
    Note that $w(X_j)=w(Y_j)=\frac{12}{43}$ for $2\leq j\leq 6$, while $w(X_1)=w(Y_1)=\frac{18}{43}$ (see Table \ref{tab:aniso 2}).
\end{exmp}

\begin{table}
    \centering
    \begin{tabular}{|c|c|c|c|c|c|c|}\hline
        $\pi_1$ & $X_1$ & $X_2$ & & $X_4$ & $X_5$ & $p_1=\frac{43}{6}$ \\ \hline
        $\pi_2$ & $X_1$ & & $X_3$ & & $X_5$ & $p_2=\frac{43}{6}$ \\ \hline
        $\pi_3$ & $X_1$ & $X_2$ & & & & $p_3=\frac{43}{6}$ \\ \hline
        $\pi_4$ & & & $X_3$ & $X_4$ & & $p_4=\frac{43}{6}$ \\ \hline
        $\pi_5$ & & $Y_2$ & $Y_3$ & $Y_4$ & $Y_5$ & $p_5=\frac{43}{18}$ \\ \hline
        $\pi_6$ & $Y_1$ & & & & & $p_6=\frac{43}{12}$ \\ \hline
        $w(X_i)=w(Y_i)$ & $\frac{18}{43}$ & $\frac{12}{43}$ & $\frac{12}{43}$ & $\frac{12}{43}$ & $\frac{12}{43}$ & \\
        \hline
    \end{tabular}
    \caption{This chart illustrates the weights from Example \ref{Ex: aniso 2}. We list on the $j$-th row the vector fields that everywhere form a basis for $\ker D\pi_j$, and observe that $\ker D\pi_j\cap\ker D\pi_{j^\prime}\neq0$ for some distinct $j,j^\prime$. In addition, the weights of the vector fields are not necessarily the same.}
    \label{tab:aniso 2}
\end{table}

We will see that when the weights of $X_i,Y_i$ are the same for all $i$, as in Examples \ref{Ex:direct sum} and \ref{Ex:isoperimetric}, one can easily adapt the refinement technique in \cite{christ2011quasiextremals} and apply Finner's inequality\cite{finner1992generalization} to prove \eqref{E:outline rwt}. However, the weights are in general not the same, as in Examples \ref{Ex:aniso} and \ref{Ex: aniso 2}, wherein there does not seem to be any straightforward adaption of existing methods to prove \eqref{E:outline rwt}.

To better understand the difficulty caused by varying $w(X_i)$, we will outline the proofs of the restricted weak-type inequalities \eqref{E:outline rwt} in Examples \ref{Ex:isoperimetric} and \ref{Ex:aniso}.

\subsection{Proof sketch for Example \ref{Ex:isoperimetric}}\label{S:proof ex iso} Without loss of generality, assume that $\Omega$ is bounded and has finite measure. We observe that \eqref{E:outline rwt} is equivalent to
\begin{equation}\label{E:even rwt avg}
    |\Omega|^3\gtrsim\left(\displaystyle\prod_{j=1}^3\alpha_j^2\right)\alpha_4^4,\qquad\alpha_j=\frac{|\Omega|}{|\pi_j(\Omega)|},\quad j=1,2,3,4.
\end{equation}
Note that $\alpha_j$ is the average size of $\Omega\cap\pi_j^{-1}(z)$.

Inspired by \cite{TaoWright} and \cite{christ2011quasiextremals}, one may hope to construct an almost diffeomorphic inflation map that maps from a ''product" set of the $\pi_j$ fibers to a product of $\Omega$. The author does not know how to construct such an almost diffeomorphism when $\ker D\pi_j\cap\ker D\pi_{j^\prime}\neq0$, but, equipped with Finner's inequality, we can circumvent this issue by working on the following auxiliary projections:
\[
\pi(x,y,t)\coloneqq(y,t+\frac{1}{2}x\cdot y),\qquad\pi_*(x,y,t)\coloneqq(x,t-\frac{1}{2}x\cdot y).
\]
(Note that $\pi_*=\pi_4$.) Write
\begin{equation}\label{E:outline alpha alpha_*}
    \alpha=\frac{|\Omega|}{|\pi(\Omega)|},\qquad\alpha_*=\frac{|\Omega|}{|\pi_*(\Omega)|}.
\end{equation}
The standard $L^\frac{5}{4}\longrightarrow L^5$ estimate for the Radon-like transform on $\R^4$ implies the following inequality\cite{christ2011quasiextremals,christ2014extremizers}:
\begin{equation}\label{E:outline radon even}
    |\Omega|^3\gtrsim\alpha^4\alpha_*^4.
\end{equation}
This can be done by refining $\Omega$, so that $|\Omega\cap\pi^{-1}(z)|\gtrsim\alpha,|\Omega\cap\pi_*^{-1}(z)|\gtrsim\alpha_*$ whenever they are not empty. We use these refined fibers to construct a product-like set
\begin{equation}
    \Omega^\flat\coloneqq\{(s,\mathbf u,\mathbf v):s\in S; \forall1\leq i\leq3,u^i\in\scriptF(s),v^i\in\scriptG(s,u^i)\},
\end{equation}
where $s\in\R^3,\mathbf u=(u^1,u^2,u^3),\mathbf v=(v^1,v^2,v^3)\in\R^{3\times3}$, $S,\scriptF(s),\scriptG(s,u)\subset\R^3$ are essentially $\pi^{-1}(z)\cap\Omega,\pi_*^{-1}(z)\cap\Omega$, and satisfy
\begin{equation}
    |S|,|\scriptG(s,u)|\gtrsim\alpha,\quad|\scriptF(s)|\gtrsim\alpha_*.
\end{equation}
Let $\Phi:\Omega^\flat\longrightarrow\Omega^3$ be defined as
\begin{equation}\label{E:outline diffeo}
    \Phi(s,\mathbf u,\mathbf v)\coloneqq(e^{v^i\cdot\mathbf X}e^{u^i\cdot\mathbf Y}e^{s\cdot\mathbf X}(z))_{i=1}^3
\end{equation}
for some $z\in\Omega$. Here $e^{V}(z)$ denotes the exponential flow along vector field $V$ starting from $z$ (see Section \ref{S:refinement}). $\Phi$ is a quadratic polynomial that is a local diffeomorphism almost everywhere in $\Omega^\flat$ with Jacobian $D\Phi(s,\mathbf u,\mathbf v)=\det(\mathbf u)$, where $\mathbf u$ is treated as a $3\times3$ matrix with $u^i$ being its columns. By B\'ezout's Theorem and lower bound of $\scriptG(s,u^i)$,
\[
|\Omega|^3\geq|\Phi(\Omega^\flat)|\gtrsim\alpha^3\int_S\int_{(\scriptF(s))^3}|\det(\mathbf u)|d\mathbf uds.
\]
Then \eqref{E:outline radon even} can be obtained by approximating $\scriptF(s)$ with ellipsoids and invoking the lower bounds of $S,\scriptF(s)$.

Note that when restricted to $\omega\coloneqq\pi^{-1}(z)\cap\Omega$, $\pi_1|_\omega,\pi_2|_\omega,\pi_3|_\omega$ are affine coordinate projections, and equality of $w(X_j)$ implies the condition under which Finner's inequality holds true (see Theorem \ref{T:finner weak}). If $|\omega|=\alpha$, and $|\omega|/|\pi_j|_\omega(\omega)|=\alpha_j$ for each $\omega$ (this is a little white lie that helps illustrate how Finner's inequality comes into play), then we have $\alpha\gtrsim\prod_{j=1}^3\alpha_j^\frac{1}{2}$. Thus \eqref{E:even rwt avg} can be deduced from \eqref{E:outline radon even}.

\subsection{Proof sketch for Example \ref{Ex:aniso}}\label{S:proof ex aniso} The reason that the inflation map in \eqref{E:outline diffeo} and use of Finner's inequality work in Example \ref{Ex:isoperimetric} is equality of the weights $w(X_i),w(Y_i)$ for all $i$. Indeed, the proof sketch in Subsection \ref{S:proof ex iso} can be generalized to any case with equal weights, but it cannot be applied to Example \ref{Ex:aniso} or \ref{Ex: aniso 2}, where the weights vary. Therefore, new inflation maps and generalization of Finner's inequality are needed. The main obstruction is that the refinement technique only gives a lower bound $\alpha_j$ of each $\pi_j^{-1}(z)\cap\Omega$, while the upper bound can be arbitrarily large. To tackle the issue, we first prove the restricted weak-type inequality \eqref{E:outline rwt} for $\Omega$ satisfying a certain regularity assumption. Then we can adapt the dyadic method in \cite{christ2011quasiextremals} to remove the extra assumption.

Similarly, \eqref{E:outline rwt} in Example \ref{Ex:aniso} is equivalent to
\begin{equation}\label{E:outline rwt aniso}
    |\Omega|^9\gtrsim\left(\displaystyle\prod_{j=1}^3\alpha_j^5\right)\alpha_4^{10}\alpha_5^{15},\qquad\alpha_j=\frac{|\Omega|}{|\pi_j(\Omega)|}.
\end{equation}
Recall $w(X_i)=w(Y_i)=\frac{10}{31},i=1,2,3$, while $w(X_4)=w(Y_4)=\frac{15}{31}$. Note that $\{X_i\}$ ($\{Y_i\}$, respectively) commute. Write $x=(x_1,x^\prime)\in\R\times\R^3$. Consider
\begin{align*}
    \widetilde\pi_1(x,y,t)&=(x^\prime,y,t+\frac{1}{2}x\cdot y), \\
    \widetilde\pi_2(x,y,t)&=(y,t+\frac{1}{2}x\cdot y), \\
    \widetilde\pi_1^*(x,y,t)&=\pi_6(x,y,t)=(x,y^\prime,t-\frac{1}{2}x\cdot  y), \\
    \widetilde\pi_2^*(x,y,t)&=(x,t-\frac{1}{2}x\cdot y).
\end{align*}
Then, $\ker d\widetilde\pi_1=\Span\{X_1\}<\ker d\widetilde\pi_2=\mathbf X,\ker d\widetilde\pi_1^*=\Span\{Y_1\}<\ker d\widetilde\pi_2^*=\mathbf Y$. Let $$\beta_j=\frac{|\Omega|}{|\widetilde\pi_j(\Omega)|},\qquad\beta_j^*=\frac{|\Omega|}{|\widetilde\pi_j^*(\Omega)|},\qquad j=1,2.$$
We first construct an inflation map by flowing along fibers of $\widetilde\pi_j$ to prove
\begin{equation}\label{E:outline avg medium ex aniso}
    |\Omega|^9\gtrsim\displaystyle\beta_1^5\beta_2^{10}\alpha_4^{10}\alpha_5^{15},
\end{equation}
or equivalently,
\begin{equation}\label{E:outline medium ex aniso}
    |\Omega|\lesssim|\widetilde\pi_1(\Omega)|^\frac{5}{31}|\widetilde\pi_2(\Omega)|^\frac{10}{31}|\pi_4(\Omega)|^\frac{10}{31}|\pi_5(\Omega)|^\frac{15}{31}.
\end{equation}
Note that the exponents in \eqref{E:outline medium ex aniso} satisfy the conditions in Theorem \ref{T:base case}. Moreover, the ratio of $w(X_i)$ in Example \ref{Ex:aniso} is $w(X_1):w(X_2):w(X_3):w(X_4)=3:2:2:2$, same as that in \eqref{E:outline medium ex aniso}.

To see \eqref{E:outline avg medium ex aniso}, write $\mathbf s=(s^i)_{i=1}^9$ with $s^i\in\R^4$. The domain $\Omega^\flat$ of the inflation map takes the form
\[
\Omega^\flat\coloneqq\{(\mathbf s,\mathbf u,\mathbf v)\in \Xi\times\R^{2\times4\times9}: u^i\in\scriptF(s^i),v^i\in\scriptG(s^i,u^i)\},
\]
where $\Xi\subset S^9$, and $S,\scriptF(s^i),\scriptG(s^i,u^i)\subset\R^4$ are essentially $(\widetilde\pi_2)^{-1}(z)\cap\Omega,(\widetilde\pi_2^*)^{-1}(z)\cap\Omega$, and after refinements, have sizes
\begin{gather}
    \frac{|S|}{|L_j(S)|}\gtrsim\alpha_j,\quad\frac{|\scriptG|}{|L_j(\scriptG)|}\gtrsim\alpha_j,\qquad1\leq j\leq3; \\
    \frac{|\scriptF|}{|L_j(\scriptF)|}\gtrsim\alpha_j,\qquad j=4,5; \\
    \label{E:outline beta_1 bound}\frac{|S|}{|P_1(S)|}\gtrsim\beta_1,\qquad\frac{|\scriptG|}{|P_1(\scriptG)|}\gtrsim\beta_1; \\
    \label{E:outline beta_1^* bound}\frac{|\scriptF|}{|P_1(\scriptF)|}\gtrsim\beta_1^*; \\[5pt]
    \label{E:outline beta_2 bound}|\scriptF(s^i)|\gtrsim\beta_2^*,|\scriptG(s^i,u^i)|\gtrsim\beta_2.
\end{gather}
$L_j(x)=x_{j+1}$ are the linear parts of $\pi_j$, and $P_1$ the linear part of $\widetilde\pi_1$, or equivalently, the orthogonal projection onto the last $3$ coordinates. $\Xi$ is an embedding into $S^9$ of some product-like set in $\R^9$ with its $9$-dimensional Hausdorff measure $|\Xi|\gtrsim\beta_1^5\beta_2$. To be specific, the embedding maps $\xi=(\xi^i)_{i=1}^6\in\Xi$ with $\xi^i\in\R^1$ for $1\leq i\leq5$ and $\xi^6\in\R^4$ to Matrix \ref{M:ex_Xi}.
\begin{figure}
    \centering
    \includegraphics[width=1\linewidth]{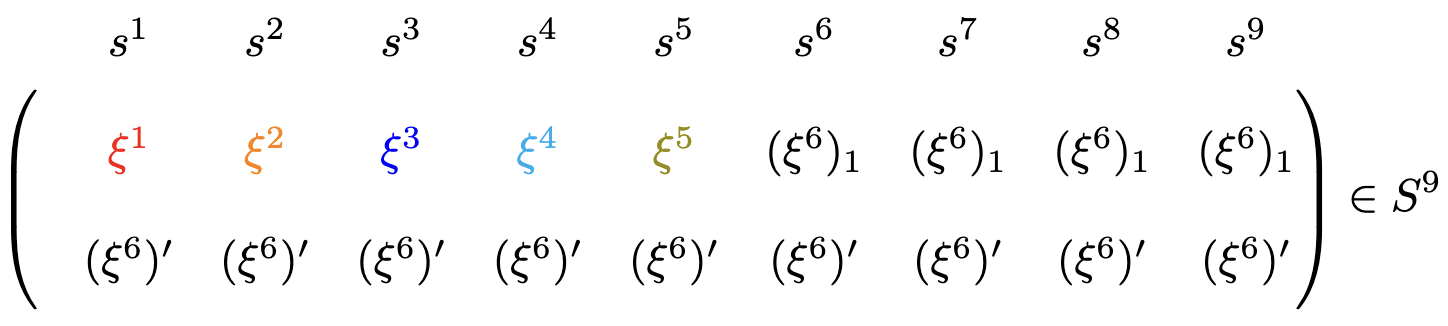}
    \captionsetup{labelformat=matrix}
    \caption{}\label{M:ex_Xi}
\end{figure}

The new inflation map $\Phi:\Omega^\flat\longrightarrow\Omega^9$ is defined as
\begin{equation}
    \Phi(s,\mathbf u,\mathbf v)\coloneqq(e^{v^i\cdot\mathbf X}e^{u^i\cdot\mathbf Y}e^{s^i\cdot\mathbf X}(z))_{i=1}^9
\end{equation}
for some $z\in\Omega$. $\Phi$ is a quadratic polynomial that is a local diffeomorphism almost everywhere in $\Omega^\flat$ with Jacobian $D\Phi(s,\mathbf u,\mathbf v)=\det(\mathbf{u^\prime})\prod_{i=1}^5(u^i)_1$, where $\mathbf{u^\prime}$ is the $4\times4$ matrix with columns $\{u^i\}_{i=6}^9$. B\'ezout's Theorem yields
\[
|\Omega|^9\geq|\Phi(\Omega^\flat)|\gtrsim\beta_2^9\int_\Xi\displaystyle\prod_{i=1}^5\left(\int_{\scriptF(s^i)}|u^i|du^i\right)\left(\int_{(\scriptF(s^6))^4}|\det(\mathbf{u^\prime})|d\mathbf{u^\prime}\right)ds.
\]
If $\scriptF$ are ellipsoids and $s^1=s^2=\cdots=s^9$, then the integrand is $\sim|\scriptF|^{10}|P_1^\perp(\scriptF)|^5$, where $P_1^\perp(x)=x_1$. Finner's inequality (see Corollary \ref{C:Lj Pj perp finner}) and lower bounds of $\scriptF$ imply the integrand is $\gtrsim\alpha_4^{10}\alpha_5^{15}$. For general $\scriptF$, we approximate them by ellipsoids. The new obstruction is that $s^1,\dots,s^9$ are not the same in general. If we further assume ``regularity" -- i.e., we substitute $\sim$ for $\gtrsim$ in \eqref{E:outline beta_1 bound}-\eqref{E:outline beta_2 bound}, then this can be resolved by repeating applications of H\"older's inequality (see Lemma \ref{L:am-gm}).

To pass from \eqref{E:outline avg medium ex aniso} to \eqref{E:outline rwt aniso}, it remains to show
\begin{equation}\label{E:ex avg gen finner}
    \frac{|S|}{|P_1(S)|}|S|^2\gtrsim\displaystyle\prod_{j=1}^3\frac{|S|}{|L_j(S)|},
\end{equation}
We observe that,  \eqref{E:ex avg gen finner} can be deduced from Finner's inequality
\begin{equation}\label{E:outline pseudo finner}
    |P_1(S)|\lesssim\displaystyle\prod_{j=1}^3|L_j(S)|,
\end{equation}
and regularity is not needed in this case. Nevertheless, counterparts of \eqref{E:outline pseudo finner} in other cases, say, in Example \ref{Ex: aniso 2}, cannot be easily deduced from Finner's inequality. A generalized Finner's inequality (see Lemma \ref{L:gen finner})
 is required. This generalization is made possible thanks to the nested kernels of $\widetilde\pi_j$ and regularity.

\section{Reduce Theorem \ref{T:main} to a base case}\label{S:main to weak}
In this section, we will prove by induction Theorem \ref{T:main} assuming Theorem \ref{E:base case}. Note that the conditions $p_j\in[1,\infty]$, \eqref{A}, \eqref{B}, and \eqref{C} characterize a compact convex polytope
\[
\scriptP\coloneqq\{\mathbf{p^{-1}}\in[0,1]^{M}:(p_1,\dots,p_M)\text{ satisfies conditions \eqref{A}, \eqref{B}, and \eqref{C}}\}.
\]
For $\mathbf p=(p_1,\dots,p_M)$, write $\mathbf{p^{-1}}=(p_1^{-1},\dots,p_M^{-1})$.
\begin{definition}
    $\mathbf{p^{-1}}\in\scriptP$ is said to be \textit{extreme} if $\{\mathbf{p^{-1}}+r\mathbf v:|r|<1\}\not\subseteq\scriptP$ for any $\mathbf v\in R^M\setminus\{0\}$.
\end{definition}

\begin{definition}
A coordinate subspace $V\leq\R^n$ is \textit{critical} if equality holds in \eqref{B}.
\end{definition}

\begin{lemma} \label{L:0 critical}
If $\{0\}$ is critical, then \eqref{E:strong} holds true.
\end{lemma}

\begin{proof}
$\{0\}$ being critical implies $\sum_{j=1}^{M} p_j^{-1} = 1$.
Plugging this and \ref{C} into \eqref{A} yields
$$\displaystyle\sum_{j=1}^{M}\frac{n}{p_j}=\displaystyle\sum_{j=1}^m\frac{n_j}{p_j}+\displaystyle\sum_{j=m+1}^M\frac{n}{p_j}=\displaystyle\sum_{j=1}^m\frac{n}{p_j}+\displaystyle\sum_{j=m+1}^M\frac{n_j}{p_j}.$$
Thus, for each $1\leq j\leq M$, we have $p_j=\infty$ or $n_j=n$.  In the former case, $\scriptM$ can be reduced to an $(M-1)$-linear operator. If $p_j<\infty$ for all $j$, then we have $L_j = \mathbb \I_n$, so \eqref{E:strong} becomes H\"older's inequality.
\end{proof}

Thus, we may assume that $\{0\}$ is not a critical subspace.
\begin{lemma} \label{L:base case}
Suppose $\scriptP\neq\emptyset$, and suppose there is no proper critical subspace $W<\R^n$, then either
\begin{itemize}
    \item $\scriptP=\{\mathbf{p^{-1}}\}$ is a singleton with $\mathbf p$ satisfying the assumptions of Theorem \ref{T:base case}; or
    \item for every extreme point $\mathbf p^{-1}$ of $\scriptP$, $p_j=\infty$ for some $1\leq j\leq M$.
\end{itemize}
\end{lemma}

\begin{proof}
Suppose there are no proper critical subspaces, and that $\scriptP=\{\mathbf p^{-1}\}$ is a singleton. Then $\mathbf p\in\Q^M$. Furthermore, by definition of critical subspaces, $\mathbf p$ satisfies the assumptions in Theorem \ref{T:base case}.

Suppose $\scriptP\neq\emptyset$ is not a singleton. We observe that $\scriptP$ is contained in the interior of the set
\[
\{\mathbf{p^{-1}}\in[0,1]^{M}:\mathbf p\text{ satisfies \eqref{B} for all coordinate subspaces }V<\R^n\}.
\]
By continuity and the fact that $\scriptP$ is closed, we have
\[
\scriptP=\{\mathbf{p^{-1}}\in[0,1]^{M}:\text{$\mathbf p$ satisfies \eqref{A} and \eqref{C}}\},
\]
which is the intersection of finitely many affine subspaces of $\R^M$ and the unit cube $[0,1]^M$. Since $\scriptP$ is not a singleton, the solutions to \eqref{A} and \eqref{C} are not unique. Then any extreme point $\mathbf{p^{-1}}$ must lie on the boundary of the unit cube $[0,1]^{M}$; that is, there exists some $1\leq j\leq M$ such that $p_j^{-1}\in\{0,1\}$. Suppose $p_j=1$ for some $1\leq j\leq M$. Then we can deduce from \eqref{A} and \eqref{C} that $p_i=\infty$ for all $i\neq j$.
\end{proof}

Under the assumption of Theorem \ref{T:base case}, Lemma \ref{L:base case} implies that \eqref{E:strong} holds true when there is no proper critical subspace. To close the induction, it remains to consider the case where a nonzero proper critical subspace exists. Let's recall Euclidean Brascamp--Lieb inequalities. For each $j$, we define the linear mapping $\pi_j^\flat$ on $\R^{2n}$ by
$$
\pi_j^\flat(x,y)\coloneqq
\begin{cases}
(L_j x,y), & 1 \leq j \leq m; \\
(x,L_jy), & m<j\leq M.
\end{cases} 
$$

The associated multilinear form is 
$$
\scriptM^\flat(f_1,\ldots,f_{M}) = 
\iint_{\R^{2n}} \prod_{j=1}^{M} f_j \circ \pi_j^\flat(x,y)\, dx\, dy.
$$
In this special Finner case, the Brascamp--Lieb inequality can be rewritten as follows:

\begin{proposition}[\cite{finner1992generalization,BCCTlong,BCCTshort}] \label{P:linear BL}
    Suppose $p_j\in[1,\infty]$ and 
\begin{gather}
   \tag{$A^\flat$}\label{Aflat} 2n=\sum_{j=1}^{M}\frac{n+n_j}{p_j};\\
   \tag{$B^\flat$}\label{Bflat} \dim V \leq \sum_{j=1}^{m} \frac{\dim L_jV_1+\dim V_2}{p_j}+\sum_{j=m+1}^M\frac{\dim V_1+\dim L_jV_2}{p_j},
\end{gather}
where $V=V_1\times V_2\leq\R^n\times\R^n$ for arbitrary coordinate subspaces $V_1,V_2\leq\R^n$. Then
\begin{equation} \label{E:linear BL ineq}
    \scriptM^\flat(f_1,\ldots,f_{M}) \leq \prod_{j=1}^{M}\|f_j\|_{p_j}
    \end{equation}
holds uniformly over nonnegative continuous functions $f_j$ of compact suport.
\end{proposition} 

\begin{lemma}\label{L:induction on n}
If $W < \R^n$ is a nonzero proper critical subspace, and Theorem \ref{T:main} holds with $\R^n$ replaced by $W$, then \eqref{E:strong} holds for $L_j:\R^n\longrightarrow V_j$.
\end{lemma}

\begin{proof}
By definition, $W$ is a coordinate subspace, for any coordinate projection $L\longrightarrow V$, the restrictions
$L^\prime=L|_W:W\longrightarrow V^\prime, L^{\prime\prime}=L|_{W^\perp}:W^\perp\longrightarrow V^{\prime\prime}$
of $L$ to $W,W^\perp$, respectively, are coordinate projections to $V^\prime=W\cap V, V^{\prime\prime}=W^\perp\cap V$, respectively. We observe that $V=V^\prime\oplus V^{\prime\prime}$.

For $x\in\R^n$, write $x=(x^\prime,x^{\prime\prime})\in W\times W^\perp$. Then one has $Lx=(L^\prime x^\prime,L^{\prime\prime}x^{\prime\prime})$. For $1\leq j\leq M$, we write $x_j^\prime=L_j^\prime x^\prime,x_j^{\prime\prime}=L_j^{\prime\prime}x^{\prime\prime}$. Furthermore, a coordinate subspace of $W$ is a coordinate subspace of $\R^n$. Hence, the given data $\mathbf p$ satisfy \eqref{A}, \eqref{B}, and \eqref{C} with $R^n$ replaced by $W$. Then $\scriptM(f_1,\dots,f_{M})$ can be expressed as
\[
\begin{array}{cl}
    \int_{W^\perp\times W^\perp}\int_{W\times W\times\R} & \displaystyle\prod_{j=1}^{m}f_j(x_j^\prime,x_j^{\prime\prime},y,t+\frac{1}{2}x\cdot y) \\
    & \displaystyle\prod_{j=m+1}^Mf_j(x,y_j^\prime,y_j^{\prime\prime},t-\frac{1}{2}x\cdot y)dx^\prime dy^\prime dtdx^{\prime\prime}dy^{\prime\prime}
\end{array}
\]

Define auxiliary functions
\[
\begin{array}{rcll}
    f_j^{x^{\prime\prime},y^{\prime\prime}}(x_j^\prime,y^\prime,t) & \coloneqq & f_j(x_j^\prime,x_j^{\prime\prime},y,t+\frac{1}{2}x^{\prime\prime}\cdot y^{\prime\prime}), & 1\leq j\leq m, \\
    f_j^{x^{\prime\prime},y^{\prime\prime}}(x^\prime,y_j^\prime,t) & \coloneqq & f_j(x,y_j^\prime,y_j^{\prime\prime},t-\frac{1}{2}x^{\prime\prime}\cdot y^{\prime\prime}), & m<j\leq M.
\end{array}
\]
By assumption, one has
\[
\scriptM(f_1,\dots,f_{M})\lesssim\int_{W^\perp\times W^\perp}\displaystyle\prod_{j=1}^{M}\|f_j^{x^{\prime\prime},y^{\prime\prime}}\|_{p_j}dx^{\prime\prime}dy^{\prime\prime}
\]
Since $x^{\prime\prime},y^{\prime\prime}$ are constants in the $L^{p_j}$ norms, we have for $1\leq j\leq m$
\begin{align*}
    \|f_j^{x^{\prime\prime},y^{\prime\prime}}\|_{p_j} = & \left(\int_{V_j^\prime\times W\times\R}|f_j(x_j^\prime,x_j^{\prime\prime},y,t+\frac{1}{2}x^{\prime\prime}\cdot y^{\prime\prime})|^{p_j}dx_j^\prime dy^\prime dt\right)^\frac{1}{p_j} \\[5pt]
    =  & \left(\int_{V_j^\prime\times W\times\R}|f_j(x_j^\prime,x_j^{\prime\prime},y,t)|^{p_j}dx_j^\prime dy^\prime dt\right)^\frac{1}{p_j} \eqqcolon \widetilde f_j(x_j^{\prime\prime},y^{\prime\prime}).
\end{align*}
We observe that $\widetilde f_j$ is a function on $L_j^{\prime\prime}W^\perp\times W^\perp$. We can similarly rewrite $\|f_j^{x^{\prime\prime},y^{\prime\prime}}\|_{p_j}$ as a function $\widetilde f_j$ on $W^\perp\times L_j^{\prime\prime}W^\perp$ for $m<j\leq M$.

We apply Proposition \ref{P:linear BL} to linear coordinate projections
\[
\pi_j^\flat=
\begin{cases}
L_j^{\prime\prime}\times\Id_{W^\perp}, & 1\leq j\leq m, \\
\Id_{W^\perp}\times L_j^{\prime\prime}, & m<j\leq M,
\end{cases}
\]
on $W^\perp\times W^\perp$. Since $W$ is critical, subtracting \eqref{B} with $W$ from \eqref{A} gives \eqref{Aflat}
\[
2\dim W^\perp=\displaystyle\sum_{j=1}^{M}\frac{\dim L_j^{\prime\prime}W^\perp+\dim W^\perp}{p_j}
\]

Let $V_1,V_2$ be arbitrary coordinate subspaces of $W^\perp$. Then $V_1\oplus W\leq \R^n$ is coordinate subspace satisfying \eqref{B}:
\[
\dim V_1+\dim W+1\leq\displaystyle\sum_{j=1}^{m}\frac{\dim L_j V_1+\dim L_jW+1}{p_j}+\displaystyle\sum_{j=m+1}^M\frac{\dim V_1+\dim W+1}{p_j}.
\]
On the other hand, \eqref{B} with $W$ gives
\[
\dim W+1=\displaystyle\sum_{j=1}^{m}\frac{\dim L_jW+1}{p_j}+\displaystyle\sum_{j=m+1}^M\frac{\dim W+1}{p_j}.
\]
Taking the difference yields
\begin{equation}\label{E:B flat 1}
    \dim V_1\leq\displaystyle\sum_{j=1}^{m}\frac{\dim L_jV_1}{p_j}+\displaystyle\sum_{j=m+1}^M\frac{\dim V_1}{p_j}.
\end{equation}
Similarly, one has
\begin{equation}\label{E:B flat 2}
    \dim V_2\leq\displaystyle\sum_{j=1}^{m}\frac{\dim V_2}{p_j}+\displaystyle\sum_{j=m+1}^M\frac{\dim L_jV_2}{p_j}.
\end{equation}
Summing \eqref{E:B flat 1} and \eqref{E:B flat 2} gives \eqref{Bflat}. Then by Proposition \ref{P:linear BL}, we conclude that
\begin{align*}
    \scriptM(f_1,\dots,f_{M}) \lesssim & \int_{W^\perp\times W^\perp}\displaystyle\prod_{j=1}^{M}\|f_j^{x^{\prime\prime},y^{\prime\prime}}\|_{p_j}dx^{\prime\prime}dy^{\prime\prime} \\[5pt]
    = & \int_{W^\perp\times W^\perp}\displaystyle\prod_{j=1}^{M}\widetilde f_j\circ\pi_j^\flat(x^{\prime\prime},y^{\prime\prime})dx^{\prime\prime}dy^{\prime\prime} \\[5pt]
   \lesssim & \displaystyle\prod_{j=1}^{M}\|\widetilde f_j\|_{p_j} = \displaystyle\prod_{j=1}^{M} \|f_j\|_{p_j}.
\end{align*}
\end{proof}

\section{Auxiliary maps with nested kernels}\label{S:more notations}
The rest of the article is dedicated to proving Theorem \ref{T:base case}. In this section, we will introduce auxiliary maps with nested kernels. Since the arguments can get very technical, it is good to keep Example \ref{Ex:aniso} in mind while reading.

Recall the notations in Theorems \ref{T:main} and \ref{T:base case}. Without loss of generality, we may assume $p_j<\infty$. Since $p_j\in\Q$, we may write $p_j^{-1}=q_j/q$ for some positive integers $q_j,q$ with $\gcd(q_1,\dots,q_M,q)=1$. Let $N=(\sum_{j=1}^Mq_j)-q$. \eqref{A} implies $N=(\sum_{j=1}^mk_jq_j)/(n+1)\geq0$, where $k_j=\dim K_j$.

If $N=0$, then the $\pi_j$ are diffeomorphisms with $|\det D\pi_j|\equiv1$. Then \eqref{E:base case} boils down to H\"older's inequality. Hence, we may assume $N>0$, i.e.,
\begin{equation}\label{E:Lp improving}
    \displaystyle\sum_{j=1}^M\frac{1}{p_j}>1.
\end{equation}

We observe that the corresponding restricted weak-type inequality is
\begin{equation}\label{E:rwt}
    |\Omega|^N\gtrsim\displaystyle\prod_{j=1}^M\alpha_j^{q_j},\qquad\alpha_j=\frac{|\Omega|}{|\pi_j(\Omega)|}.
\end{equation}
and the conditions \eqref{A}, \eqref{B twiddle}, and \eqref{C} become
\begin{gather}
    \tag{$A^\prime$}\label{A'}\displaystyle\sum_{j=1}^mk_jq_j=\displaystyle\sum_{j=m+1}^Mk_jq_j=(n+1)N; \\[15pt]
    \tag{$B^\prime$}\label{B'}N\dim V<\displaystyle\sum_{i:e_i\in V}\displaystyle\sum_{\substack{1\leq j\leq m\\e_i\in K_j}}q_j=\displaystyle\sum_{i:e_i\in V}\displaystyle\sum_{\substack{m<j\leq M\\e_i\in K_j}}q_j<N(\dim V+1); \\
    \tag{$C^\prime$}\label{C'}\displaystyle\sum_{j=1}^mq_j\dim(K_j\cap V)=\displaystyle\sum_{j=m+1}^Mq_j\dim(K_j\cap V),
\end{gather}
for arbitrary coordinate subspaces $V<\R^n$, where $e_i$ is the $i$-th coordinate. Setting
\begin{equation}\label{E:Q_i}
    Q_i\coloneqq\displaystyle\sum_{\substack{1\leq j\leq m\\e_i\in K_j}}q_j=\displaystyle\sum_{\substack{m<j\leq M\\e_i\in K_j}}q_j,
\end{equation}
we observe that \eqref{B'} implies
\begin{equation}\label{B''}
    \forall i=1,\dots,n\qquad N<Q_i<2N.
\end{equation}

Let $\widetilde m$ be the number of distinct values in $\{Q_i\}_{i=1}^n$. Without loss of generality, we may assume
\begin{equation}\label{E:order of Q_i}
    Q_1=\dots=Q_{\widetilde k_1}>Q_{\widetilde k_1+1}=\dots=Q_{\widetilde k_2}>\dots>Q_{\widetilde k_{\widetilde m-1}+1}=\dots=Q_{\widetilde k_{\widetilde m}},
\end{equation}
where $1\leq\widetilde k_1<\widetilde k_2<\dots<\widetilde k_{\widetilde m}=n$. The $\widetilde k_j$ will be the coranks of our auxiliary maps.

The following two vertical projections $\pi,\pi_*:\HH^n\longrightarrow\R^{n+1}$ are key to our proof:
\begin{equation}\label{E:pi,pi*}
    \pi(x,y,t)=(y,t+\frac{1}{2}x\cdot y),\qquad\pi_*(x,y,t)=(x,t-\frac{1}{2}x\cdot y).
\end{equation}
Define $P_j:\R^n\longrightarrow\R^{n-\widetilde k_j}$ by $P_j(x)=(x_{\widetilde k_j+1},\dots,x_n)$ for $j=1,\dots,\widetilde m$, and define the auxiliary maps
\begin{equation}\label{E:pi j twiddle}
    \begin{array}{rcl}
        \widetilde\pi_j(x,y,t) & = & (P_j(x),y,t+\frac{1}{2}x\cdot y), \\
        \widetilde\pi_j^*(x,y,t) & = & (x,P_j(y),t-\frac{1}{2}x\cdot y).
    \end{array}
\end{equation}

Consider the following left-invariant vector fields of $\HH^n$:
\begin{align*}
    \forall1\leq i\leq n\qquad & X_i(x,y,t)\coloneqq\frac{\partial}{\partial x_i}-\frac{1}{2}y_i\frac{\partial}{\partial t}, \\[5pt]
    & Y_i(x,y,t)\coloneqq\frac{\partial}{\partial y_i}+\frac{1}{2}x_i\frac{\partial}{\partial t}.
\end{align*}
Write $\mathbf X=(X_1,\dots,X_n),\mathbf Y=(Y_1,\dots,Y_n)$. We observe that the only nontrivial Lie bracket of these left-invariant vector fields is
\begin{equation}
    T\coloneqq\frac{\partial}{\partial t}=[X_i,Y_i],\qquad1\leq i\leq n,
\end{equation}
and $X_1,\dots,X_n,Y_1,\dots,Y_n,T$ form a global frame for $\HH^n$.

Note that the fibers of $\widetilde\pi_j,\widetilde\pi_j^*$ have frames $\{X_1,\dots,X_{\widetilde k_j}\},\{Y_1,\dots,Y_{\widetilde k_j}\}$, respectively, while those of $\pi_j$ have frames $\{X_i:e_i\in K_j\}$ for $1\leq j\leq m$ and $\{Y_i:e_i\in K_j\}$ for $m<j\leq M$. Then $\ker d\widetilde\pi_j<\ker d\widetilde\pi_{j+1}$. Since $\widetilde k_{\widetilde m}=n$, we have $\ker D\pi_j\leq\ker d\widetilde\pi_{\widetilde m}=\ker D\pi$ for each $1\leq j\leq m$. One has similar observations about $\widetilde\pi_*,\widetilde\pi_j^*$ and $\pi_j,m<j\leq M$.

Let
\begin{equation}\label{E:qj twiddle}
    \widetilde q_j=
    \begin{cases}
        Q_{\widetilde k_j}-Q_{\widetilde k_{j+1}}, & 1\leq j<\widetilde m \\
        Q_{\widetilde k_{\widetilde m}} & j=\widetilde m.
    \end{cases}
\end{equation}
Note that $\widetilde q_j>0$ for each $1\leq j\leq\widetilde m$. Furthermore, the exponents $\{q^\prime/\widetilde q_j\}_{j=1}^{\widetilde m}\cup\{q^\prime/q_j\}_{j=m+1}^M$ with $q^\prime=Q_1+\sum_{j=m+1}^Mq_j-N$ (\eqref{B''} implies $q^\prime>0$) satisfy the assumptions in Theorem \ref{T:base case} for $\{\widetilde\pi_j\}_{j=1}^{\widetilde m}\cup\{\pi_j\}_{j=m+1}^M$. This corresponds to the restricted weak-type inequality
\begin{equation}\label{E:intermediate}
    |\Omega|^N\gtrsim\displaystyle\prod_{j=1}^{\widetilde m}\beta_j^{\widetilde q_j}\displaystyle\prod_{j=m+1}^M\alpha_j^{q_j},\qquad\beta_j=\frac{|\Omega|}{|\widetilde\pi_j(\Omega)|},
\end{equation}
which will be proven later as an intermediate step toward \eqref{E:rwt}. The main advantage of working on $\widetilde\pi_j,\widetilde\pi_j^*$ is the strictly linear ordering among their kernels $\ker P_1<\ker P_2<\cdots<\ker P_{\widetilde m}=\R^n$, which is in general not enjoyed by the $K_j=\ker L_j$ associated with $\pi_j$. This ordering turns out to be powerful, allowing a generalization of Finner's inequality and the construction of an almost diffeomorphic inflation map.

\begin{remark}
    Recalling the weights of vector fields $X_i,Y_i$ from Definition \ref{D:weight}, we observe that $w(X_i)=w(Y_i)=\frac{Q_i}{q}$ for each $i=1,\dots,n$. In Examples \ref{Ex:direct sum} and \ref{Ex:isoperimetric}, $\widetilde m=1$, while in Examples \ref{Ex:aniso} and \ref{Ex: aniso 2}, $\widetilde m=2$. Moreover, $\widetilde\pi_j,\widetilde\pi_j^*$ defined in Subsection \ref{S:proof ex aniso} are consistent with \eqref{E:pi j twiddle}. In Subsection \ref{S:proof ex aniso}, we have $\widetilde q_1=5,\widetilde q_2=10$.
\end{remark}

\section{Finner's inequality and generalization thereof}\label{S:finner}
In this section, we will present a generalized Finner's inequality. Here's a bird's-eye view of its vital role. To pass from \eqref{E:intermediate} to \eqref{E:rwt}, we need
\begin{equation}\label{E:general finner 1}
    \displaystyle\prod_{j=1}^{\widetilde m}\beta_j^{\widetilde q_j}\overset{\textcolor{red}{?}}{\gtrsim}\displaystyle\prod_{j=1}^m\alpha_j^{q_j},
\end{equation}
which is not true in general (see Example \ref{Ex:beta>alpha}). Now let's look at the fibers within $\pi^{-1}(z)$ (see \eqref{E:pi_j fibers in pi}). Consider $\omega=\pi^{-1}(z)\cap\Omega\neq\emptyset$. Since $\{\pi_j|_\omega\}_{j=1}^m,\{\widetilde\pi_j|_\omega\}_{j=1}^{\widetilde m}$ are essentially $\{L_j\}_{j=1}^m,\{P_j\}_{j=1}^{\widetilde m}$, respectively (see \eqref{E:pi_j restriction are affine L_j}), proving \eqref{E:general finner 1} boils down to establishing
\begin{equation}\label{E:general finner 1+}
    \displaystyle\prod_{j=1}^{\widetilde m}(\frac{|\omega|}{|P_j(\omega)|})^{\widetilde q_j}\overset{\textcolor{red}{?}}{\gtrsim}\displaystyle\prod_{j=1}^m\left(\frac{|\omega|}{|L_j(\omega)|}\right)^{q_j},\qquad\omega\in\R^n.
\end{equation}
We observe that \eqref{E:general finner 1+} is equivalent to
\begin{equation}\label{E:general finner 2}
    |\omega|\overset{\textcolor{red}{?}}{\lesssim}\displaystyle\prod_{j=1}^m|L_j(\omega)|^\frac{q_j}{N^\prime}\displaystyle\prod_{j=1}^{\widetilde m}\left(\frac{|\omega|}{|P_j(\omega)|}\right)^\frac{\widetilde q_j}{N^\prime},\qquad N^\prime=\displaystyle\sum_{j=1}^mq_j.
\end{equation}
We will prove momentarily, in Corollary \ref{C:Lj Pj perp finner}, that \eqref{E:general finner 2} is ``nearly" a consequence of Finner's inequality, but not quite because we have $|\omega|/|P_j(\omega)|$ in place of $|P^\perp(\omega)|$. First, let's recall the restricted weak-type Finner's inequality.

\begin{theorem}[Theorem 2.1 of \cite{finner1992generalization}]\label{T:finner weak}
    Let $n,d$ be positive integers, and let $\{e_i\}_{i=1}^n$ be the standard basis of $\R^n$. For $1\leq j\leq d$, let $l_j:\R^n\longrightarrow W_j$ be coordinate projections. If $p_j\in[1,\infty]$ and
    \begin{equation}\label{E:weak finner assumption}
        \forall1\leq i\leq n \quad\displaystyle\sum_{j:e_i\in W_j}\frac{1}{p_j}=1,
    \end{equation}
    then for any measurable set $\omega\subset\R^n$,
    \begin{equation}\label{E:finner}
        |\omega|\leq\displaystyle\prod_{j=1}^d|l_j(\omega)|^\frac{1}{p_j}.
    \end{equation}
\end{theorem}

\begin{corollary}\label{C:Lj Pj perp finner}
    Recall the linear maps $\{L_j\}_{j=1}^M,\{P_j\}_{j=1}^{\widetilde m}$ from \eqref{E:pi_j} and \eqref{E:pi j twiddle}. Let $P_j^\perp$ denote the orthogonal projections from $\R^n$ onto the first $\widetilde k_j$ coordinates. Let $p_j=q/q_j$ be as in Section \ref{S:more notations}, and let $\widetilde q_j$ be as in \eqref{E:qj twiddle}. Then for any measurable subset $\omega\subset\R^n$,
    \begin{equation}\label{E:Lj Pj perp finner}
        |\omega|\leq\left(\displaystyle\prod_{j=1}^{\widetilde m}|P_j^\perp(\omega)|^\frac{\widetilde q_j}{N^\prime}\right)\left(\displaystyle\prod_{j=1}^m|L_j(\omega)|^\frac{q_j}{N^\prime}\right),\qquad N^\prime=\sum_{j=1}^mq_j.
    \end{equation}
    
    We have a similar result if $\{(L_j,q_j)\}_{j=1}^m$ is replaced by $\{(L_j,q_j)\}_{j=m+1}^M$.
\end{corollary}

\begin{proof}
    Note that for $1\leq i\leq n$, we have $\sum_{j=1}^mq_j=\sum_{j:e_i\in V_j}q_j+\sum_{j:e_i\in K_j}q_j$ (recall the $V_j,K_j$ from Section \ref{S:intro}). Then from \eqref{E:Q_i}, \eqref{E:order of Q_i}, and \eqref{E:qj twiddle} we deduce the following equivalent form of \eqref{E:weak finner assumption}:
    \begin{equation}\label{E:char Lj Pj perp finner}
        \forall1\leq i\leq n\qquad\displaystyle\sum_{\substack{j:e_i\in V_j\\1\leq j\leq m}}q_j+\displaystyle\sum_{\substack{j:\widetilde k_j\geq i\\1\leq j\leq\widetilde m}}\widetilde q_j=\displaystyle\sum_{j=1}^mq_j.
    \end{equation}
    The arguments for $\{(L_j,q_j)\}_{j=m+1}^M$ are almost the same.
\end{proof}

Now we see that \eqref{E:general finner 2} looks like \eqref{E:Lj Pj perp finner} with an implicit constant and $|P_j^\perp(\omega)|$ replaced by $|\omega|/|P_j(\omega)|$. In general, we may ask the following question: Does \eqref{E:finner} remain true if some of the $|l_j(\omega)|$ are replaced by $|\omega|/|l_j^\perp(\omega)|$, where $l_j^\perp$ is the orthogonal projection onto $\ker l_j$. The main obstruction is that $|\omega|/|l_j^\perp(\omega)|$, the average size of $(l_j^\perp)^{-1}(x)\cap\omega$, can be arbitrarily small compared to $l_j(\omega)$, the projection of all such fibers (see Example \ref{Ex:fibers<<projections}). The following lemma provides a positive answer to the question when the substituted $l_j$ have nested kernels and satisfy a certain ``regularity" condition. The proof is adapted from \cite{loomis1949inequality,finner1992generalization}.

\begin{lemma}\label{L:gen finner}
    Let $n,k,d$ be positive integers with $k\leq d$, let $\varepsilon\in(0,1)$, and let $\{e_i\}_{i=1}^n$ be the standard basis of $\R^n$. For $1\leq j\leq d$, let $l_j$ be coordinate projections on $\R^n$ satisfying
    \begin{equation}\label{E:nested kernels}
        \ker l_{\sigma(1)}\leq\ker l_{\sigma(2)}\leq\cdots\leq\ker l_{\sigma(k)}
    \end{equation}
    for some permutation $\sigma$ of $\{1,2,\dots,k\}$. Suppose $p_j\in[1,\infty]$ for $1\leq j\leq d$, and
    \begin{equation}\label{E:gen finner assumption}
        \forall1\leq i\leq n \quad\displaystyle\sum_{\substack{1\leq j\leq k\\e_i\in\ker l_j}}\frac{1}{p_j}+\displaystyle\sum_{\substack{k<j\leq d\\e_i\not\in\ker l_j}}\frac{1}{p_j}=1.
    \end{equation}
    Then there exists a constant $C>1$ depending only on $n,p_j$ such that for any bounded $\omega\subset\R^n$,
    \begin{gather}
        \label{E:regular}\forall1\leq j\leq k\quad\forall x\in l_j(\omega)\qquad|l_j^{-1}(x)\cap\omega|\lesssim\varepsilon^{-1}\frac{|\omega|}{|l_j(\omega)|}, \\
        \notag\Downarrow \\
        \label{E:gen finner}|\omega|\lesssim\varepsilon^{-C}\displaystyle\prod_{j=1}^k\left(\frac{|\omega|}{|l_j(\omega)|}\right)^\frac{1}{p_j}\displaystyle\prod_{j=k+1}^d|l_j(\omega)|^\frac{1}{p_j}.
    \end{gather}
\end{lemma}

\begin{remark}
    The assumptions \eqref{E:nested kernels} and \eqref{E:regular} cannot be omitted. (See Examples \ref{Ex:counter ex nest gen finner} and \ref{Ex:counter ex reg gen finner}.)
\end{remark}

\begin{proof}
    We will prove \eqref{E:gen finner} by induction on $n$. If $n=1$, then $l_j$ is either the zero map or the identity map. If $l_j=0$, then $|l_j(\omega)|_0=\#\{0\}=1$; if $l_j$ is the identity map, then $l_j(\omega)=\omega$. It is straightforward to verify that \eqref{E:gen finner assumption} implies \eqref{E:gen finner}.
    
    Now consider $n>1$. Suppose \eqref{E:gen finner} is true for $\omega\in\R^{n-1}$. Without loss of generality, we may assume
    \begin{equation}\label{E:nested kernels n-1}
        \ker l_1\leq\ker l_2\leq\cdots\leq\ker l_k\leq\langle e_1,\dots,e_{n-1}\rangle.
    \end{equation}    
    Let $P$ denote the orthogonal projection onto the last coordinate. Note that for each $x_n\in\R$, the restriction of $l_j$ to $P^{-1}(x_n)$ are (affine) coordinate projections on $\R^{n-1}$ satisfying \eqref{E:nested kernels n-1} and \eqref{E:gen finner assumption} (for $1\leq i\leq n-1$). $P^{-1}(x_n)\cap\omega$ need not satisfy \eqref{E:regular} in general, but this can be resolved by the following refinement of $\omega$.
    
    Write $\beta_j\coloneqq|\omega|/|l_j(\omega)|$ for $1\leq j\leq k$. Let $\omega_0\coloneqq\omega$, and let
    \begin{equation}\label{E:refine omega}
        \omega_j\coloneqq\{x\in\omega_{j-1}:|l_j^{-1}(l_j(x))\cap\omega_{j-1}|\geq c\beta_j\},\qquad1\leq j\leq k
    \end{equation}
    for some sufficiently small constant $c>0$. Then
    \begin{equation}
        |\omega_0\setminus\omega_1|=\int_{l_1(\omega_0\setminus\omega_1)}|l_1^{-1}(x_j)|dx_j<c\beta_1|l_1(\omega)|=c|\omega|.
    \end{equation}
    Then, $|\omega_1|\geq(1-c)|\omega|$ and $|\omega_1|/|l_j(\omega_1)|\gtrsim\beta_j$ for $1\leq j\leq k$. Repeating the argument for $\omega_j,2\leq j\leq k$, we obtain $|\omega_k|\sim|\omega|$. Furthermore, \eqref{E:nested kernels n-1} implies that for every $x\in\omega$ and $1\leq j\leq k$, either $l_j^{-1}(l_j(x))\subset\omega_k$ or $l_{j^\prime}^{-1}(l_{j^\prime}(x))\cap\omega_k=\emptyset$ holds true. Therefore, we have $|l_j^{-1}(x_j)\cap\omega_k|\geq c\beta_j$ for all $1\leq j\leq k$.
    
    Note that for every $x_n\in P(\omega)$ and $1\leq j\leq k$,
    \begin{align*}
        & \frac{|P^{-1}(x_n)\cap\omega_k|}{|l_j(P^{-1}(x_n)\cap\omega_k)|} \\[5pt]
        = & \frac{1}{|l_j(P^{-1}(x_n)\cap\omega_k)|}\int_{l_j(P^{-1}(x_n)\cap\omega_k)}|l_j^{-1}(x_j)\cap(P^{-1}(x_n)\cap\omega_k)|dx_j,
    \end{align*}
    and for every $x_j\in l_j(P^{-1}(x_n)\cap\omega_k)$,
    \begin{equation}
        |l_j^{-1}(x_j)\cap\omega_k|=|l_j^{-1}(x_j)\cap(P^{-1}(x_n)\cap\omega_k)|\leq|l_j^{-1}(x_j)\cap\omega|\lesssim\varepsilon^{-1}\beta_j,
    \end{equation}
    where the equality is due to \eqref{E:nested kernels n-1} and the last inequality is due to \eqref{E:regular}. Therefore,
    \begin{equation}
        c\beta_j\leq\frac{|P^{-1}(x_n)\cap\omega_k|}{|l_j(P^{-1}(x_n)\cap\omega_k)|}\lesssim\varepsilon^{-1}\beta_j.
    \end{equation}
    
    By the induction hypothesis, for every $x_n\in P(\omega_k)$,
    \begin{align*}
        |P^{-1}(x_n)\cap\omega_k|\lesssim & \varepsilon^{-C}\displaystyle\prod_{j=1}^k\left(\frac{|P^{-1}(x_n)\cap\omega_k|}{|l_j(P^{-1}(x_n)\cap\omega_k)|}\right)^\frac{1}{p_j}\displaystyle\prod_{j=k+1}^d|l_j(P^{-1}(x_n)\cap\omega_k)|^\frac{1}{p_j} \\[5pt]
        \lesssim & \varepsilon^{-C}\left(\displaystyle\prod_{j=1}^k\beta_j^\frac{1}{p_j}\right)\displaystyle\prod_{j=k+1}^d|l_j(P^{-1}(x_n)\cap\omega_k)|^\frac{1}{p_j}.
    \end{align*}
    If $e_n\in\ker l_j$, then $|l_j(P^{-1}(x_n)\cap\omega_k)|\leq|l_j(P^{-1}(x_n)\cap\omega)|\leq|l_j(\omega)|$. Thus,
    \begin{align}
        \label{E:step 1}& |\omega|\sim|\omega_k| = \int_{P(\omega_k)}|P^{-1}(x_n)\cap\omega_k|dx_n \\[5pt]
        \notag\lesssim & \varepsilon^{-C}\left(\int_{P(\omega)}\displaystyle\prod_{\substack{k<j\leq d\\e_n\not\in\ker l_j}}|l_j(P^{-1}(x_n)\cap\omega)|^\frac{1}{p_j}dx_n\right)\left(\displaystyle\prod_{\substack{k<j\leq d\\e_n\in\ker l_j}}|l_j(\omega)|^\frac{1}{p_j}\right)\displaystyle\prod_{j=1}^k\beta_j^\frac{1}{p_j}.
    \end{align}
    Furthermore, we deduce from \eqref{E:gen finner assumption} and \eqref{E:nested kernels n-1}
    \begin{equation}
        \displaystyle\sum_{\substack{k<j\leq d\\e_n\not\in\ker l_j}}\frac{1}{p_j}=1.
    \end{equation}
    By H\"older's inequality,
    \begin{align}
        \notag& \int_{P(\omega)}\displaystyle\prod_{\substack{k<j\leq d\\e_n\not\in\ker l_j}}|l_j(P^{-1}(x_n)\cap\omega)|^\frac{1}{p_j}dx_n \\[5pt]
        \label{E:step 2}\leq & \displaystyle\prod_{\substack{k<j\leq d\\e_n\not\in\ker l_j}}\left(\int_{P(\omega)}|l_j(P^{-1}(x_n)\cap\omega)|\right)^\frac{1}{p_n} = \displaystyle\prod_{\substack{k<j\leq d\\e_n\not\in\ker l_j}}|l_j(\omega)|^\frac{1}{p_j}.
    \end{align}
    Combining \eqref{E:step 1} and \eqref{E:step 2} yields \eqref{E:gen finner}.
\end{proof}

\begin{corollary}\label{C:specialize gen finner}
    Let $\varepsilon\in(0,1)$. Recall the linear maps $\{L_j\}_{j=1}^M,\{P_j\}_{j=1}^{\widetilde m}$ from \eqref{E:pi_j} and \eqref{E:pi j twiddle}. Let $p_j=q/q_j$ be as in Section \ref{S:more notations}, and let $\widetilde q_j$ be as in \eqref{E:qj twiddle}. Then there exists a constant $C>1$ depending only on $n,p_j$ such that for any bounded $\omega\subset\R^n$,
    \begin{gather}
        \notag\forall1\leq j\leq k\quad\forall x\in\omega\qquad|P_j^{-1}(x)\cap\omega|\lesssim\varepsilon^{-1}\frac{|\omega|}{|P_j(\omega)|}, \\\notag
        \Downarrow \\
        \label{E:specialize gen finner}|\omega|\lesssim\varepsilon^{-C}\left(\displaystyle\prod_{j=1}^m|L_j(\omega)|^\frac{q_j}{N^\prime}\right)\displaystyle\prod_{j=1}^{\widetilde m}\left(\frac{|\omega|}{|P_j(\omega)|}\right)^\frac{\widetilde q_j}{N^\prime},\qquad N^\prime=\sum_{j=1}^mq_j.
    \end{gather}
    Likewise, \eqref{E:specialize gen finner} holds true if we replace $\{(L_j,q_j)\}_{j=1}^m$ by $\{(L_j,q_j)\}_{j=m+1}^M$.
\end{corollary}

\begin{proof}
    We verify the case of $\{(L_j,q_j)\}_{j=1}^m$. The argument for $\{(L_j,q_j)\}_{j=m+1}^M$ is almost identical. By construction, we have $\ker P_1<\ker P_2<\cdots<\ker P_{\widetilde m}$. By the same argument in the proof of Corollary \ref{C:Lj Pj perp finner}, we have \eqref{E:char Lj Pj perp finner}. Thus, the conclusion follows from Lemma \ref{L:gen finner}.
\end{proof}

We finish this section by approximating convex quasiextremizers of Finner's inequality by ellipsoids.

\begin{definition}\label{D:partition}
    Let $n,d,\widetilde d$ be positive integers, and let $\{l_j\}_{j=1}^d$ be coordinate projections on $\R^n$. We say that $\{U_a\}_{a=1}^{\widetilde d}$ is a \textit{maximal partition} of $\R^n$ given by $\{l_j\}_{j=1}^d$ if
    \begin{itemize}
        \item the $U_a$ are coordinate subspaces of $\R^n$ satisfying $\R^n=\oplus_{a=1}^{\widetilde d}U_a$;

        \item for every $1\leq j\leq d$,
        \begin{equation}\label{E:partition}
            \text{either $U_a\leq\ker l_j$ or $U_a\leq l_j(\R^n)$ holds true; and}
        \end{equation}

        \item the partition is maximal in the sense that if a subspace $V\leq\R^n$ satisfies \eqref{E:partition} for every $1\leq j\leq d$, then $V\leq U_a$ for some $1\leq i\leq\widetilde d$.
    \end{itemize}

    We say that \textit{an ellipsoid $\scriptC\subset\R^n$ is adapted to $\{l_j\}_{j=1}^d$} if its principal axes are along an orthonormal basis $\cup_{a=1}^{\widetilde d}\{\widetilde e_{ab}\}_{b=1}^{\dim U_a}$, where $\{\widetilde e_{ab}\}_{b=1}^{\dim U_a}$ form a basis for $U_a$ for each $a$.
\end{definition}

\begin{lemma}\label{L:unique partition}
    Let $n,d,\widetilde d$ be positive integers, and let $\{l_j\}_{j=1}^d$ be coordinate projections on $\R^n$. The maximal partition of $\R^n$ given by $\{l_j\}_{j=1}^d$ exists and is unique up to permutation.
\end{lemma}

\begin{proof}
    Define $f:\{1,2,\dots,n\}\longrightarrow\{0,1\}^d$ by
    \begin{equation}
        (f(j))_i=
        \begin{cases}
            1, & e_i\in\ker l_j; \\
            0, & e_i\notin\ker l_j.
        \end{cases}
    \end{equation}
    Write $f(\{1,2,\dots,n\})=\{w_1,w_2,\dots,w_{\widetilde d}\}$ as $\widetilde d$ distinct elements, and define
    \begin{equation}
        U_a\coloneqq\Span\{e_i:f(i)=w_a\},\qquad1\leq a\leq\widetilde d.
    \end{equation}
    It is straightforward to verify that $\{U_a\}_{a=1}^{\widetilde d}$ is a maximal partition of $\R^n$.

    Suppose both $\{U_a\}_{a=1}^{d_1},\{V_b\}_{b=1}^{d_2}$ are maximal partitions of $\R^n$ given by $\{l_j\}_{j=1}^d$. Then the $U_a,V_b$ satisfy \eqref{E:partition}. By maximality of $\{U_a\}_{a=1}^{d_1},\{V_b\}_{b=1}^{d_2}$, we must have $U_a=V_b$ for some $1\leq b\leq d_2$. By symmetry, every $V_b$ is equal to $U_a$ for some $1\leq a\leq d_1$.
\end{proof}

\begin{lemma}\label{L:finner quasiextremal}
    Let $\varepsilon\in(0,1)$, and let $n,d$ be positive integers. For $1\leq j\leq d$, let $l_j$ be coordinate projections satisfying \eqref{E:weak finner assumption}, and let $\{U_a\}_{a=1}^{\widetilde d}$ be the maximal partition of $\R^n$ given by $\{l_j\}_{j=1}^d$. Suppose $\omega\subset\R^n$ is bounded and convex, and satisfies
    \begin{equation}\label{E:finner quasiextremal}
        |\omega|\geq\varepsilon\displaystyle\prod_{j=1}^d|l_j(\omega)|^\frac{1}{p_j}.
    \end{equation}
    Then there exists a constant $C>0$ depending only on $n$ and an ellipsoid $\scriptC$ adapted to $\{l_j\}_{j=1}^d$ such that
    \begin{equation}
        \forall A\subset\{1,2,\dots,\widetilde d\}\qquad P_A(\omega)\subset P_A(\scriptC),\qquad|P_A(\scriptC)|\lesssim\varepsilon^{-C}|P_A(\omega)|,
    \end{equation}
    where $P_A:\R^n\longrightarrow\oplus_{a\in A}U_a$ is the orthogonal projection. In particular, when $A=\{1,2,\dots,\widetilde d\}$,
    \begin{equation}
        \omega\subset\scriptC,\qquad|\scriptC|\lesssim\varepsilon^{-C}|\omega|,
    \end{equation}
    and when $A=\{a\}$ for any $1\leq a\leq\widetilde d$, we may write $P_a\coloneqq P_{\{a\}}$, and
    \begin{equation}
        P_a(\omega)\subset P_a(\scriptC),\qquad|P_a(\scriptC)|\lesssim\varepsilon^{-C}|P_a(\omega)|.
    \end{equation}

    If $\omega$ is in addition balanced, i.e., $x\in\omega$ if and only if $-x\in\omega$, then $\scriptC$ can be further chosen to be centered at the origin.
\end{lemma}

Before proving Lemma \ref{L:finner quasiextremal}, let's recall some useful properties of convex sets.

\begin{lemma}\label{L:convex large fibers}
    Let $n$ be positive integers. Let $l$ be a coordinate projection, and let $U\leq l(\R^n)$. Suppose $\omega\subset\R^n$ is bounded and convex. Then
    \begin{equation}\label{E:convex obs}
        \frac{|l(\omega)|}{|P(l(\omega))|}\gtrsim\frac{|\omega|}{|P(\omega)|},
    \end{equation}
    where $P:\R^n\longrightarrow U^\perp$ is the coordinate projection onto $U^\perp$.
\end{lemma}

\begin{proof}
    Suppose $U\leq l(\R^n)$. We first observe that every fiber of $l$ in $\omega$ is bounded above by the average size of the $l$ fibers because of convexity. To be specific,
    \begin{equation}\label{E:convex avg fiber 1}
        \forall x^\prime\in l(\omega)\qquad|l^{-1}(x^\prime)\cap\omega|\lesssim\frac{|\omega|}{|l(\omega)|}.
    \end{equation}
    Moreover, $l(\omega)$ is also convex. The same argument that leads to \eqref{E:convex avg fiber 1} yields
    \begin{equation}\label{E:convex avg fiber 2}
        \forall\widetilde x\in P(l(\omega))\qquad|\widetilde P^{-1}(\widetilde x)\cap l(\omega)|\lesssim\frac{|l(\omega)|}{|P(l(\omega))|}.
    \end{equation}
    Let $x_0\in P(\omega)$ with $|(P)^{-1}(x_0)\cap\omega|\geq|\omega|/|P(\omega)|$. (Existence of such $x_0$ is guaranteed by Fubini's theorem.) Recalling $\ker P=U\leq l(\R^n)$, we have $|P^{-1}(x_0)\cap\omega|_{\dim U}=|l(P^{-1}(x_0)\cap\omega)|_{\dim U}$. Note that $l(P^{-1}(x_0)\cap\omega)\subset\widetilde P^{-1}(l(x_0))\cap l(\omega)$ and $\ker\widetilde P=U$. Therefore, $|P^{-1}(x_0)\cap\omega|_{\dim U}\leq|\widetilde P^{-1}(l(x_0))\cap l(\omega)|_{\dim U}$. Hence, we deduce from \eqref{E:convex avg fiber 2}, $x_0\in P(\omega)$, and $P\circ l=l\circ P$
    \begin{equation}
        \frac{|l(\omega)|}{|P(l(\omega))|}\gtrsim|\widetilde P^{-1}(l(x_0))\cap l(\omega)|\geq|P^{-1}(x_0)\cap\omega|\geq\frac{|\omega|}{|P(\omega)|}.
    \end{equation}
\end{proof}

\begin{lemma}\label{L:convex}
    Let $\varepsilon\in(0,1)$, and let $n,d$ be positive integers. For $1\leq j\leq d$, let $l_j$ be coordinate projections satisfying \eqref{E:weak finner assumption}, and let $\{U_a\}_{a=1}^{\widetilde d}$ be the maximal partition of $\R^n$ given by $\{l_j\}_{j=1}^d$. Suppose $\omega\subset\R^n$ is bounded and convex, and satisfies \eqref{E:finner quasiextremal}. Let $A\subset\{1,2,\dots,\widetilde d\}$, and let $P_A$ be the coordinate projection onto $\oplus_{a\in A}U_a<\R^n$. Then $P_A(\omega)$ is convex, and there exists a constant $C$ depending only on $n$ such that
    \begin{equation}
        |P_A(\omega)|\geq C\varepsilon\displaystyle\prod_{j=1}^d|l_j\circ P_A(\omega)|^\frac{1}{p_j}.
    \end{equation}
\end{lemma}

\begin{proof}
    It suffices to prove the lemma for $A=\{2,\dots,\widetilde d\}$. Without loss of generality, we may assume $l_j\neq\I_n$. Let $\omega\subset\R^n$ be a bounded convex set satisfying \eqref{E:finner quasiextremal}. Note that $\{U_a\}_{a=2}^{\widetilde d}$ is a maximal partition of $U_1^\perp$ given by $\{l_j\circ P_A\}_{j=1}^d$. By \eqref{E:finner quasiextremal},
    \begin{equation}\label{E:convex induction}
        |\omega|\geq\varepsilon\displaystyle\prod_{j=1}^d|l_j(\omega)|^\frac{1}{p_j}=\varepsilon\left(\displaystyle\prod_{j:U_1\leq\ker l_j}|l_j(\omega)|^\frac{1}{p_j}\right)\left(\displaystyle\prod_{j:U_1\leq l_j(\R^n)}|l_j(\omega)|^\frac{1}{p_j}\right).
    \end{equation}
    Applying Lemma \ref{L:convex large fibers} with $P=P_A,U=U_1$ and $l=l_j$ for $U_1\leq l_j(\R^n)$ yields
    \begin{align*}
        |\omega| \geq & \varepsilon\left(\displaystyle\prod_{j:U_1\leq\ker l_j}|P_1^\perp(l_j(\omega))|^\frac{1}{p_j}\right)\displaystyle\prod_{j:U_1\leq l_j(\R^n)}\left(|P_1^\perp(l_j(\omega))|\frac{|l_j(\omega)|}{|P_1^\perp(l_j(\omega))|}\right)^\frac{1}{p_j} \\[5pt]
        \gtrsim & \varepsilon\left(\displaystyle\prod_{j=1}^d|l_j(P_1^\perp(\omega))|^\frac{1}{p_j}\right)\frac{|\omega|}{|P_1^\perp(\omega)|},
    \end{align*}
    where the last inequality is due to \eqref{E:weak finner assumption} and $l_j\circ P_1^\perp=P_1^\perp\circ l_j$. On the other hand, we have $|\omega|=|P_1^\perp(\omega)|(|\omega|/|P_1^\perp(\omega)|)$. Thus,
    \begin{equation}
        |P_1^\perp(\omega)|\gtrsim\varepsilon\displaystyle\prod_{j=1}^d|l_j(P_1^\perp(\omega))|^\frac{1}{p_j}.
    \end{equation}
\end{proof}

\begin{proof}[Proof of Lemma \ref{L:finner quasiextremal}]
    We prove the lemma by induction on $\widetilde d$. If $\widetilde d=1$, then $l_j=\I_n$ for all $j$. Let $\scriptC$ be the smallest ellipsoid containing $\omega$. John ellipsoid theorem implies $|\scriptC|\sim|\omega|$. Thus, we may assume $\widetilde d>1$.
    
    Let $\{U_a\}_{a=1}^{\widetilde d}$ be a maximal partition of $\R^n$ given by $\{l_j\}_{j=1}^d$. Without loss of generality, we may assume $l_j\neq\I_n$ for any $j$. Let $\omega\subset\R^n$ be a bounded convex set satisfying \eqref{E:finner quasiextremal}. Let $P_a,P_a^\perp$ denote the coordinate projections onto $U_a,U_a^\perp$, respectively. Then $\{U_a\}_{a=2}^{\widetilde d}$ is a maximal partition of $U_1^\perp$ given by $\{P_1^\perp\circ l_j\}_{j=1}^d$. Applying Lemma \ref{L:convex} with $A=\{2,\dots,\widetilde d\}$, we have
    \begin{equation}
        |P_1^\perp(\omega)|\gtrsim\varepsilon\prod_{j=1}^d|l_j(P_1^\perp(\omega))|^\frac{1}{p_j}.
    \end{equation}
    By the inductive hypothesis, there exists a constant $C>0$ depending only on $n$ and an ellipsoid $\scriptC_1^\perp\subset U_1^\perp$ adapted to $\{P_1^\perp\circ l_j\}_{j=1}^d$ such that
    \begin{equation}
        \forall A^\prime\subset\{2,\dots,\widetilde d\}\qquad P_{A^\prime}(\omega)\subset P_{A^\prime}(\scriptC_1^\perp),\qquad|P_{A^\prime}(\scriptC_1^\perp)|\lesssim\varepsilon^{-C}|P_{A^\prime}(\omega)|.
    \end{equation}
    
    Under the assumption of \eqref{E:weak finner assumption}, Theorem \ref{T:finner weak} yields
    \begin{align*}
        |\omega| = & |P_1^\perp(\omega)|\frac{|\omega|}{|P_1^\perp(\omega)|} \\[5pt]
        \leq & \left(\displaystyle\prod_{j=1}^d|l_j(P_1^\perp(\omega))|^\frac{1}{p_j}\right)\frac{|\omega|}{|P_1^\perp(\omega)|} \\[5pt]
        = & \left(\displaystyle\prod_{\substack{j:1\leq j\leq d\\U_1\leq\ker l_j}}|l_j(P_1^\perp(\omega))|^\frac{1}{p_j}\right)\left(\displaystyle\prod_{\substack{j:1\leq j\leq d\\U_1\leq l_j(\R^n)}}|l_j(P_1^\perp(\omega))|^\frac{1}{p_j}\right)\frac{|\omega|}{|P_1^\perp(\omega)|} \\[5pt]
        = & \left(\displaystyle\prod_{\substack{j:1\leq j\leq d\\U_1\leq\ker l_j}}|l_j(\omega)|^\frac{1}{p_j}\right)\left[\displaystyle\prod_{\substack{j:1\leq j\leq d\\U_1\leq l_j(\R^n)}}\left(|l_j(P_1^\perp(\omega))|\frac{|\omega|}{|P_1^\perp(\omega)|}\right)^\frac{1}{p_j}\right],
    \end{align*}
    Invoking \eqref{E:finner quasiextremal}, we obtain
    \begin{align*}
        & \displaystyle\prod_{\substack{j:1\leq j\leq d\\U_1\leq l_j(\R^n)}}\left(|l_j(P_1^\perp(\omega))|\frac{|\omega|}{|P_1^\perp(\omega)|}\right)^\frac{1}{p_j} \\[5pt]
        \gtrsim & \varepsilon\displaystyle\prod_{\substack{j:1\leq j\leq d\\U_1\leq l_j(\R^n)}}|l_j(\omega)|^\frac{1}{p_j} = \varepsilon\displaystyle\prod_{\substack{j:1\leq j\leq d\\U_1\leq l_j(\R^n)}}\left(|P_1^\perp(l_j(\omega))|\frac{|l_j(\omega)|}{|P_1^\perp(l_j(\omega))|}\right)^\frac{1}{p_j}.
    \end{align*}
    Applying Lemma \ref{L:convex large fibers} with $U=U_1,l=l_j$ with $U_1\leq l_j(\R^n)$ yields
    \begin{equation}\label{E:last step 1}
        \frac{|l_j(\omega)|}{|P_1^\perp(l_j(\omega))|}\lesssim\varepsilon^{-C}\frac{|\omega|}{|P_1^\perp(\omega)|}.
    \end{equation}
    Recalling the assumption $l_j\neq\I_n$ and maximality of $\{U_a\}$, we observe that $l_j(\R^n)$ is the direct sum of a proper subset of $\{U_a\}_{a=1}^{\widetilde d}$. Invoking Lemma \ref{L:convex} and the inductive hypothesis again yields an ellipsoid $\scriptC^\prime\subset l_j(\R^n)$ adapted to $\{l_{j^\prime}\circ l_j\}_{j^\prime=1}^d$ satisfying
    \begin{equation}
        l_j(\omega)\subset\scriptC^\prime,\qquad|\scriptC^\prime|\leq\varepsilon^{-C}|l_j(\omega)|.
    \end{equation}
    Under the assumption of $U_1\leq l_j(\R^n)$, we have
    \begin{equation}\label{E:last step 2}
        \frac{|l_j(\omega)|}{|P_1^\perp(l_j(\omega))|}\gtrsim\varepsilon^C\frac{|\scriptC^\prime|}{|P_1^\perp(\scriptC^\prime)|}\sim\varepsilon^C|P_1(\scriptC^\prime)|\gtrsim\varepsilon^C|P_1(\omega)|.
    \end{equation}
    Let $\scriptC=\scriptC_1\times\scriptC_1^\perp$, where $\scriptC_1$ is the smallest ellipsoid containing $P_1(\omega)$. Then $|\scriptC_1|\sim|P_1(\omega)|$, and $\omega\subset\scriptC$, and invoking John ellipsoid theorem, \eqref{E:last step 1}, and \eqref{E:last step 2} yields
    \begin{equation}
        |\scriptC|\sim|\scriptC_1^\perp||\scriptC_1|\lesssim\varepsilon^{-C}|P_1^\perp(\omega)|\frac{|\omega|}{|P_1^\perp(\omega)|}=\varepsilon^{-C}|\omega|.
    \end{equation}
    Let $\{1\}\subsetneqq A\subsetneqq\{1,2,\dots,\widetilde d\}$. Invoking Lemma \ref{L:convex large fibers} with $U=P_1,l=P_A$ yields
    \begin{equation}
        \frac{|\omega|}{|P_1^\perp(\omega)|}\lesssim\frac{|P_A(\omega)|}{|P_{A^\prime}(\omega)|}.
    \end{equation}
    Then
    \begin{equation}
        |P_A(\scriptC)|\sim|\scriptC_1||P_{A^\prime}(\scriptC_1^\perp)|\lesssim\varepsilon^{-C}\frac{|\omega|}{|P_1^\perp(\omega)|}|P_{A^\prime}(\omega)|\lesssim\varepsilon^{-C}|P_A(\omega)|.
    \end{equation}
    Applying John ellipsoid theorem again, we may replace $\scriptC$ by the smallest ellipsoid containing itself adapted to $\{l_j\}_{j=1}^d$. If $\omega$ is in addition balanced, i.e., $x\in\omega$ if and only if $-x\in\omega$, then $\scriptC$ can be further chosen to be centered at the origin.
\end{proof}

\begin{corollary}\label{C:Lj Pj perp finner quasiextremal}
    Assume the conditions in Corollary \ref{C:Lj Pj perp finner}. Let $\{U_a\}_{a=1}^{\widetilde d}$ be the maximal partition of $\R^n$ given by $\{L_j\}_{j=1}^m$. Let $\varepsilon\in(0,1)$. Suppose
    \[
    |\omega|\geq\varepsilon\left(\displaystyle\prod_{j=1}^{\widetilde m}|P_j^\perp(\omega)|^\frac{\widetilde q_j}{N^\prime}\right)\left(\displaystyle\prod_{j=1}^m|L_j(\omega)|^\frac{q_j}{N^\prime}\right),\qquad N^\prime=\sum_{j=1}^mq_j
    \]
    for some bounded and convex subset $\omega\subset\R^n$. Then there exists a constant $C>0$ depending only on $n$ and an ellipsoid $\scriptC$ adapted to $\{L_j\}_{j=1}^m$ such that
    \begin{equation}
        \forall A\subset\{1,2,\dots,\widetilde d\}\qquad P_A(\omega)\subset P_A(\scriptC),\qquad|P_A(\scriptC)|\lesssim\varepsilon^{-C}|P_A(\omega)|,
    \end{equation}
    where $P_A:\R^n\longrightarrow\oplus_{a\in A}U_a$ is the orthogonal projection. In particular, when $A=\{1,2,\dots,\widetilde d\}$,
    \begin{equation}
        \omega\subset\scriptC,\qquad|\scriptC|\lesssim\varepsilon^{-C}|\omega|,
    \end{equation}
    and when $A=\{a\}$ for any $1\leq a\leq\widetilde d$, we may write $P_a\coloneqq P_{\{a\}}$, and
    \begin{equation}
        P_a(\omega)\subset P_a(\scriptC),\qquad|P_a(\scriptC)|\lesssim\varepsilon^{-C}|P_a(\omega)|.
    \end{equation}

    If $\omega$ is in addition balanced, i.e., $x\in\omega$ if and only if $-x\in\omega$, then $\scriptC$ can be further chosen to be centered at the origin.

    We obtain similar conclusions if $\{(L_j,q_j)\}_{j=1}^m$ are replaced by $\{(L_j,q_j)\}_{j=m+1}^M$.
\end{corollary}

\begin{proof}
    We verify the case of $\{(L_j,q_j)\}_{j=1}^m$. The argument for $\{(L_j,q_j)\}_{j=m+1}^M$ is almost identical. By the same argument in the proof of Corollary \ref{C:Lj Pj perp finner}, we have \eqref{E:char Lj Pj perp finner}. Furthermore, we observe that the maximal partition given by $\{L_j\}_{j=1}^m\cup\{P_j^\perp\}_{j=1}^{\widetilde m}$ is the same as that given by $\{L_j\}_{j=1}^m$. Thus, the conclusion follows from Lemma \ref{L:finner quasiextremal}.
\end{proof}

\section{Refining \texorpdfstring{$\Omega$}{Omega}}\label{S:refinement}
In the following two sections, we aim to construct an inflation map by flowing along the fibers of $\widetilde\pi_j,\widetilde\pi_j^*$. Given a smooth manifold, a flow map along some vector field $V$ in a small neighborhood is defined uniquely via the ODE
\[
\frac{d}{d\tau}\theta(t,z)=V_{\theta(t,z)},\qquad\theta(0,z)=z.
\]
For each $z$, $\theta(t,z)$ is a curve invariant under $V$, i.e., it's always tangent to $V$. The exponential flow along $V$ is defined as $e^V(z)=\theta(1,z)$. Following the notations in Section \ref{S:more notations}, we consider the exponential flows along vector fields spanned by $\mathbf X,\mathbf Y$, respectively. To be specific, for $s\in\R^n$,
\begin{equation}\label{E:exp flows}
    \begin{array}{c}
        e^{s\cdot\mathbf X}(x,y,t)\coloneqq(x+s,y,t-\frac{1}{2}s\cdot y)\in\pi^{-1}(\pi(x,y,t)), \\
        e^{s\cdot\mathbf Y}(x,y,t)\coloneqq(x,y+s,t+\frac{1}{2}x\cdot s)\in\pi_*^{-1}(\pi_*(x,y,t)).
    \end{array}
\end{equation}
For $z\in\HH^n,l\in\Z_{>0}$, we further define $\Phi_l(z):\R^{ln}\longrightarrow\HH^n$ by
\begin{equation}\label{E:compound exp flows}
    \begin{array}{rcl}
        \Phi_1(s)(z) & \coloneqq & e^{s\cdot\mathbf X}(z), \\
        \Phi_{2l}(s,u)(z) & \coloneqq & e^{u\cdot\mathbf Y}\Phi_{2l-1}(s)(z), \\
        \Phi_{2l+1}(s,u)(z) & \coloneqq & e^{u\cdot\mathbf X}\Phi_{2l}(s)(z).
    \end{array}
\end{equation}
One can similarly define $\Phi_l^*$ by swapping $\mathbf X,\mathbf Y$. When no ambiguity occurs, we may drop $z$. Note that
\begin{equation}\label{E:pi_j fibers in pi}
    \begin{array}{ll}
        \forall1\leq j\leq m&\qquad\pi_j^{-1}(\pi_j(z))\subset\pi^{-1}(\pi(z)); \\
        \forall m<j\leq M&\qquad\pi_j^{-1}(\pi_j(z))\subset\pi_*^{-1}(\pi_*(z)); \\
        \forall1\leq j\leq\widetilde m&\qquad\widetilde\pi_j^{-1}(\widetilde\pi_j(z))\subset\pi^{-1}(\pi(z)), \\
        &\qquad(\widetilde\pi_j^*)^{-1}(\widetilde\pi_j^*(z))\subset\pi_*^{-1}(\pi_*(z)).
    \end{array}
\end{equation}
Furthermore, when restricted to a fiber of $\pi$, $\{\pi_j\}_{j=1}^m,\{\widetilde\pi_j\}_{j=1}^{\widetilde m}$ are affine coordinate projections. To be specific,
\begin{align}\label{E:pi_j restriction are affine L_j}
    \pi_j|_{\pi^{-1}(y,t)}(s,y,t-\frac{1}{2}s\cdot y) = & (L_js,y,t), & 1\leq j\leq m, \\
    \widetilde\pi_j|_{\pi^{-1}(y,t)}(s,y,t-\frac{1}{2}s\cdot y) = & (P_js,y,t), & 1\leq j\leq\widetilde m.\notag
\end{align}
One has similar conclusions for $\{\pi_j\}_{j=m+1}^M,\{\widetilde\pi_j^*\}_{j=1}^{\widetilde m}$ restricted to fibers of $\pi_*$. Thus, it is convenient to parametrize the fibers of $\pi_j,\widetilde\pi_j,\widetilde\pi_j^*$. For $\Omega\subset\HH^n,z\in\HH^n$, define
\begin{equation}\label{E:T}
    \begin{array}{rcl}
        T^\Omega(z) & \coloneqq & \{s\in\R^n:e^{s\cdot\mathbf X}(z)\in\Omega\}; \\
        T^\Omega_*(z) & \coloneqq & \{s\in\R^n:e^{s\cdot\mathbf Y}(z)\in\Omega\}; \\
        T^\Omega_j(z) & \coloneqq &
        \begin{cases}
            \{s_j\in K_j:e^{s_j\cdot\mathbf X}(z)\in\Omega\}, & 1\leq j\leq m,\\
            \{s_j\in K_j:e^{s_j\cdot\mathbf Y}(z)\in\Omega\}, & m<j\leq M;
        \end{cases} \\
        \widetilde T^\Omega_j(z) & \coloneqq & \{s\in\R^{\widetilde k_j}:e^{(s,0)\cdot\mathbf X}(z)\in\Omega\}; \\
        (\widetilde T^*_j)^\Omega(z) & \coloneqq & \{s\in\R^{\widetilde k_j}:e^{(s,0)\cdot\mathbf Y}(z)\in\Omega\}.
    \end{array}
\end{equation}
When no ambiguity occurs, we may drop $\Omega$ and $z$. Then $\pi_j|_{\pi^{-1}(\pi(z))\cap\Omega},1\leq j\leq m$ can be identified with $L_j|_{T^\Omega(z)}$, and $\widetilde\pi_j|_{\pi^{-1}(\pi(z))\cap\Omega}$ with $P_j|_{\widetilde T^\Omega(z)}$. One has similar observations about $\pi_j,m<j\leq M$ and $\widetilde\pi_j^*$.

\begin{definition}\label{D:refinement}
Let $\varepsilon>0$. $\widetilde\Omega$ is called an \textit{$\varepsilon$-refinement} of $\Omega\subset\R^d$ if $\widetilde\Omega\subset\Omega$ and $|\widetilde\Omega|\gtrsim\varepsilon|\Omega|$. If $\varepsilon=1$, we say $\widetilde\Omega$ is a \textit{refinement} of $\Omega$.
\end{definition}

The following lemma is adapted from Lemma 1 of \cite{christ1998convolution} and Lemma 3.1 of \cite{christ2011quasiextremals}. Unlike \cite{christ1998convolution,TaoWright,christ2011quasiextremals}, where $|\pi_j^{-1}(\pi_j(z))\cap\Omega|\gtrsim\alpha_j$ for all $z\in\Omega$ after refinement, we refine the larger fibers $\pi^{-1}(\pi(z))\cap\Omega$ so that within each larger fiber, the average size of the $\pi_j$ fibers ($|T^\Omega(z)|/|L_j(T^\Omega(z))|$) are bounded below by $\alpha_j$ for each $j\leq m$, while a specific $\pi_j$ fiber can be small (though not many such $\pi_j$ fibers exist). Similarly, we can bound the average size of the $\pi_j$ fibers for $j>m$ within each $\pi_*$ fiber. However, for each intersection of $\Omega$ and a $\widetilde\pi_j$ or $\widetilde\pi_j^*$ fiber, the refinement guarantees a lower bound by $\beta_j$ or $\beta_j^*$, respectively, i.e., $|\widetilde T^\Omega_j(z)|\gtrsim\beta_j,|(\widetilde T^*_j)^\Omega(z)|\gtrsim\beta_j^*$.

\begin{lemma}\label{L:refinement}
    Let $\Omega\in\HH^n$ be bounded and measurable, and let
    \begin{gather}
        \alpha_j=\frac{|\Omega|}{|\pi_j(\Omega)|},\qquad j=1,\dots,M; \\
        \beta_j=\frac{|\Omega|}{|\widetilde\pi_j(\Omega)|},\quad\beta_j^*=\frac{|\Omega|}{|\widetilde\pi_j^*(\Omega)|},\qquad j=1,\dots,\widetilde m.
    \end{gather}
    Let $A\in\Z_{>0}. $Then there exist $c>0$ depending on $n,M,A$, a chain of refinements $\Omega_1\subset\Omega_2\subset\cdots\subset\Omega_A\subset\Omega$ (the implicit constants may depend on $A$), $z_0\in\HH^n$, and sets $S_l\in\R^{ln}$ for $l=1,2,\dots,A$, such that
    \begin{gather}
        \label{E:refinement 0}S_1=T^{\Omega_1}(z_0); \\
        \label{E:refinement 1}S_{l+1}=
        \begin{cases}
            \{(s^\prime,s)\in S_l\times\R^n:s\in T^{\Omega_{l+1}}_*(\Phi_l(s^\prime)(z_0))\}, & \text{$l$ is odd}, \\
            \{(s^\prime,s)\in S_l\times\R^n:s\in T^{\Omega_{l+1}}(\Phi_l(s^\prime)(z_0))\}, & \text{$l$ is even};
        \end{cases} \\
        \label{E:refinement 2}\frac{|S_1|}{|L_j(S_1)|}\geq c\alpha_j,\quad\frac{|\scriptG_l(s)|}{|L_j(\scriptG_l(s))|}\geq c\alpha_j,\qquad j=1,\dots,m, s\in S_{2l}; \\
        \label{E:refinement 3}\frac{|\scriptF_l(s)|}{|L_j(\scriptF_l(s))|}\geq c\alpha_j,\quad j=m+1,\dots,M, s\in S_{2l-1}; \\
        \label{E:refinement 4}\forall z\in\Omega_l\qquad
        \begin{cases}
            |\widetilde T^{\Omega_l}_j(z)|\geq c\beta_j, & \text{$l$ is odd}, \\[5pt]
            |(\widetilde T^*_j)^{\Omega_l}(z)|\geq c\beta_j^*, & \text{$l$ is even}. \\
        \end{cases}
    \end{gather}
    Here, for $l\in\Z_{>0},s\in S_{2l-1}$,
    \[
    \scriptF_l(s)\coloneqq\{u\in\R^n:(s,u)\in S_{2l}\};
    \]
    for $s\in S_{2l}$,
    \[
    \scriptG_l(s)\coloneqq\{v\in\R^n:(s,v)\in S_{2l+1}\}.
    \]

    We call the $(z_0,S_1,\scriptF_l,\scriptG_l)$ a \textit{refined flow scheme} of $(\Omega,\pi,A)$. Similarly, one can obtain a refined flow scheme of $(\Omega,\pi_*,A)$ by substituting $\Phi_l^*$ for $\Phi_l$.
\end{lemma}

\begin{proof}
    We present the proof for odd $A$. The case of even $A$ is handled similarly. Let $c>0$ be a small constant whose value will be determined later. Let
    \begin{align*}
        \Omega_j^\prime\coloneqq\left\{z\in\Omega:\frac{|T^\Omega(z)|}{|L_j(T^\Omega(z))|}\geq c\alpha_j\right\},\qquad1\leq j\leq m; \\[5pt]
        \Omega_{j+m}^\prime\coloneqq\left\{z\in\Omega:\frac{|T^\Omega(z)|}{|P_j(T^\Omega(z))|}\geq c\beta_j\right\},\qquad1\leq j\leq\widetilde m.
    \end{align*}
    Recall that the $P_j$ are defined in \eqref{E:pi j twiddle}. By the coarea formula,
    \[
    \begin{array}{rcl}
        |\Omega\setminus\Omega_j^\prime| & = & \int_{\pi(\Omega\setminus\Omega_j^\prime)}|T^\Omega(0,y,t)|dydt \\
        & < & c\alpha_j\int_{\pi(\Omega)}|L_j(T^\Omega(0,y,t))|dydt \\
        & = & c\alpha_j|\pi_j(\Omega)| \\
        & = & c|\Omega|,
    \end{array}
    \]
    Similarly, $|\Omega\setminus\Omega_j^\prime|<c|\Omega|$ for $m<j\leq m+\widetilde m$. Take $\Omega^\prime=\cap_{j=1}^{m+\widetilde m}\Omega_j^\prime$. Choosing a sufficiently small $c<\frac1{100(m+\widetilde m)}$, we have $|\Omega^\prime|\geq(1-c(m+\widetilde m))|\Omega|$. Note that each $\pi^{-1}(\pi(z))\cap\Omega\neq\emptyset$ is either completely discarded or untouched. Within each untouched fiber, the average size of the $\pi_j$ or $\widetilde\pi_j$ fibers is large. To be specific,
    \begin{equation}\label{E:avg alpha_j beta_j in pi inverse}
        \frac{|T^{\Omega^\prime}(z)|}{|L_j(T^{\Omega^\prime}(z))|}\geq c\alpha_j,\quad1\leq j\leq m;\qquad\frac{|T^{\Omega^\prime}(z)|}{|P_j(T^{\Omega^\prime}(z))|}\geq c\beta_j,\quad1\leq j\leq\widetilde m.
    \end{equation}
    Let $\widetilde\Omega_{\widetilde m}\coloneqq\Omega^\prime$. Then $|\widetilde\Omega_{\widetilde m}|\geq(1-c_{\widetilde m})|\Omega|$ for some small constant $c_{\widetilde m}\in(0,1)$ depending on $m,n$. Suppose for $j=\widetilde m,\widetilde m-1,\cdots,1$, we have $|\widetilde\Omega_j|\geq(1-c_j)|\Omega|$ for some small constant $c_j\in(0,1)$ depending on $n$. Let
    \[
    \widetilde\Omega_{j-1}\coloneqq\{z\in\widetilde\Omega_j:|\widetilde T_j^{\widetilde\Omega_j}(z)|<c\beta_j\}.
    \]
    Note that $(0,x_j,y,t-\frac{1}{2}(0,x_j)\cdot y)\in\widetilde\pi_j^{-1}(x_j,y,t)$. The coarea formula yields
    \begin{align*}
        |\widetilde\Omega_j\setminus\widetilde\Omega_{j-1}|=&\int_{\widetilde\pi_j(\widetilde\Omega_j\setminus\widetilde\Omega_{j+1})}|\{s\in\R^{\widetilde k_j}:(s,x_j,y,t-\frac{1}{2}(s,x_j)\cdot y)\in\widetilde\Omega_j\}|dx_jdydt \\
        =&\int_{\widetilde\pi_j(\widetilde\Omega_j\setminus\widetilde\Omega_{j-1})}|\widetilde T^{\widetilde\Omega_j}_j(0,x_j,y,t-\frac{1}{2}(0,x_j)\cdot y)|dx_jdydt \\
        <&c|\widetilde\pi_j(\Omega)|\beta_j \\
        \leq&c_{j-1}|\widetilde\Omega_j|,
    \end{align*}
    where $c_{j-1}=\frac{c}{1-c_j}$. Furthermore, \eqref{E:pi_j fibers in pi}, \eqref{E:avg alpha_j beta_j in pi inverse}, and the inclusion relationship of $\ker P_j$ -- i.e., $\ker P_1<\ker P_2<\cdots<\ker P_{\widetilde m}=\R^n$ -- imply that for small $c_{j-1}>0$,
    \[
    \forall z\in\widetilde\Omega_{l-1}\quad\forall j\leq l\leq\widetilde m\qquad|\widetilde T^{\widetilde\Omega_{j-1}}_l(z)|\geq c_{j-1}\beta_l,
    \]
    and \eqref{E:avg alpha_j beta_j in pi inverse} is preserved. Let $\Omega_A=\widetilde\Omega_0$. We repeat the process $A-1$ more times to obtain $\Omega_A\supset\Omega_{A-1}\supset\cdots\supset\Omega_1$. Specifically, we refine $T^{\Omega_{A-2l}}_*$ to obtain $\Omega_{A-(2l+1)}\subset\Omega_{A-2l}$, and refine $T^{\Omega_{A-(2l+1)}}$ to obtain $\Omega_{A-{2l+2}}\subset\Omega_{A-(2l+1)}$. We finish the proof by picking any $z_0\in\Omega_1$.
\end{proof}

\section{An inflation map}\label{S:inflation}
Recall the notations introduced in Section \ref{S:more notations}. Let
\begin{equation}\label{E:qj double twiddle}
    \dbtilde q_j=
    \begin{cases}
        \widetilde q_j, & 1\leq j<\widetilde m, \\
        \widetilde q_{\widetilde m}-N, & j=\widetilde m.
    \end{cases}
\end{equation}
We deduce from \eqref{B''}, \eqref{E:order of Q_i}, and \eqref{E:qj twiddle} that $\dbtilde q_j\geq1$. Furthermore, \eqref{A'} implies
\begin{equation}\label{A''}
    \displaystyle\sum_{j=1}^{\widetilde m}\widetilde k_j\dbtilde q_j=\left(\displaystyle\sum_{j=1}^{\widetilde m}\widetilde k_j\widetilde q_j\right)-nN=\left(\displaystyle\sum_{j=m+1}^Mk_jq_j\right)-nN=N.
\end{equation}

Let $\Omega\subset\HH^n$. By Lemma \ref{L:refinement}, there exists a refined flow scheme $(z_0,S_1,\scriptF_l,\scriptG_l)$ of $(\Omega,\pi,A)$ for any $A\geq3$. Write $S=S_1,\scriptF=\scriptF_1,\scriptG=\scriptG_1$. Write $s=(s^{j,l}),\mathbf u=(u^{j,l,a})$ with $s^{j,l},u^{j,l,a}\in\R^n$ and the indices from
\begin{equation}
    1\leq j\leq\widetilde m,\quad1\leq l\leq\dbtilde q_j,\quad1\leq a\leq\widetilde k_j.
\end{equation}
The domain $\Omega^\flat$ of the inflation map takes the form
\begin{equation}\label{E:Omega flat}
    \Omega^\flat\coloneqq\{(s,\mathbf u,\mathbf v)\in\Xi\times\R^{2nN}: u^{j,l,a}\in\scriptF(s^{j,l}),v^{j,l,a}\in\scriptG(s^{j,l},u^{j,l,a})\},
\end{equation}
and the inflation map $\Psi:\Omega^\flat\longrightarrow\R^{(2n+1)N}$ is defined as
\begin{equation}\label{E:inflation}
    \Psi(s, \mathbf u, \mathbf v)\coloneqq(e^{v^{j,l,a}\cdot\mathbf X}e^{u^{j,l,a}\cdot\mathbf Y}e^{s^{j,l}\cdot\mathbf X}(z_0))_{1\leq j\leq\widetilde m,1\leq l\leq\dbtilde q_j,1\leq a\leq\widetilde k_j}.
\end{equation}
Here, $|\scriptF(s^{j,l})|\gtrsim\beta_{\widetilde m}^*,|\scriptG(s^{j,l},u^{j,l,a})|\gtrsim\beta_{\widetilde m}$, and $\Xi$ is an embedding of a subset of $\R^N$ into $\R^{n(\sum_{j=1}^{\widetilde m}\dbtilde q_j)}$ via $\iota:\Xi\hookrightarrow\R^{n(\sum_{j=1}^{\widetilde m}\dbtilde q_j)}$. $\iota,\Xi$ are defined within a broader class of embeddings $\iota(\widetilde k,\dbtilde q,E,\tau)$ and sets $\Xi(\widetilde k,\dbtilde q,E,\tau)$.

\begin{definition}\label{D:iota Xi}
    Consider a \textit{class of parameters} $(\widetilde k,\dbtilde q,E,\tau)$, where
    \begin{itemize}
        \item $\widetilde k=(\widetilde k_0,\widetilde k_1,\dots,\widetilde k_{\widetilde m})$ with $\widetilde m>0$ satisfies $0=\widetilde k_0<\widetilde k_1<\widetilde k_2<\cdots<\widetilde k_{\widetilde m}$;

        \item $\dbtilde q=(\dbtilde q_1,\dots,\dbtilde q_{\widetilde m})\in\Z_{>0}^{\widetilde m}$;

        \item $E\subset\R^{\widetilde k_{\widetilde m}}$; and

        \item $\tau=(\tau_1,\dots,\tau_d)$ is a $d$-tuple of reals for some $d\geq0$. If $d=0$, then $\tau=()$ is the empty tuple.
    \end{itemize}
    Write $\widetilde Q_j=\sum_{i=j}^{\widetilde m}\dbtilde q_i,\xi=(\xi^{j,l}),s=(s^{j,l})$ with $\xi^{j,l}\in\R^{\widetilde k_j},s^{j,l}\in\R^{\widetilde k_{\widetilde m}+d}$ for $1\leq j\leq\widetilde m,1\leq l\leq\dbtilde q_j$.
    \begin{itemize}
        \item[$\blacksquare$] Define $\iota(\widetilde k,\dbtilde q,E,\tau):(\xi^{j,l})\mapsto(s^{j,l})$ as in Matrix \ref{M:Xi_embedding_1}. (See also Matrix \ref{M:s^jl s^(j+1)l}.) Specializing in $E=S,\tau=()$, we let $\iota\coloneqq\iota(\widetilde k,\dbtilde q,S,())$.
        \begin{figure}
            \centering
            \includegraphics[width=1\linewidth]{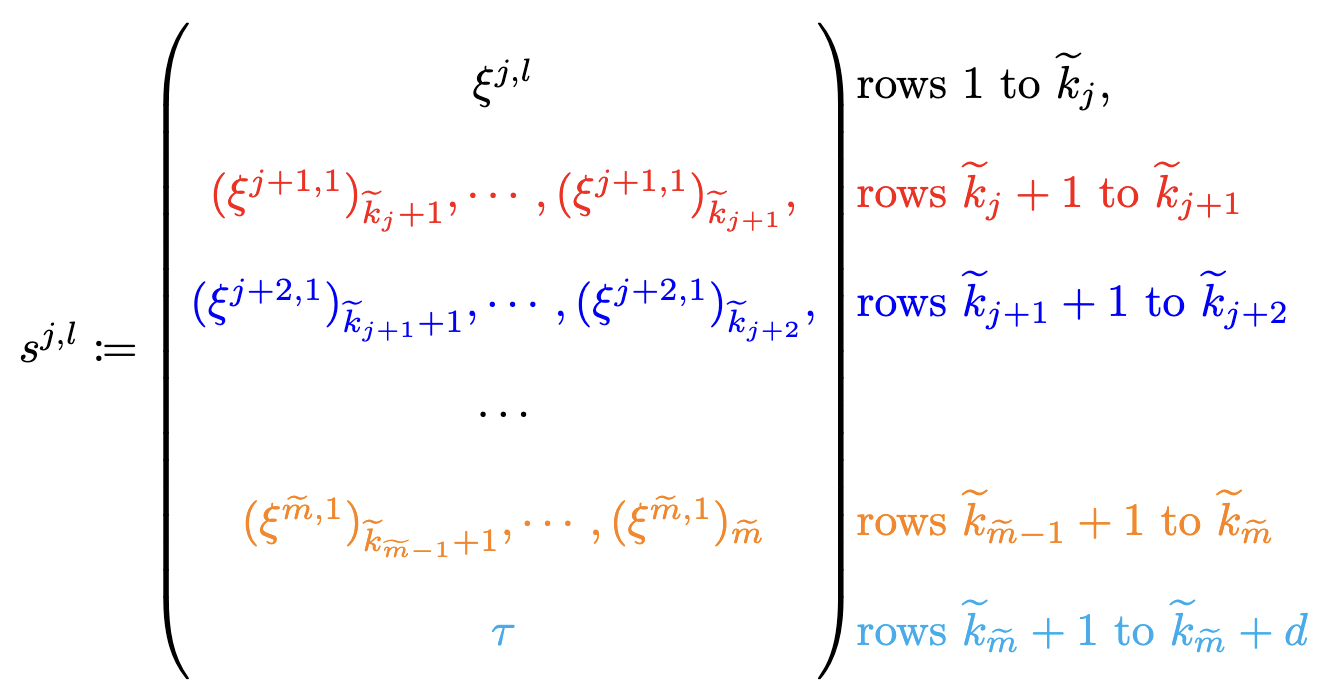}
            \captionsetup{labelformat=matrix}
            \caption{}\label{M:Xi_embedding_1}
        \end{figure}
        
        \item[$\blacksquare$] Define $\Xi(\widetilde k,\dbtilde q,E,\tau)\coloneqq(\iota(\widetilde k,\dbtilde q,E,\tau))^{-1}((E\times\{\tau\})^{\widetilde Q_1})$. Specializing in $E=S,\tau=()$, we let $\Xi\coloneqq\Xi(\widetilde k,\dbtilde q,S,())$. From now on, we identify $\xi$ and $s$, and refer to $\xi$ if we want $\xi^{j,l}\in\R^{\widetilde k_j}$, and to $s$ if we want $s^{j,l}\in\R^{\widetilde k_{\widetilde m}+d}$.
        
        \item[$\blacksquare$] Let $G_0:\R^{\widetilde k_{\widetilde m}}\longrightarrow[0,\infty)$ be a bounded measurable function. Define
        \begin{equation}\label{E:G}
            G_j(s_j)\coloneqq\left(\int_{P_{j-1}(P_j^{-1}(s_j)\cap E)}G_{j-1}^\frac1{\widetilde Q_j}(s)ds\right)^{\widetilde Q_j},\qquad1\leq j\leq\widetilde m.
        \end{equation}
        where $s_j$ is a $(\widetilde k_{\widetilde m}-\widetilde k_j)$-tuple of reals, $P_j$ is the orthogonal projection from $\R^{\widetilde k_{\widetilde m}}$ onto the last $\widetilde k_{\widetilde m}-\widetilde k_j$ coordinates for $0\leq j<\widetilde m$, $P_{\widetilde m}$ is the zero map with $P_{\widetilde m}^{-1}(s_{\widetilde m})=\R^{\widetilde k_{\widetilde m}}$. For fixed $s_j$, $s\in P_{j-1}(P_j^{-1}(s_j)\cap E)$ is of the form $s=(s^\prime,s_j)$. The integral over $P_{j-1}(P_j^{-1}(s_j)\cap E)$ is understood as the Lebesgue integral against $s^\prime$.
    \end{itemize}
\end{definition}

\begin{remarks}\hfill
    \begin{enumerate}
        \item Although we have defined the $P_j$ in \eqref{E:pi j twiddle}, they are consistent with the $P_j$ above when $\widetilde k_{\widetilde m}=n$, so the abuse of notation is well tolerated.
        
        \item In Example \ref{Ex:isoperimetric}, we have $\widetilde m=1,\widetilde k_1=3,\dbtilde q_1=1,E=S,\tau=()$. Then $\Omega^\flat$ defined in this section is consistent with that in Subsection \ref{S:proof ex iso}. In Example \ref{Ex:aniso}, we have $\widetilde m=2,\widetilde k=(0,1,4),\dbtilde q=(5,1),E=S,\tau=()$, and see that the $\Omega^\flat$ in Subsection \ref{S:proof ex aniso} also matches the definition here.
    \end{enumerate}
\end{remarks}

\begin{figure}
    \centering
    \includegraphics[width=1\linewidth]{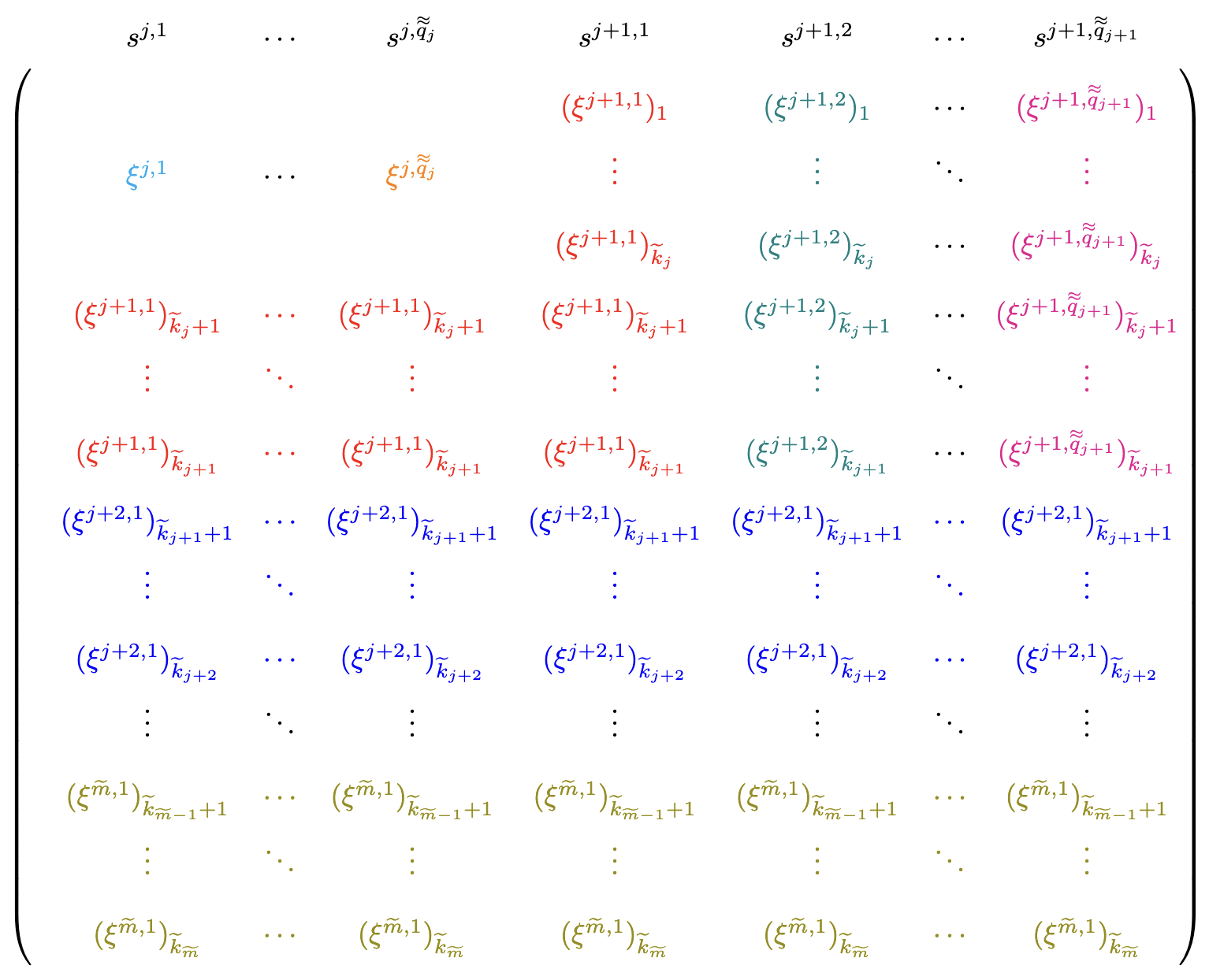}
    \captionsetup{labelformat=matrix}
    \caption{These are the forms of \(s^{j,l}\) and \(s^{j+1,l}\) when \(\tau=()\). Each \(\xi^{j,l}\) is distinctly colored. \(\Xi_{j+1}\) includes \(\xi^{r,l}\) for \(r\leq j\) or \(r=j+1,l=1\).}\label{M:s^jl s^(j+1)l}
\end{figure}

We record some properties of the parameters $(\widetilde k,\dbtilde q,E,\tau)$ below. 

\begin{props}\label{Prop:Xi} Let $(\widetilde k,\dbtilde q,E,\tau)$ be a class of parameters of length $\widetilde m$.
    \begin{enumerate}
        \item\label{Prop:m twiddle=1} If $\widetilde m=1$, then $\Xi(\widetilde k,\dbtilde q,E,\tau)=E^{\dbtilde q_1}$ is a product set.

        \item\label{Prop:nested Xi} For $1\leq j\leq\widetilde m$, consider a $(\widetilde k_{\widetilde m}-\widetilde k_j)$-tuple $s_j$ and
        \[
        \Xi_j(s_j)\coloneqq\Xi((\widetilde k_1,\dots,\widetilde k_j),(\dbtilde q_1,\dots,\dbtilde q_{j-1},\dbtilde q_j+\delta_{j-\widetilde m}),E^\prime(s_j),(s_j,\tau))
        \]
        with $E^\prime(s_j)=\{s\in\R^{\widetilde k_j}:(s,s_j)\in E\}$ and $\delta_{j-\widetilde m}=1$ if $j=\widetilde m$, otherwise zero. Then $\Xi_1(s_1)=(P_1^{-1}(s_1)\cap E)^{\dbtilde q_1+1}$, and for $1<j\leq\widetilde m$,
        \[
        \Xi_j(s_j)=\left(\displaystyle\bigsqcup_{s_{j-1}\in P_{j-1}(P_j^{-1}(s_j))}\Xi_{j-1}(s_{j-1})\right)\times(P_j^{-1}(s_j)\cap E)\times\{\tau\})^{\dbtilde q_j-\delta_{j-\widetilde m}},
        \]
        and $\Xi_{\widetilde m}(\tau_{\widetilde m})=\Xi(\widetilde k,\dbtilde q,E,\tau)$. (See also Matrix \ref{M:s^jl s^(j+1)l}.)

        \item\label{Prop:E layer} Let $f:\R^{n-\widetilde k_{j-1}}\longrightarrow[0,\infty)$ be any bounded measurable function. Then,
        \[
        \int_{P_{j-1}(P_j^{-1}(s_j)\cap E)}\int_{P_{j-1}^{-1}(s_{j-1})\cap E}f(s)dsds_{j-1}=\int_{P_j^{-1}(s_j)\cap E}f(s)ds.
        \]

        Let $\sigma\in(0,1)>0$. If $E$ satisfies the following property:
        \begin{equation}\label{E:regular E}
            \forall 1\leq j<\widetilde m\quad\forall s_j\in P_j(E)\qquad\frac{|E|}{|P_j(E)|}\lesssim|P_j^{-1}(s_j)\cap E|\lesssim\sigma^{-1}\frac{|E|}{|P_j(E)|},
        \end{equation}
        then
        \begin{equation}\label{E:section k j twiddle to k j+1}
            |P_j(P_{j+1}^{-1}(s_{j+1})\cap E)|\gtrsim\sigma\frac{|P_j(E)|}{|P_{j+1}(E)|},\qquad1\leq j<\widetilde m.
        \end{equation}
    \end{enumerate}
\end{props}

\begin{lemma}\label{L:G m twiddle lower bound}
    Let $(\widetilde k,\dbtilde q,E,\tau)$ be a class of parameters of length $\widetilde m$, and let $\sigma\in(0,1)>0$. Assume \eqref{E:regular E} and $G_0\equiv1$. Then there exists some constant $C>1$ depending only on $(\widetilde k,\widetilde q)$ such that
    \begin{equation}\label{E:G m twiddle lower bound}
        G_{\widetilde m}\gtrsim\sigma^C\displaystyle\prod_{j=1}^{\widetilde m}\left(\frac{|E|}{|P_j(E)|}\right)^{\dbtilde q_j}.
    \end{equation}
\end{lemma}

\begin{proof}
    Write $\beta_j\coloneqq|\omega|/|P_j(\omega)|$ for $1\leq j\leq\widetilde m$, and $\widetilde Q_j=\sum_{r=j}^{\widetilde m}\dbtilde q_r$. Note that $G_1(s_1)=|P_1^{-1}(s_1)\cap E|^{\widetilde Q_1}\gtrsim\beta_1^{\widetilde Q_1}$. If $\widetilde m=1$, then we are done. Otherwise, suppose
    \begin{equation}\label{E:Gj lower bound}
        G_j(s_j)\gtrsim\sigma^C\beta_j^{\widetilde Q_j}\displaystyle\prod_{r=1}^{j-1}\beta_r^{\dbtilde q_r},
    \end{equation}
    holds for $1\leq j<\widetilde m$. It remains to prove \eqref{E:Gj lower bound} holds for $j+1$. By the induction hypothesis and \eqref{E:section k j twiddle to k j+1},
    \begin{align*}
        G_{j+1}(s_{j+1})=    &    \left(\int_{P_j(P_{j+1}^{-1}(s_{j+1})\cap E)}G_j^\frac1{\widetilde Q_{j+1}}(s)ds\right)^{\widetilde Q_{j+1}} \\[5pt]
        \gtrsim   &   \sigma^C\left(\frac{\beta_{j+1}}{\beta_j}\right)^{\widetilde Q_{j+1}}\left[\beta_j^{\widetilde Q_j}\displaystyle\prod_{r=1}^{j-1}\beta_r^{\dbtilde q_r}\right]^\frac{\widetilde Q_{j+1}}{\widetilde Q_{j+1}} \\[5pt]
        =   &   \sigma^C\beta_{j+1}^{\widetilde Q_{j+1}}\displaystyle\prod_{r=1}^{j}\beta_r^{\dbtilde q_r}.
    \end{align*}
    The constant $C$ may vary line by line, but it always depends on $\widetilde m,\dbtilde q$.
\end{proof}

\begin{lemma}\label{L:G m twiddle upper bound}
    Let $(\widetilde k,\dbtilde q,E,\tau)$ be a class of parameters of length $\widetilde m$. Assume \eqref{E:regular E} for some $\sigma\in(0,1)$. Let $\varepsilon\in(0,\sigma]$. Suppose
    \begin{equation}\label{E:G m twiddle upper bound}
        G_{\widetilde m}\lesssim\varepsilon^{-C_{\widetilde m}}\displaystyle\prod_{j=1}^{\widetilde m}\left(\frac{|E|}{|P_j(E)|}\right)^{\dbtilde q_j}.
    \end{equation}
    for some constant $C_{\widetilde m}>1$ depending only on $(\widetilde k,\widetilde q)$ and $\varepsilon$. Then there exists a refinement of $E$ on which $G_0\lesssim\varepsilon^{-C}$. If $E$ further satisfies
    \begin{equation}\label{E:E large fibers}
        \forall1\leq j\leq\widetilde m\quad\forall x\in P_j(E)\qquad|P_j^{-1}(x)\cap E|\gtrsim\frac{|E|}{|P_j(E)|},
    \end{equation}
    then the refinement can be chosen to preserve \eqref{E:E large fibers}.
\end{lemma}

\begin{proof}
    Write $\beta_j\coloneqq|\omega|/|P_j(\omega)|$ for $1\leq j\leq\widetilde m$, and $\widetilde Q_j=\sum_{r=j}^{\widetilde m}\dbtilde q_r$. Set $\beta_0=1$ and
        \begin{equation}
            B_j=
            \begin{cases}
                1, & j=0, \\
                \beta_j^{\widetilde Q_j}\prod_{i=1}^{i-1}\beta_i^{\dbtilde q_i}, & 1\leq j\leq\widetilde m.
            \end{cases}
        \end{equation}
        Then $B_j/B_{j-1}=(\beta_j/\beta_{j-1})^{\widetilde Q_j}$. Then \eqref{E:G m twiddle upper bound} becomes
        \begin{equation}\label{E:G m upper base}
            G_{\widetilde m}\leq C\varepsilon^{-C}B_{\widetilde m}.
        \end{equation}
        For $1<j\leq\widetilde m$ and a $(\widetilde k_{\widetilde m}-\widetilde k_j)$-tuple $s_j$, define
        \begin{equation}
            E_{j-1}(s_j)\coloneqq\left\{s_{j-1}\in P_{j-1}(P_j^{-1}(s_j)\cap E):G_{j-1}(s_{j-1})\leq C_{j-1}\varepsilon^{-C_{j-1}}B_{j-1}\right\},
        \end{equation}
        where $C_{j-1}\geq1$ are large constants to be determined. Suppose for $j\leq r\leq\widetilde m$, there exists a large constant $C_r\geq1$ such that 
        \begin{equation}\label{E: G induction}
            G_r(s_r)\leq C_r\varepsilon^{-C_r}B_r.
        \end{equation}
        (\eqref{E:G m upper base} is the base case.) By Markov's inequality, \eqref{E: G induction} and \eqref{E:section k j twiddle to k j+1},
        \begin{align*}
            & |P_{j-1}(P_j^{-1}(s_j)\cap E)\setminus E_{j-1}(s_j)| \\
            \leq & \left(\frac{\varepsilon^{C_{j-1}}}{C_{j-1}B_{j-1}}\right)^\frac{1}{\widetilde Q_j}\int_{P_{j-1}(P_j^{-1}(s_j)\cap E)}G_{j-1}(s_{j-1})^\frac{1}{\widetilde Q_j}ds_{j-1} \\[5pt]
            = & \left(\frac{\varepsilon^{C_{j-1}}}{C_{j-1}B_{j-1}}G_j\right)^\frac{1}{\widetilde Q_j} \leq \varepsilon^{(C_{j-1}-C_j)/\widetilde Q_j-1}|P_{j-1}(P_j^{-1}(s_j)\cap E)|.
        \end{align*}
        Thus, for sufficiently large $C_{j-1}$, \eqref{E:regular E} implies that $E\cap(\sqcup_{s_j}P_j^{-1}(s_j))$ is a refinement of $E$. Let this refinement be our new $E$. In particular, \eqref{E: G induction} holds true for every $j-1\leq r\leq\widetilde m$ and every $s_r$ in $P_r(E)$ (the $P_r$ image of the new $E$). By (a downward) induction on $j$, we have $G_0\lesssim\varepsilon^{-C}$ on $E$ (which has been refined at most $\widetilde m$ times). If need be, we can further refine $E$ to preserve \eqref{E:E large fibers}. This is possible due to the nested kernels of $P_j$. (See the proof of \eqref{E:refinement 4} in Lemma \ref{L:refinement}.)
\end{proof}

\begin{lemma}\label{L:diffeo}
    Let $\Omega\subset\HH^n$ be bounded and measurable. Then $\Psi(\Omega^\flat)\subset\Omega^N$, and
    \begin{equation}\label{E:diffeo}
        |\Omega|^N\geq|\Psi(\Omega^\flat)|\gtrsim\beta_{\widetilde m}^N\int_\Xi\displaystyle\prod_{j=1}^{\widetilde m}\displaystyle\prod_{l=1}^{\dbtilde q_j}f_{j,l}(s^{j,l})ds,
    \end{equation}
    where the integral over $\Xi$ is understood as the Lebesgue integral against $\xi$,
    \[
    f_{j,l}(s)\coloneqq\int_{(\scriptF(s))^{\widetilde k_j}}|\det(\mathbf u^{j,l})|\prod_{a=1}^{\widetilde k_j}du^{j,l,a},\qquad s\in S,
    \]
    and $\mathbf u^{j,l}$ is the $\widetilde k_j\times\widetilde k_j$ matrix whose $a$-th column is the first $\widetilde k_j$ coordinates of $u^{j,l,a}$.
\end{lemma}

\begin{proof}
    By the construction of $\Omega^\flat$ and $\Psi$ and by Lemma \ref{L:refinement}, we have $\Psi(\Omega^\flat)\subset\Omega$. Unpacking the definition of $\Psi$, we observe that it is a vector-valued function whose components are linear or quadratic polynomials. Furthermore,
    \[
    \det\Psi(s,\mathbf u,\mathbf v)=\displaystyle\prod_{j=1}^{\widetilde m}\displaystyle\prod_{l=1}^{\dbtilde q_j}\det(\mathbf u^{j,l})\neq0
    \]
    almost everywhere. By B\'ezout's Theorem, there is a submanifold $\Gamma$ with $\dim\Gamma<(2n+1)N$ such that $\#\Psi^{-1}(z)\lesssim1$ for $z\notin\Gamma$. Thus \eqref{E:diffeo} holds.
\end{proof}

Later on, we will show that on the diagonal (i.e., $s^{j,l}$ all take the same value, denoted by $s$, for all $j,l$),
\[
\displaystyle\prod_{j,l}f_{j,l}(s)\gtrsim\displaystyle\prod_{j=m+1}^M\alpha_j^{q_j}.
\]
However, the $s^{j,l}$ are in general not equal. The following lemma relates the off-diagonal terms (i.e., the $s^{j,l}$ don't take the same value) to the diagonal terms.

\begin{lemma}\label{L:am-gm}
    Consider a class of parameters $(\widetilde k,\dbtilde q,E,\tau)$ in Definition \ref{D:iota Xi} and their associated $\iota,\Xi,G_j$ with $G_0=\prod_{j,l}f_{j,l}$, where $G_j$ are defined in \eqref{E:G} and $f_{j,l}$ are bounded measurable nonnegative functions on $\R^n$ for $1\leq j\leq\widetilde m,1\leq l\leq\dbtilde q_j$. Then,
    \begin{equation}\label{E:am-gm}
        \int_{\Xi(\widetilde k,\dbtilde q,E,\tau)}\displaystyle\prod_{j=1}^{\widetilde m}\displaystyle\prod_{l=1}^{\dbtilde q_j}f_{j,l}(s^{j,l})ds\geq G_{\widetilde m}.
    \end{equation}
\end{lemma}

\begin{proof}
    We will prove the lemma by induction and repeated applications of H\"older's inequality. Note that $P_0(P_1^{-1}(s_1)\cap E)=P_1^{-1}(s_1)\cap E$. If $\widetilde m=1$, then \eqref{E:am-gm} is H\"older's inequality. Otherwise, H\"older's inequality and Property \ref{Prop:Xi}.\ref{Prop:nested Xi} yield
    \begin{align*}
        G_1(s_1)= & \left(\int_{P_1^{-1}(s_1)\cap E}\displaystyle\prod_{j=1}^{\widetilde m}\displaystyle\prod_{l=1}^{\dbtilde q_j}(f_{j,l}(s))^\frac{1}{\widetilde Q_1}ds\right)^{\widetilde Q_1} \\[5pt]
        \leq & \left(\int_{P_1^{-1}(s_1)\cap E}\displaystyle\prod_{\substack{j=1,1\leq l\leq\dbtilde q_1;\\\text{or }j=2,l=1}}(f_{j,l}(s))^\frac{1}{\widetilde q_1+1}ds\right)^{\widetilde q_1+1} \\[5pt]
        & \times\displaystyle\prod_{\substack{j=2,1<l\leq\dbtilde q_2;\\\text{or }2<j\leq\widetilde m,1\leq l\leq\dbtilde q_j}}\left(\int_{P_1^{-1}(s_1)\cap E}f_{j,l}(s)ds\right) \\[5pt]
        \leq & \displaystyle\prod_{j=1}^{\widetilde m}\displaystyle\prod_{l=1}^{\dbtilde q_j}\int_{P_1^{-1}(s_1)\cap E}f_{j,l}(s^{j,l})ds^{j,l} \\[5pt]
        = & \left(\int_{\Xi_1(s_1)}\displaystyle\prod_{\substack{j=1,1\leq l\leq\dbtilde q_1;\\\text{or }j=2,l=1}}f_{j,l}(s^{j,l})ds\right) \\[5pt]
        & \times\displaystyle\prod_{\substack{j=2,1<l\leq\dbtilde q_2;\\\text{or }2<j\leq\widetilde m,1\leq l\leq\dbtilde
        q_j}}\int_{P_1^{-1}(s_1)\cap E}f_{j,l}(s^{j,l})ds^{j,l}.
    \end{align*}
    
    Suppose for $1\leq j<\widetilde m$, we have
    \begin{align}\label{E:G_j}
        G_j(s_j)\leq&\int_{\Xi_j(s_j)}\displaystyle\prod_{\substack{1\leq r\leq j,1\leq l\leq\dbtilde q_r;\\\text{or }r=j+1,l=1}}f_{r,l}(s^{r,l})ds \\[5pt]
        & \times\displaystyle\prod_{\substack{r=j+1,1<l\leq\dbtilde q_{j+1};\\\text{or }j+1<r\leq\widetilde m,1\leq l\leq\dbtilde q_r}}\int_{P_j^{-1}(s_j)\cap E}f_{r,l}(s)ds.\notag
    \end{align}
    By the induction hypothesis and H\"older's inequality,
    \begin{align}\label{E:G j+1}
        G_{j+1}(s_{j+1})= & \left(\int_{P_j(P_{j+1}^{-1}(s_{j+1})\cap E)}G_j^\frac{1}{\widetilde Q_{j+1}}(s)ds\right)^{\widetilde Q_{j+1}} \\[5pt]\notag
        \leq & \left(\int_{P_j(P_{j+1}^{-1}(s_{j+1})\cap E)}\left(\int_{\Xi_j(s_j)}\displaystyle\prod_{\substack{1\leq r\leq j,1\leq l\leq\dbtilde q_r;\\\text{or }r=j+1,l=1}}f_{r,l}(s^{r,l})ds\right.\right. \\[5pt]\notag
        & \left.\left.\times\displaystyle\prod_{\substack{r=j+1,1<l\leq\dbtilde q_{j+1};\\\text{or }j+1<r\leq\widetilde m,1\leq l\leq\dbtilde q_r}}\int_{P_j^{-1}(s_j)\cap E}f_{r,l}(s)ds\right)^\frac{1}{\widetilde Q_{j+1}}ds_j\right)^{\widetilde Q_{j+1}} \\[5pt]\notag
        \leq & \int_{P_j(P_{j+1}^{-1}(s_{j+1})\cap E)}\int_{\Xi_j(s_j)}\displaystyle\prod_{\substack{1\leq r\leq j,1\leq l\leq\dbtilde q_r;\\\text{or }r=j,l=1}}f_{r,l}(s^{r,l})dsds_j \\[5pt]
        & \times\displaystyle\prod_{\substack{r=j,1<l\leq\dbtilde q_j;\\\text{or }j+1<r\leq\widetilde m,1\leq l\leq\dbtilde q_r}}\int_{P_j(P_{j+1}^{-1}(s_{j+1})\cap E)}\int_{P_j^{-1}(s_j)\cap E}f_{r,l}(s)dsds_j,\notag
    \end{align}
    If $j<\widetilde m-1$, then we deduce \eqref{E:G_j} for $j+1$ from Properties \ref{Prop:Xi}.\ref{Prop:nested Xi} and \ref{Prop:Xi}.\ref{Prop:E layer}. By induction, \eqref{E:G_j} holds for $1\leq j<\widetilde m$. We then obtain \eqref{E:am-gm} by taking $j=\widetilde m-1$ in \eqref{E:G j+1} and applying Properties \ref{Prop:Xi}.\ref{Prop:nested Xi} and \ref{Prop:Xi}.\ref{Prop:E layer} again.
\end{proof}

\section{Approximation by ellipsoids}\label{S:appr}
Recall the $f_{j,l}$ in Lemma \ref{L:diffeo} are of the form
\begin{equation}\label{E:f_jl integral}
    \int_{E^k}\det(\mathbf{u^\prime})\prod_{j=1}^kdu_i,
\end{equation}
where $E\subset\R^n$ has measure bounded below, $1\leq k\leq n$, and $\mathbf{u^\prime}$ is the $k\times k$-matrix whose $i$-th column consists of the first $k$ coordinates of $u_i$. In this subsection, we aim to bound \eqref{E:f_jl integral} by approximating $E$ by some balanced -- i.e., $x\in E$ implies $-x\in E$ -- and convex sets (or equivalently, ellipsoids centered at the origin).

\begin{lemma}\label{L:f_jl bound}
    Let $1\leq k\leq d,\eta>0$. Suppose $P:\R^d\longrightarrow\R^k$ is the orthogonal projection onto the first $k$ coordinates. Then there exists $c>0$ with the following property. For any Lebesgue measurable set $E\subset\R^d$ satisfying $0<|E|<\infty$, there exists a bounded balanced convex set $\scriptC\subset\R^d$ such that $|E\cap\scriptC|\sim|E|$, and
\begin{equation}\label{E:f_jl bound}
    \int_{E^k}|\det(\mathbf{u^\prime})|\displaystyle\prod_{i=1}^kdu_i\geq c\left(\left(\frac{|E|}{|\scriptC|}\right)^\eta|E|\right)^k|P(\scriptC)|,
\end{equation}
where $\mathbf{u^\prime}$ denotes the $k\times k$ matrix whose $i$-th column is $P(u_i)$ for $1\leq i\leq k$, and $c>0$ depends only on $d,\eta$.
\end{lemma}

If $E$ is balanced and convex, then the right-hand side of \eqref{E:f_jl bound} is simply $|E|^k|P(E)|$. Otherwise, we can approximate $E$ by balanced, convex sets. The idea descends from central sets on the real line in \cite{TaoWright}. Christ sharpens the result in \cite{christ2002mixed}. The higher-dimensional case is formulated in Lemma 5.2 of \cite{christ2011quasiextremals}.

\begin{lemma}[Lemma 5.2 of \cite{christ2011quasiextremals}]\label{L:appr}
For any $d\geq1,\eta>0$, there exists $c\in(0,1)$ with the following property. For any Lebesgue measurable set $S\subset\R^d$ satisfying $0<|S|<\infty$, there exists a bounded balanced convex set $\scriptC\subset\R^d$ such that $|S\cap\scriptC|\sim|S|$, and for any balanced convex set $\scriptC^\prime\subset\scriptC$
\begin{equation}\label{E:appr}
    |\scriptC^\prime|\leq\frac{1}{2}|\scriptC|\quad\implies\quad|S\cap(\scriptC\setminus\scriptC^\prime)|\geq c\left(\frac{|S|}{|\scriptC|}\right)^\eta|S|.
\end{equation}
\end{lemma}

The following lemma is a generalization of Lemma 6.1 of \cite{christ2011quasiextremals}.

\begin{lemma}\label{L:det}
Let $\scriptC\subset\R^d$ be a bounded, balanced, convex set. Let $\mu$ be a positive, finite measure supported on $\scriptC$. Let $1\leq k\leq d$. Suppose $P:\R^d\longrightarrow\R^k$ is the orthogonal projection onto the first $k$ coordinates. Let $\delta,\lambda>0$. Suppose that for any balanced convex subset $\scriptC^\prime\subset\scriptC$ satisfying $|\scriptC^\prime|\leq\delta|\scriptC|$, one has $\mu(\scriptC\backslash\scriptC^\prime)\geq\lambda$. Then
\begin{equation}\label{E:det}
    \int_{\scriptC^k}|\det(\mathbf{u^\prime})|\displaystyle\prod_{i=1}^kd\mu(u_i)\geq c\delta^k\lambda^k|P(\scriptC)|,
\end{equation}
where $\mathbf{u^\prime}$ denotes the $k\times k$ matrix whose $i$-th column is $P(u_i)$ for $1\leq i\leq k$, and $c>0$ depends only on $d$.
\end{lemma}

\begin{proof}
    By a change of variables in $\R^d$, we may assume $P(\scriptC)$, which is a bounded, balanced, convex set in $\R^k$, is the unit ball centered at the origin. The factor $|P(\scriptC)|$ on the right hand side of \eqref{E:det} results from the transformation law for $|\det(\mathbf{u^\prime})|$.

    Note that $|\det(\mathbf{u^\prime})|=\prod_{i=1}^k\dist(P(u_i),V_{i-1})$, where $V_0=\{0\}$, and $V_i$ is the linear span of $\{P(u_1),\dots,P(u_i)\}$ for $i=1,\dots,k-1$. Fix $u_1,\dots,u_{k-1}$. Let
    \begin{equation}
        \scriptC_1\coloneqq\{u_k\in\R^d:\dist(P(u_k),V_{k-1})<c\delta\},
    \end{equation}
    where $c$ is a constant chosen to be sufficiently small so that $|\scriptC_1|\leq\delta|\scriptC|$. We observe that $\scriptC_1$ is convex and balanced. Then
    \begin{equation}
        \int_{\scriptC}\dist(u_k,V_{k-1})d\mu(u_k)\geq c\delta\mu(\scriptC\setminus\scriptC_1)\geq c\delta\lambda.
    \end{equation}
    Fix $u_1,\dots,u_{k-2}$. Now let
    \begin{equation}
        \scriptC_2\coloneqq\{u_k\in\R^d:\dist(P(u_k),V_{k-2})<c\delta\},
    \end{equation}
    for another sufficiently small constant $c$. By the same reasoning, one has
    \begin{equation}
        \int_{\scriptC}\dist(u_{k-1},V_{k-2})d\mu(u_{k-1})\geq c\delta\mu(\scriptC\setminus\scriptC_2)\geq c\delta\lambda.
    \end{equation}
    Repeating this argument $k-2$ more times gives the desired bound.
\end{proof}

\begin{proof}[Proof of Lemma \ref{L:f_jl bound}]
    By Lemma \ref{L:appr}, one can find a constant $c>0$ and a bounded balanced convex set $\scriptC\subset\R^d$ such that $|E\cap\scriptC|\sim|E|$, and for any balanced convex set $\scriptC^\prime\subset\scriptC$
    \begin{equation}\label{EC:appr}
        |\scriptC^\prime|\leq\frac{1}{2}|\scriptC|\quad\implies\quad|E\cap(\scriptC\setminus\scriptC^\prime)|\geq c\left(\frac{|E|}{|\scriptC|}\right)^\eta|E|.
    \end{equation}
    Letting $\lambda=c(|E|/|\scriptC|)^\eta|E|$, $\delta=\frac{1}{2}$, and $\mu$ be the Lebesgue measure restricted to $E\cap\scriptC$, one may apply Lemma \ref{L:det} to obtain
    \begin{equation}
        \begin{array}{rcl}
            \int_{E^k}|\det(\mathbf{u^\prime})|\displaystyle\prod_{i=1}^kdu_i & \geq & \int_{\scriptC^k}|\det(\mathbf{u^\prime})|\displaystyle\prod_{i=1}^kd\mu(u_i) \\
            & \geq & c_1c^k2^{-k}\left(\left(\frac{|E|}{|\scriptC|}\right)^\eta|E|\right)^k|P(\scriptC)|.
        \end{array}
    \end{equation}
\end{proof}

\section{Semiregularity, regularity, and quasiextremality}\label{S:srq}
In this section, we make the notions of ``semiregularity" and ``regularity" rigorous. We will once again make use of the notations from \eqref{E:pi j twiddle} and \eqref{E:T}. 

\begin{definition}[$\varepsilon$-semiregular sets]\label{D:semi regular}
    Let $\varepsilon>0$. A bounded subset $\Omega\subset\HH^n$ of positive measure is said to be \textit{$\varepsilon$-semiregular with respect to $\pi$} if for each $1\leq j\leq\widetilde m$, the fibers of $\widetilde \pi_j$ in $\Omega$ have comparable sizes (up to $\varepsilon$). To be specific,
    \begin{equation}
        \forall z\in\Omega,j=1,\dots,\widetilde m\quad|\widetilde T_j(z)|\leq2\varepsilon^{-1}\frac{|\Omega|}{|\widetilde\pi_j(\Omega)|}.
    \end{equation}
    $\Omega$ is said to be \textit{$\varepsilon$-semiregular with respect to $\pi_*$} if
    \begin{equation}
        \forall z\in\Omega,j=1,\dots,\widetilde m\quad|\widetilde T_j^*(z)|\leq2\varepsilon^{-1}\frac{|\Omega|}{|\widetilde\pi_j^*(\Omega)|}.
    \end{equation}
    When no ambiguity occurs, we simply say $\Omega$ is \textit{$\varepsilon$-semiregular}.
\end{definition}

\begin{lemma}\label{L:semiregular refinement}
    Let $\sigma,\varepsilon>0$. Every $\sigma$-refinement (see Definition \ref{D:refinement}) $\Omega^\prime$ of an $\varepsilon$-semiregular set $\Omega$ with respect to $\pi$ ($\pi_*$, resp.) is $(C\varepsilon\sigma)$-semiregular with respect to $\pi$ ($\pi_*$, resp.) for some large constant $C>0$ depending only on $n$.
\end{lemma}

\begin{proof}
    Let $\Omega\subset\HH^n$ be $\varepsilon$-semiregular with respect to $\pi$ (the argument is almost the same for remiregularity with respect to $\pi_*$), and let $\Omega^\prime\subset\Omega$ be a $\sigma$-refinement. Then there exists a small constant $c>0$ such that for all $z\in\Omega^\prime,1\leq j\leq\widetilde m$,
    \begin{equation}
        |\widetilde T_j^{\Omega^\prime}(z)|\leq|\widetilde T_j^{\Omega}(z)|\leq2\varepsilon^{-1}\frac{|\Omega|}{|\widetilde\pi_j(\Omega)|}\leq 2c\sigma^{-1}\varepsilon^{-1}\frac{|\Omega^\prime|}{|\widetilde\pi_j(\Omega^\prime)|}.
    \end{equation}
    Then $\Omega^\prime$ is $(c^{-1}\sigma\varepsilon)$-semiregular with respect to $\pi$.
\end{proof}

\begin{definition}[$\varepsilon$-quasiextremality]\label{D:quasiextremal}
We say a bounded, measurable subset $\Omega\subset\HH^n$ is \textit{$\varepsilon$-quasiextremal} if
\begin{equation}
    |\Omega|\geq\varepsilon\displaystyle\prod_{j=1}^M|\pi_j(\Omega)|^\frac{1}{p_j},
\end{equation}
and we say the functions $\{f_j\}_{j=1}^M$ are \textit{$\varepsilon$-quasiextremal} if
\begin{equation}
    \scriptM(f_1,\dots,f_M)\geq\varepsilon\displaystyle\prod_{j=1}^M\|f_j\|_{p_j}.
\end{equation}
\end{definition}

\begin{remark}
    Note that we don't \textit{a priori} assume \eqref{E:base case} or the restricted weak-type inequality
    \begin{equation}
        |\Omega|\lesssim\displaystyle\prod_{j=1}^M|\pi_j(\Omega)|^\frac{1}{p_j}
    \end{equation}
    to call $\{f_j\}$ or $\Omega$ $\varepsilon$-quasiextremal, respectively.
\end{remark}

\begin{lemma}\label{L:quasiextremal refinement}
    Let $\sigma,\varepsilon>0$. Every $\sigma$-refinement $\Omega^\prime$ of an $\varepsilon$-quasiextremal set $\Omega$ is $C\varepsilon\sigma$-quasiextremal for some large constant $C>0$.
\end{lemma}

\begin{definition}[$\varepsilon$-regularity]\label{D:epsilon regularity}
    Let $\varepsilon>0$, and $\Omega\subset\HH^n$. $\Omega$ is said to be \textit{$\varepsilon$-regular} if it is $\varepsilon$-semiregular with respect to both $\pi,\pi_*$ and $\varepsilon$-quasiextremal.
\end{definition}

\begin{lemma}\label{L:regular refinement}
    Let $\sigma,\varepsilon>0$. Every $\sigma$-refinement $\Omega^\prime$ of an $\varepsilon$-regular set $\Omega$ is $C\varepsilon\sigma$-regular for some large constant $C>0$.
\end{lemma}


\begin{lemma}\label{L:semiregularity to regularity}
    Let $0<\sigma\leq\varepsilon\leq1$, let $\Omega\subset\HH^n$ be $\varepsilon$-semiregular with respect to $\pi$, and let $\Omega^\prime\subset\Omega$ be an $\varepsilon^C$-refinement that is $\varepsilon$-semiregular with respect to $\pi_*$. Suppose $\Omega$ is $\varepsilon$-quasiextremal, and
    \begin{equation}
        \Omega^{\prime\prime}=\{z\in\Omega^\prime:\forall1\leq j\leq\widetilde m\quad|(T_j^*)^{\Omega^\prime}(z)|\leq\varepsilon^{-C}\frac{|\Omega^\prime|}{|\widetilde\pi_j(\Omega^\prime)|}\}
    \end{equation}
    satisfies $|\Omega^\prime|\sim\sigma|\Omega|$. Then $\Omega^\prime$ is $\sigma^C$-regular. 
\end{lemma}

\section{The restricted weak-type inequality for semiregular sets}\label{S:rwt semireg}
Now we have all the ingredients in place to prove \eqref{E:rwt} for semiregular sets.

\begin{proposition}\label{P:semi reg rwt}
    Let $\Omega\subset\HH^n$ be a $\sigma$-semiregular set with respect to $\pi$ of for some $\sigma\in(0,1)$, and let $p_j$ be the exponenets satisfying the assumptions in Theorem \ref{T:base case}. Then,
    \begin{equation}\label{E:semi reg rwt avg}
        |\Omega|^N\gtrsim\sigma^C\left(\displaystyle\prod_{j=1}^{\widetilde m}\beta_j^{\widetilde q_j}\right)\left(\displaystyle\prod_{j=m+1}^M\alpha_j^{q_j}\right)\gtrsim\sigma^C\displaystyle\prod_{j=1}^M\alpha_j^{q_j},
    \end{equation}
    where
    \[
    \alpha_j=\frac{|\Omega|}{|\pi_j(\Omega)|},\quad1\leq j\leq M;\qquad\beta_j=\frac{|\Omega|}{|\widetilde\pi_j(\Omega)|},\quad1\leq j\leq\widetilde m.
    \]
    $N,q_j,\widetilde q_j,\widetilde\pi_j$ are as in Section \ref{S:more notations}. In particular,
    \begin{equation}\label{E:semi reg rwt}
        |\Omega|\lesssim\sigma^{-C}\displaystyle\prod_{j=1}^M|\pi_j(\Omega)|^\frac{1}{p_j}.
    \end{equation}

    Let $L_j$ be as in \eqref{E:pi_j}. Suppose $\Omega$ is in addition $\varepsilon$-quasiextremal for some $0<\varepsilon\leq\sigma$. Then for any $A\geq3$, there exists a refined flow scheme $(z_0,S_1,\scriptF_l,\scriptG_l)$ of $(\Omega,\pi,A)$ satisfying the following property. For each $s\in S_1$, there exists an ellipsoid $\scriptC(s)\subset\R^n$ centered at the origin adapted to $\{L_j\}_{j=m+1}^M$ as in Definition \ref{D:partition} such that
    \begin{equation}\label{E:small r(s)}
        \scriptF_1(s)\subset\scriptC(s),\qquad|\scriptC(s)|\lesssim\varepsilon^{-C}|\scriptF_1(s)|.
    \end{equation}

    The implicit constants and $C$ depend only on $n,\pi_j,p_j$.
\end{proposition}

\begin{proof}
    Recall the notations from Section \ref{S:more notations} and \eqref{E:qj double twiddle}. Write $\widetilde k=(0,\widetilde k_1,\dots,\widetilde k_{\widetilde m}),\dbtilde q=(\dbtilde q_1,\dots,\dbtilde q_{\widetilde m})$. Let $(z_0,S_1,\scriptF_l,\scriptG_l)$ be a refined flow scheme of $(\Omega,\pi,A)$ with $A\geq3$ as in Lemma \ref{L:refinement}, and write $S=S_1,\scriptF=\scriptF_1,\scriptG=\scriptG_1$. \eqref{E:refinement 2}-\eqref{E:refinement 4} and $\sigma$-semiregularity imply
    \begin{gather}
        \label{E:SG alpha bound}\frac{|S|}{|L_j(S)|},\frac{|\scriptG|}{|L_j(\scriptG)|}\gtrsim\alpha_j,\qquad1\leq j\leq m; \\
        \label{E:F bound}\frac{|\scriptF|}{|L_j(\scriptF)|}\gtrsim\alpha_j,\qquad m<j\leq M; \\
        \label{E:SG beta bound}\beta_j\lesssim|P_j^{-1}(s)\cap S|,|P_j^{-1}(v)\cap\scriptG|\lesssim\sigma^{-1}\beta_j,\qquad s\in P_j(S),v\in P_j(\scriptG),1\leq j\leq\widetilde m.
    \end{gather}
    
    We first prove an intermediate result:
    \begin{equation}\label{E:intermediate semi reg rwt avg}
        |\Omega|^N\gtrsim\sigma^C\displaystyle\prod_{j=1}^m\beta_j^{\widetilde q_j}\displaystyle\prod_{j=m+1}^M\alpha_j^{q_j}.
    \end{equation}    
    Recall \eqref{E:Omega flat} and Lemma \ref{L:diffeo}. Then
    \[
    |\Omega|^N\gtrsim\beta_{\widetilde m}^N\int_\Xi\displaystyle\prod_{j=1}^{\widetilde m}\displaystyle\prod_{l=1}^{\dbtilde q_j}f_{j,l}(s^{j,l})ds,
    \]
    where
    \[
    f_{j,l}(s)\coloneqq\int_{(\scriptF(s))^{\widetilde k_j}}|\det(\mathbf u^{j,l})|\prod_{a=1}^{\widetilde k_j}du^{j,l,a},\qquad s\in S.
    \]
    
    For the time being, we consider the diagonal case, i.e., $s=s^{j,l}$ for all $j,l$, and aim to bound $\prod_{j,l}f_{j,l}(s)$. Let $\eta>0$ be a small constant depending only on $n$ to be determined later. By Lemma \ref{L:f_jl bound}, there exists a bounded balanced convex set $\scriptC(s)\subset\R^n$ such that $|\scriptC(s)|\sim 2^{r(s)}|\scriptF(s)|$ for some $r(s)\geq0$, $|\scriptF(s)\cap\scriptC(s)|\sim|\scriptF(s)|$, and
    \begin{equation}
        f_{j,l}(s)\gtrsim
        \begin{cases}
            2^{-r(s)\widetilde k_j\eta}|\scriptF(s)\cap\scriptC(s)|^{\widetilde k_j}|P_j^\perp(\scriptF(s)\cap\scriptC(s))|, & 1\leq j<\widetilde m, \\
            2^{r(s)(1-n\eta)}|\scriptF(s)\cap\scriptC(s)|^n|\scriptF(s)\cap\scriptC(s)|, & j=\widetilde m,
        \end{cases}
    \end{equation}
    where we recall $\widetilde k_{\widetilde m}=n$ from Section \ref{S:more notations} and that $P_j^\perp:\R^n\longrightarrow\R^{\widetilde k_j}$ denotes the orthogonal projection onto the first $\widetilde k_j$ coordinates. $\widetilde k_{\widetilde m}=n$, \eqref{E:qj double twiddle}, and \eqref{A''} yield
    \begin{equation}
        \displaystyle\prod_{j=1}^{\widetilde m}\displaystyle\prod_{l=1}^{\dbtilde q_j}f_{j,l}(s)\gtrsim2^{r(s)(\dbtilde q_{\widetilde m}-N\eta)}\displaystyle\prod_{j=1}^{\widetilde m}|P_j^\perp(\scriptF(s)\cap\scriptC(s))|^{\widetilde q_j}.
    \end{equation}
    Recall \eqref{B''}, \eqref{E:order of Q_i}, \eqref{E:qj twiddle}, and \eqref{E:qj double twiddle}. Since $\dbtilde q_{\widetilde m}\geq1$, we can consider $\eta=(\dbtilde q_{\widetilde m}-0.5)/N>0$. (This is the only place where we require the strict inequality \eqref{B twiddle} instead of \eqref{B}.) We now call $2r(s)(\dbtilde q_{\widetilde m}-N\eta)$ our new $r(s)$.
    
    Applying Corollary \ref{C:Lj Pj perp finner} with $\omega=\scriptF(s)\cap\scriptC(s)$ yields
    \begin{equation}\label{E:F(s) finner}
        \displaystyle\prod_{j=1}^{\widetilde m}(|P_j^\perp(\scriptF(s)\cap\scriptC(s))|)^{\widetilde q_j}\geq\displaystyle\prod_{j=m+1}^M\left(\frac{|\scriptF(s)\cap\scriptC(s)|}{|L_j(\scriptF(s)\cap\scriptC(s))|}\right)^{q_j}.
    \end{equation}
    Let
    \begin{equation}\label{E:r twiddle}
        2^{\widetilde r(s)}\coloneqq\frac{C\displaystyle\prod_{j=1}^{\widetilde m}|P_j^\perp(\scriptF(s)\cap\scriptC(s))|^{\widetilde q_j}}{\displaystyle\prod_{j=m+1}^M\alpha_j^{q_j}}.
    \end{equation}
    $|\scriptF(s)\cap\scriptC(s)|\sim|\scriptF(s)|$, \eqref{E:F bound}, and \eqref{E:F(s) finner} imply $\widetilde r(s)\geq0$ for sufficiently large constant $C$. Thus,
    \begin{equation}
        \displaystyle\prod_{j=1}^{\widetilde m}\displaystyle\prod_{l=1}^{\dbtilde q_j}f_{j,l}(s)\gtrsim2^{r(s)+\widetilde r(s)}\displaystyle\prod_{j=m+1}^M\alpha_j^{q_j}.
    \end{equation}
    
    By Lemma \ref{L:am-gm},
    \begin{equation}\label{E:avg C(s) size}
        |\Omega|^N\gtrsim\beta_{\widetilde m}^N\displaystyle\prod_{j=m+1}^M\alpha_j^{q_j}G_{\widetilde m},
    \end{equation}
    where $G_j$ is associated with $(\widetilde k,\dbtilde q,S,())$ and $g(s)=2^{r(s)+\widetilde r(s)}\geq1$. \eqref{E:SG beta bound}, Lemma \ref{L:G m twiddle lower bound}, and \eqref{E:qj double twiddle} yield \eqref{E:intermediate semi reg rwt avg}. Then \eqref{E:semi reg rwt avg} follows from \eqref{E:SG beta bound}, and Corollary \ref{C:specialize gen finner}.

    Assume $\Omega$ is in addition $\varepsilon$-quasiextremal for some $0<\varepsilon\leq\sigma$. Then the chain of inequality
    \begin{equation}
        |\Omega|^N\gtrsim\beta_{\widetilde m}^N\displaystyle\prod_{j=m+1}^M\alpha_j^{q_j}G_{\widetilde m}\gtrsim\sigma^C\displaystyle\prod_{j=1}^M\alpha_j^{q_j}
    \end{equation}
    can be reversed as follows:
    \begin{equation}
        |\Omega|^N\lesssim\varepsilon^{-C_1}\beta_{\widetilde m}^N\displaystyle\prod_{j=m+1}^M\alpha_j^{q_j}G_{\widetilde m}\lesssim\varepsilon^{-C_2}\displaystyle\prod_{j=1}^M\alpha_j^{q_j}
    \end{equation}
    for sufficiently large constants $C_1,C_2$ depending only on $p_j$.
    Applying Lemma \ref{L:G m twiddle upper bound} with $E=S,\tau=()$, we have $2^{r(s)+\widetilde r(s)}\leq C\varepsilon^{-C}$ for all $s$ in a possibly smaller $S=S_1$, but $(z_0,S_1,\scriptF_l,\scriptG_l)$ remains a refined flow scheme. Since $r(s),\widetilde r(s)\geq0$, we have $2^{r(s)}\leq C\varepsilon^{-C}$ and $2^{\widetilde r(s)}\leq C\varepsilon^{-C}$. Unpacking the definitions of $r(s),\widetilde r(s)$ and invoking convexity of $\scriptC(s)$, $|\scriptF(s)\cap\scriptC(s)|\sim|\scriptF(s)|$, and \eqref{E:F bound}, we obtain
    \begin{align*}
        \displaystyle\prod_{j=1}^{\widetilde m}|P_j^\perp(\scriptC(s))|^{\widetilde q_j}\lesssim    &   \varepsilon^{-C}\displaystyle\prod_{j=1}^{\widetilde m}|P_j^\perp(\scriptC(s)\cap\scriptF(s))|^{\widetilde q_j} \\
        \lesssim     &  \varepsilon^{-C}\displaystyle\prod_{j=m+1}^M\left(\frac{|\scriptF(s)\cap\scriptC(s)|}{|L_j(\scriptF(s)\cap\scriptC(s))|}\right)^{q_j} \\
        \lesssim    &   \varepsilon^{-C}\displaystyle\prod_{j=m+1}^M\left(\frac{|\scriptC(s)|}{|L_j(\scriptC(s))|}\right)^{q_j},
    \end{align*}
    where $C$ may vary line by line, but only depends on $n,\pi_j,p_j$. By Corollary \ref{C:Lj Pj perp finner quasiextremal}, we can choose an ellipsoid $\scriptC(s)\subset\R^n$ centered at the origin adapted to $\{L_j\}_{j=m+1}^M$ satisfying \eqref{E:small r(s)}.
\end{proof}

\begin{corollary}\label{C:C appr S_1}
    Let $\varepsilon>0,A\geq2$. Suppose $\Omega$ is $\varepsilon$-regular. Let $L_j$ be as in \eqref{E:pi_j}. Then there exist a large constant $C$ depending only on $n,\pi_j,p_j$, a refined flow scheme $(z_0,S_1,\scriptF_l,\scriptG_l)$ of $(\Omega,\pi,A)$ and an ellipsoid $\scriptC\subset\R^n$ centered at the origin adapted to $\{L_j\}_{j=1}^m$ as in Definition \ref{D:partition} such that
    \begin{equation}
        S_1\subset\scriptC,\qquad|\scriptC|\lesssim\varepsilon^{-C}|S_1|.
    \end{equation}
\end{corollary}

\begin{proof}
    Let $\varepsilon>0$. Suppose $\Omega\subset\HH^n$ is $\varepsilon$-regular, that is, $\Omega$ is $\varepsilon$-semiregular with respect to both $\pi,\pi_*$ and $\varepsilon$-quasiextremal. By Proposition \ref{P:semi reg rwt}, there exists a large constant depending only on $n,\pi_j,p_j$ and a refined flow scheme $(z_*,S_1^*,\scriptF_l^*,\scriptG_l^*)$ of $(\Omega,\pi_*,A+1)$ such that, for each $s\in S_1^*$, there exists an ellipsoid $\scriptC(s)\subset\R^n$ centered at the origin adapted to $\{L_j\}_{j=1}^m$ as in Definition \ref{D:partition} such that
    \begin{equation}
        \scriptF_1^*(s)\subset\scriptC(s),\qquad|\scriptC(s)|\lesssim\varepsilon^{-C}|\scriptF_1^*(s)|.
    \end{equation}
    Fix $s_0\in S_1^*$. Then $(e^{s_0\cdot\mathbf Y}(z_*),\scriptF_1^*(s_0),\scriptG_l^*(s_0,\cdot),\scriptF_{l+1}^*(s_0,\cdot))$ and $\scriptC(s_0)$ are as desired.
\end{proof}

\section{Quasiextremal regular sets}\label{S:regular set paraball}
Our next goal is to remove the semiregularity assumption in Proposition \ref{P:semi reg rwt}. This will be done by partitioning $\Omega$ into semiregular sets $\{\Omega_\ell\}_{\ell\in\Z^{\widetilde m}}$ so that the $\widetilde\pi_j$ fibers of $\Omega_\ell$ are comparable to $2^{\ell_j}$. Then, we will apply Proposition \ref{P:semi reg rwt} to each $\Omega_\ell$ in hope of the following:
\begin{equation}
    |\Omega|=\sum_\ell|\Omega_\ell|\lesssim\sum_\ell\displaystyle\prod_{j=1}^M|\pi_j(\Omega_\ell)|^\frac{1}{p_j}\overset{\textcolor{red}{?}}{\lesssim}\displaystyle\prod_{j=1}^M|\pi_j(\Omega)|^\frac{1}{p_j}.
\end{equation}
The main obstruction of the last inequality is that $\pi_j(\Omega_\ell)\cap\pi_j(\Omega_{\ell^\prime})$ may be large for $\ell\neq\ell^\prime$. In a worst-case scenario, where $\Omega_\ell$ are (quasi-)extremals of the restricted weak-type inequality (thus the first inequality is reversed), we will see that the quasiextremals can be approximated by ``paraballs". Most of the paraballs look drastically different, so their $\pi_j$ images are almost disjoint.

In this section, we will approximate regular quasiextremals adapting Christ's argument (see Section 7 of \cite{christ2011quasiextremals}).

\begin{definition}[Paraballs]\label{D:paraball}
    Let $z=(x,y,t)\in\HH^n,\rho>0$, let $\{l_j\}_{j=1}^d$ be coordinate projections on $\R^n$. Suppose $r,r^*\in(\R_{>0})^n$ satisfy $\rho=r_ir_i^*$ for $i=1,\dots,n$. A \textit{paraball $\scriptB(z,\mathbf e,r,r^*,\rho,\{l_j\}_{j=1}^d)$} is defined as
    \[
    \{(x+s,y+u,t+\frac{1}{2}(-s\cdot y+x\cdot u)+\tau):s\in\scriptC,u\in\scriptC_*,|\tau|<\rho\},
    \]
    where $\scriptC,\scriptC_*\subset\R^n$ are the ellipsoids centered at the origin adapted to $\{l_j\}_{j=1}^d$ as in Definition \ref{D:partition} whose principal axes are along $e_i$ and of length $r_j,r_j^*$, respectively. When no ambiguity occurs, we may simply denote the paraball by $\scriptB$.
\end{definition}

\begin{props}[Properties of paraballs]\label{Prop:paraball}\hfill
    \begin{enumerate}
        \item\label{Prop:left-inv para} $\scriptB$ are left-invariant under the Heisenberg group operation, i.e.,
        \[
        \forall z_1,z_2\in\HH^n\qquad z_1\bullet\scriptB(z_2,\mathbf e,r,r^*,\rho,\{l_j\}_{j=1}^d)=\scriptB(z_1\bullet z_2,\mathbf e,r,r^*,\rho,\{l_j\}_{j=1}^d).
        \]

        \item\label{Prop:quasi ext para} Every paraball associated with $\{L_j\}_{j=1}^M$ is $C$-regular for some $C>0$ depending only on $n,\pi_j,p_j$.

        \item\label{Prop:scaling-inv para} Every paraball associated with $\{L_j\}_{j=1}^M$ is scaling-invariant in the sense that
        \begin{gather}
            |\widetilde\scriptB|=a^{n+1}|\scriptB|; \\
            |\pi_j(\widetilde\scriptB)|=
            \begin{cases}
                (\prod_{i:e_i\not\in V_j}\lambda_i^*)a^{\dim V_j+1}|\pi_j(\scriptB)|, & 1\leq j\leq m, \\
               (\prod_{i:e_i\not\in V_j}\lambda_i)a^{\dim V_j+1}|\pi_j(\scriptB)|, & m<j\leq M;
            \end{cases} \\
            |\widetilde\pi_j(\widetilde\scriptB)|=(\prod_{i=1}^{\widetilde k_j}\lambda_i^*)a^{n-\widetilde k_j+1}|\pi_j(\scriptB)|,\qquad1\leq j\leq\widetilde m; \\
            |\widetilde\pi_j^*(\widetilde\scriptB)|=(\prod_{i=1}^{\widetilde k_j}\lambda_i)a^{n-\widetilde k_j+1}|\pi_j^*(\scriptB)|,\qquad1\leq j\leq\widetilde m,
        \end{gather}
        for any $\lambda,\lambda^*\in\R_{>0}$ satisfying $\lambda_i\lambda_i^*=a$ for $1\leq i\leq n$, and
        \[
        \scriptB=\scriptB(z,r,r^*,\rho,\{L_j\}_{j=1}^M),\qquad\widetilde\scriptB=\scriptB(z,\lambda\otimes r,\lambda^*\otimes r^*,a\rho,\{L_j\}_{j=1}^M),
        \]
        with $\lambda\otimes r=(\lambda_1r_1,\dots,\lambda_nr_n)$.

        \item\label{Prop:paraball little interaction} Let $\varepsilon\in(0,1)$, and let $\scriptB,\scriptB^\prime$ be two paraballs associated with dual ellipsoids $(\scriptC,\scriptC_*)$ and $(\scriptC^\prime,\scriptC_*^\prime)$, respectively. Suppose $l$ is an orthogonal projection onto a subspace spanned by some of the $e_i$ satisfying
        \begin{equation}\label{E:paraball little interaction}
            \max\{\frac{|l(\scriptC)|}{|l(\scriptC^\prime)|},\frac{|l(\scriptC^\prime)|}{|l(\scriptC)|},\frac{|l(\scriptC_*)|}{|l(\scriptC_*^\prime)|},\frac{|l(\scriptC_*^\prime)|}{|l(\scriptC_*)|}\}\gtrsim\varepsilon^{-A}.
        \end{equation}
        Then there exists $C_A$ depending on $n,\pi_j,p_j,A$ with $\lim_{A\rightarrow\infty}C_A=\infty$ and
        \begin{gather}
            |\pi_j(\scriptB)\cap\pi_j(\scriptB^\prime)|\lesssim\varepsilon^{C_A}\max\{|\pi_j(\scriptB)|,|\pi_j(\scriptB^\prime)|\},\qquad1\leq j\leq M; \\
            |\widetilde\pi_j(\scriptB)\widetilde\cap\pi_j(\scriptB^\prime)|\lesssim\varepsilon^{C_A}\max\{|\widetilde\pi_j(\scriptB)|,|\widetilde\pi_j(\scriptB^\prime)|\},\qquad1\leq j\leq\widetilde m; \\
            |\widetilde\pi_j^*(\scriptB)\widetilde\cap\pi_j^*(\scriptB^\prime)|\lesssim\varepsilon^{C_A}\max\{|\widetilde\pi_j^*(\scriptB)|,|\widetilde\pi_j^*(\scriptB^\prime)|\},\qquad1\leq j\leq\widetilde m.
        \end{gather}
    \end{enumerate}
\end{props}

\begin{proof}
    Property \eqref{Prop:left-inv para} and regularity can be seen by straightforward computation.
    
    Elementary calculus shows $|\scriptB|\sim\rho^{n+1}$ and
    \[
    |\pi_j(\scriptB)|\sim
    \begin{cases}
        |L_j(\scriptC)||\scriptC_*|\rho\sim(\prod_{i:e_i\not\in V_j}r_i)\rho^{\dim V_j+1}, & 1\leq j\leq m, \\
        |\scriptC||L_j(\scriptC_*)|\rho\sim(\prod_{i:e_i\not\in V_j}r_i^*)\rho^{\dim V_j+1}, & m<j\leq M.
    \end{cases}
    \]
    By construction of $P_j$ from Section \ref{S:more notations}, the images of $P_j$ are also spanned by some of the $e_i$.  Thus, we can similarly obtain the sizes of $\widetilde\pi_j(\scriptB),\widetilde\pi_j^*(\scriptB)$. It is straightforward to check Property \eqref{Prop:scaling-inv para} from here. Recalling conditions \eqref{A} and \eqref{C} with $V=\R^n$, we have $\prod_{j=1}^M|\pi_j(\scriptB)|^\frac{1}{p_j}\sim\rho^{n+1}$, thus Property \eqref{Prop:quasi ext para}.

    Let $\rho^n=|\scriptC||\scriptC_*|,(\rho^\prime)^n=|\scriptC^\prime||\scriptC_*^\prime|$. \eqref{E:paraball little interaction} and duality of $(\scriptC,\scriptC_*),(\scriptC^\prime,\scriptC_*^\prime)$ imply $|\scriptC\cap\scriptC^\prime|\lesssim\max\{|\scriptC|,|\scriptC^\prime|\},|\scriptC_*\cap\scriptC_*^\prime|\lesssim\max\{|\scriptC_*|,|\scriptC_*^\prime|\}$. Then Property \eqref{Prop:paraball little interaction} can be deduced from the size of $\pi_j(\scriptB)$ and the following:
    \[
    |\pi_j(\scriptB)\cap\pi_j(\scriptB^\prime)|\leq
    \begin{cases}
        |L_j(\scriptC\cap\scriptC^\prime)||\scriptC_*\cap\scriptC_*^\prime|\min\{\rho,\rho^\prime\}, & 1\leq j\leq m, \\
        |L_j(\scriptC_*\cap\scriptC_*^\prime)||\scriptC\cap\scriptC^\prime|\min\{\rho,\rho^\prime\}, & m<j\leq M.
    \end{cases}
    \]
    One has a similar estimate for $\widetilde\pi_j,\widetilde\pi_j^*$.
\end{proof}

The main result of the section is as follows:

\begin{proposition}\label{P:regular paraball}
    Let $\varepsilon>0$. Suppose $\Omega\subset\HH^n$ is $\varepsilon$-regular, then there exist a paraball $\scriptB$ associated with $\{L_j\}_{j=1}^M$ and $C>0$ depending on $\varepsilon,n,\pi_j,p_j$ such that
    \begin{equation}\label{E:paraball}
        \varepsilon^C|\scriptB|\lesssim|\Omega|\lesssim\varepsilon^{-C}|\Omega\cap\scriptB|,
    \end{equation}
    and
    \begin{align}
        \notag|\pi_j(\scriptB)|\leq|\pi_j(\Omega)|&\qquad j=1,\dots,M;\\
        \label{E:paraball proj est}|\widetilde\pi_j(\scriptB)|\leq|\widetilde\pi_j(\Omega)|&\qquad j=1,\dots,\widetilde m;\\
        \notag|\widetilde\pi_j^*(\scriptB)|\leq|\widetilde\pi_j^*(\Omega)|&\qquad j=1,\dots,\widetilde m.
    \end{align}
\end{proposition}

To prove Proposition \ref{P:regular paraball}, it suffices to prove the following proposition:

\begin{proposition}\label{P:regular weak paraball}
    Let $\varepsilon>0$. Suppose $\Omega\subset\HH^n$ is $\varepsilon$-regular. Then there exists a paraball $\scriptB$ associated with $\{L_j\}_{j=1}^M$ such that \eqref{E:paraball} holds true, and
    \begin{align}
        \notag|\pi_j(\scriptB)|\lesssim\varepsilon^{-C}|\pi_j(\Omega)|&\qquad j=1,\dots,M;\\
        \label{E:paraball weak proj est}|\widetilde\pi_j(\scriptB)|\lesssim\varepsilon^{-C}|\widetilde\pi_j(\Omega)|&\qquad j=1,\dots,\widetilde m;\\
        \notag|\widetilde\pi_j^*(\scriptB)|\lesssim\varepsilon^{-C}|\widetilde\pi_j^*(\Omega)|&\qquad j=1,\dots,\widetilde m.
    \end{align}
    $0<C<\infty$ is a constant depending only on $n,\pi_j,p_j$.
\end{proposition}

To see this, we need the following lemma adapted from Lemma 7.2 of \cite{christ2011quasiextremals}:

\begin{lemma}\label{L:covering}
    There exist $0<C,A,A_j<\infty$ for each $1\leq j\leq M$ such that for any $\delta\in(0,1]$ and any paraball $\scriptB(z,\mathbf e,r,r^*,\rho,\{L_j\}_{j=1}^M)$, there exists a family of paraballs $\{\scriptB_k\}_{k\in\scriptK}$ associated with $\{L_j\}_{j=1}^M$ satisfying
    \begin{gather*}
        \scriptB\subset\cup_{k\in\scriptK}\scriptB_k, \\
        \#\scriptK\leq C\delta^{-C}, \\
        \forall1\leq j\leq M\quad|\pi_j(\scriptB_k)|=\delta^{A_j}|\pi_j(\scriptB)|, \\
        \forall1\leq j\leq\widetilde m\quad|\widetilde\pi_j(\scriptB_k)|=\delta^{\widetilde A_j^*}|\widetilde\pi_j(\scriptB)|, \\
        \forall1\leq j\leq\widetilde m\quad|\widetilde\pi_j^*(\scriptB_k)|=\delta^{\widetilde A_j}|\widetilde\pi_j^*(\scriptB)|. \\
    \end{gather*}
\end{lemma}

\begin{proof}
    By Properties \ref{Prop:paraball}, we may assume $\mathbf e$ is the standard basis for $\R^n$, $z$ the origin, and $\rho=r_j=r_j^*=1$ for $j=1,\dots, n$. Elementary calculus shows $|\scriptB|\sim1$. Let $\eta=c\delta^2$, where $c>0$ is a small number whose value will be determined later. Consider $\{z_k\}_{k\in\scriptK}=\scriptB\cap(\eta\Z)^{2n+1}$. Then $\#\scriptK\lesssim\delta^{-C}$.
    
    For each $k\in\scriptK$, define $\scriptB_k=\scriptB(z_k,\mathbf e,r,r^*,\delta^2,\{L_j\}_{j=1}^M)$, where $r_i=r_i^*=\delta$ for $1\leq i\leq n$. Then $|\pi_j(\scriptB_k)|=\delta^{A_j}|\pi_j(\scriptB)|$ with $A_j=\dim V_j+n+2$, for $1\leq j\leq M$. We reach similar conclusions for $\widetilde\pi_j,\widetilde\pi_j^*$ with $\widetilde A_j=2n-\widetilde k_j+2$. Choosing a sufficiently small $c>0$ gives $\scriptB\subset\cup_{k\in\scriptK}\scriptB_k$.
\end{proof}

\begin{proof}[Proof of Proposition \ref{P:regular paraball} assuming Proposition \ref{P:regular weak paraball}]
    Let $\Omega\subset\HH^n$. Suppose there exists a paraball $\scriptB$ associated with $\{L_j\}_{j=1}^M$ satisfying \eqref{E:paraball} and \eqref{E:paraball weak proj est}.
    Let $\delta=\varepsilon^C$. Then there is a covering $\{\scriptB_k\}_{k\in\scriptK}$ of paraballs associated with $\{L_j\}_{j=1}^M$ as in Lemma \ref{L:covering}. Choosing a sufficiently large $C$, we can find a paraball in the covering satisfying \eqref{E:paraball} and \eqref{E:paraball proj est}.
\end{proof}

The remaining part of the section is dedicated to proving Proposition \ref{P:regular weak paraball} by adapting Section 7 in \cite{christ2011quasiextremals}. We first prove a generalization of Lemma 4.1 in \cite{christ2011quasiextremals}.

\begin{lemma}[Slicing lemma]\label{L:slicing}
    Let $B\in\R^{2n+1}$ be the unit ball centered at $0$. For $s,u,v\in\R^n$, define
    \[
    \Phi(s,u,v)=(x_0+s+v,y_0+u,t_0+\frac{1}{2}(-s\cdot y_0+(x_0+s)\cdot u-v\cdot(y_0+u))).
    \]
    Let $A:\R^n\longrightarrow\R^n$ be a symmetric invertible linear transformation. Let $s^\prime\in\R^{n-k}$. Suppose $\emptyset\neq\omega\subset\{(a,s^\prime,b)\in A(B)\times\R^{2n}:a\in\R^k,b\in\R^{2n}\}$. Then
    \[
    |\Phi(\omega)|\gtrsim|\det A|^{-1}\int_\omega|P(Au)|dvduds,
    \]
    where $P:\R^n\longrightarrow\R^k$ is the orthogonal projection onto the $k$-dimensional subspace parallel to $A^{-1}(\R^k\times\{s^\prime\})$.
\end{lemma}

\begin{proof}
    Let $\widetilde s = A^{-1}s, \widetilde u = Au, \widetilde v = A^{-1}v,\widetilde\omega\coloneqq\{(\widetilde s,\widetilde u,\widetilde v)\in\R^{3n}:(A\widetilde s,A^{-1}\widetilde u,A\widetilde v)\in\omega\}\subset B\times\R^{2n}$. Then,
    \begin{align*}
        &   \Phi(s,u,v) =   \Phi(A\widetilde s,A^{-1}\widetilde u,A\widetilde v) \\
        =   &   (x_0+A(\widetilde s+\widetilde v), y_0+A^{-1}\widetilde u, t_0+\frac{1}{2}(-(\widetilde s+\widetilde v)^TAy_0+x_0^TA^{-1}\widetilde u+(\widetilde s-\widetilde v)\cdot\widetilde u)) \\
        =   &   \widetilde A\Phi(\widetilde s,\widetilde u,\widetilde v),
    \end{align*}
    where $\widetilde A:\R^{2n+1}\longrightarrow\R^{2n+1}$ is a linear map given by
    \begin{align*}
        \widetilde A(x,y,t) = & (A(x-x_0)+x_0,A^{-1}(y-y_0)+y_0, \\
        &   t+\frac{1}{2}((x-x_0)\cdot y_0-(x-x_0)^TAy_0-x_0\cdot(y-y_0)+x_0^TA^{-1}(y-y_0))).
    \end{align*}
    Since $|\det\widetilde A|\equiv1$, we have $|\Phi(\omega)|=|\Phi(\widetilde\omega)|$.

    \begin{figure}
        \centering
        \includegraphics[width=1\linewidth]{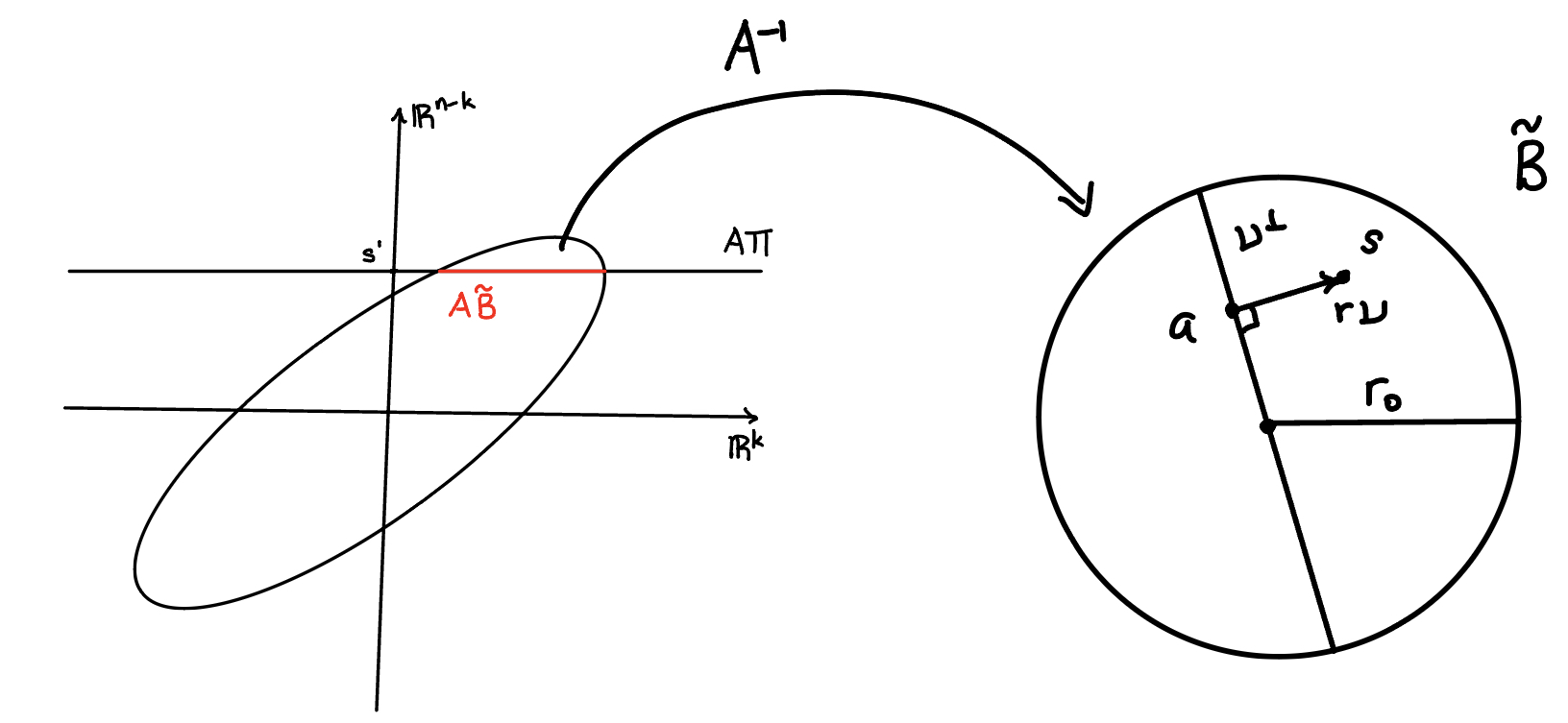}
        \caption{}
        \label{fig:slicing}
    \end{figure}

    Note that $\Pi=A^{-1}(\R^k\times\{s^\prime\})$ is an affine subspace of $\R^n$, and $\widetilde B\coloneqq\Pi\cap B$ is a $k$-dimensional ball of radius $r_0\in(0,1]$ (see Figure \ref{fig:slicing}). Let $D=\{\nu\in S^{n-1}:\nu\parallelsum\Pi\}$, and let $\nu^\perp$ be the $(k-1)$-dimensional affine plane in $\Pi$ orthogonal to $\nu\in D$ and through the center of $\widetilde B$. Then for each $s\in\widetilde B$, there exists a unique pair $(a,r)\in(\widetilde B\cap\nu^\perp)\times[0,r_0)$ such that $s=a+r\nu$. Define $G_{a,\nu}:\R\times\R^n\times\R^n\longrightarrow\HH^n$ by
    \[
    G_{a,\nu}(r,u,v)\coloneqq\Phi(a+r\nu,u,v).
    \]
    Let $\widetilde\omega_{a,\nu}=\{(r,u,v)\in\R^{2n+1}:(a+r\nu,u,v)\in\widetilde\omega\}$. Then $G_{a,\nu}(\widetilde\omega_{a,\nu})\subset\Phi(\widetilde\omega)$. It is straightforward to check that the Jacobian of $G_{a,\nu}$ is $|\nu\cdot u|$. By B\'ezout's Theorem, for each $a\in\nu^\perp$,
    \[
    |\Phi(\widetilde\omega)|\geq|G_{a,\nu}(\widetilde\omega_{a,\nu})|\sim\int_{\widetilde\omega_{a,\nu}}|\nu\cdot u|drdudv.
    \]
    Since $|\widetilde B\cap\nu^\perp|\sim r_0^{k-1}$, we have
    \begin{equation}
        |\Phi(\widetilde\omega)|\sim r_0^{1-k}\int_{\widetilde B\cap\nu^\perp}|\Phi(\widetilde\omega)|da\gtrsim\int_{\widetilde B\cap\nu^\perp}\int_{\widetilde\omega_{a,\nu}}|\nu\cdot u|drdudvda\sim\int_{\widetilde\omega}|\nu\cdot u|dsdudv.
    \end{equation}
    By the same reasoning, we may integrate $|\Phi(\widetilde\omega)|$ against $\nu$ over $D$:
    \begin{align*}
        & |\Phi(\widetilde\omega)|\sim\int_D|\Phi(\widetilde\omega)|d\nu \\
        \gtrsim &   \int_{D}\int_{\widetilde\omega}|\nu\cdot u|dsdudvd\nu = |\det A|^{-1}\int_{D}\int_{\omega}|\nu\cdot(Au)|dsdudvd\nu \\
        =   &   |\det A|^{-1}\int_{\omega}\int_{D}|\nu\cdot(Au)|d\nu dsdudv\sim|\det A|^{-1}\int_{\omega}|P(Au)|dsdudv,
    \end{align*}
    where $P:\R^n\longrightarrow\R^k$ is the orthogonal projection onto the $k$-dimensional subspace parallel to $\Pi$.
\end{proof}

\begin{proof}[Proof of Proposition \ref{P:regular weak paraball}]
    Let $\varepsilon>0$. Suppose $\Omega\subset\HH^n$ is $\varepsilon$-regular. Let
    \begin{gather}
        \alpha_j=\frac{|\Omega|}{|\pi_j(\Omega)|},\qquad1\leq j\leq M; \\
        \beta_j=\frac{|\Omega|}{|\widetilde\pi_j(\Omega)|},\quad\beta_j^*=\frac{|\Omega|}{|\widetilde\pi_j^*(\Omega)|},\qquad1\leq j\leq\widetilde m.
    \end{gather}
    By Proposition \ref{P:semi reg rwt} and $\varepsilon$-regularity, we have
    \begin{equation}\label{E:alpha beta bound}
        \displaystyle\prod_{j=1}^m\alpha_j\lesssim\varepsilon^{-C}\displaystyle\prod_{j=1}^{\widetilde m}\beta_j\lesssim\varepsilon^{-C}\displaystyle\prod_{j=1}^m\alpha_j,\qquad\displaystyle\prod_{j=m+1}^M\alpha_j\lesssim\varepsilon^{-C}\displaystyle\prod_{j=1}^{\widetilde m}\beta_j^*\lesssim\varepsilon^{-C}\displaystyle\prod_{j=m+1}^M\alpha_j.
    \end{equation}
    By Corollary \ref{C:C appr S_1}, there exist a large constant $C$ depending only on $n,\pi_j,p_j$, a refined flow scheme $(z_0,S_1,\scriptF_l,\scriptG_l)$ of $(\Omega,\pi,10)$, and an ellipsoid $\scriptC\subset\R^n$ centered at the origin adapted to $\{L_j\}_{j=1}^m$ satisfying $S_1\subset\scriptC,|\scriptC|\lesssim\varepsilon^{-C}|S_1|$.
    
    Write $S=S_1,\scriptF=\scriptF_1,\scriptG=\scriptG_1$, and $\scriptC=A(B)$, where $A$ is an invertible symmetric linear transformation on $\R^n$, and $B\subset\R^n$ the unit ball centered at the origin. We can further write $A=\mathbf e\diag(r)\mathbf e^T$, where $r=(r_1,\dots,r_n)\in(\R_{>0})^n$, and $\mathbf e=\{\widetilde e_i\}_{i=1}^n$ is orthonormal such that $\{\widetilde e_i\}_{i=1}^{\widetilde k_j}$ span the first $\widetilde k_j$ coordinates for each $1\leq j\leq\widetilde m$. Then
    \begin{gather}
        \label{E:C alpha bound}\alpha_j\lesssim\frac{|S|}{|L_j(S)|}\lesssim\frac{|\scriptC|}{|L_j(\scriptC)|},\quad\alpha_j\lesssim\frac{|\scriptG|}{|L_j(\scriptG)|},\qquad1\leq j\leq m; \\
        \label{E:F alpha bound}\alpha_j\lesssim\frac{|\scriptF|}{|L_j(\scriptF)|},\qquad m<j\leq M; \\
        \label{E:C beta bound}\beta_j\lesssim\frac{|\scriptC|}{|P_j(\scriptC)|}\sim|P_j^\perp(\scriptC)|\sim\displaystyle\prod_{i=1}^{\widetilde k_j}r_i\lesssim\varepsilon^{-C}\beta_j,\qquad1\leq j\leq\widetilde m; \\
        \label{E:SG beta bound para}\beta_j\lesssim|P_j^{-1}(s_j)\cap S|,|P_j^{-1}(v_j)\cap\scriptG|\lesssim\varepsilon^{-C}\beta_j,\qquad\forall s_j\in P_j(S),v_j\in P_j(\scriptG); \\
        \label{E:F beta bound}\beta_j^*\lesssim|P_j^{-1}(u_j)\cap\scriptF|\lesssim\varepsilon^{-C}\beta_j^*,\qquad\forall u_j\in P_j(\scriptF).
    \end{gather}

    \textbf{Step 1.} Apply Lemma \ref{L:slicing} with
    \begin{gather}
        \notag A_j=\mathbf e\diag(r_1,\dots,r_{\widetilde k_j},1,\dots,1)\mathbf e^T, \\
        \notag \omega_j=\{(s,s_j,u,v)\in\R^{3n},(s,s_j)\in S,u\in\scriptF(s,s_j),v\in\scriptG(s,s_j,u)\},\qquad s_j\in P_j(S), \\
        \label{E:phi}\phi_j(s,s_j,u,v)=e^{v\cdot\mathbf X}e^{u\cdot\mathbf Y}e^{(s,s_j)\cdot\mathbf X}(z_0), \\
        \notag P=P_j^\perp.
    \end{gather}
    We claim the following:
    \begin{equation}\label{E:beta j bound}
        \varepsilon^{C_0}\beta_{\widetilde m}\beta_{\widetilde m}^*(\beta_j\beta_j^*)^{\frac{1}{\widetilde k_j}}\lesssim|\phi(\omega_j)|\leq|\Omega|\lesssim\varepsilon^{-C}\beta_{\widetilde m}\beta_{\widetilde m}^*(\beta_j\beta_j^*)^{\frac{1}{\widetilde k_j}},\qquad1\leq j\leq\widetilde m.
    \end{equation}
    To see this, we fix $s_j^\prime\in P_j(S)$ for $j=1,\dots,\widetilde m$. Then
    \[
    |\Omega|\gtrsim\varepsilon^C\beta_j^{-2}\beta_{\widetilde m}\int_{P_j^{-1}(s_j^\prime)\cap S}\int_{\widetilde\scriptF_j(s)}|P_j^\perp(u)|duds,
    \]
    where $\widetilde\scriptF_j(s)=A_j(\scriptF(s))$. By definition of $A_j$, $P_j(\widetilde\scriptF_j(s))=P_j(\scriptF(s))$. Moreover, $|P_j(\scriptF(s))|\gtrsim\varepsilon^C\beta_{\widetilde m}^*/\beta_j^*$ by \eqref{E:F beta bound} and $\widetilde k_{\widetilde m}=n$, and $|P_j^{-1}(u_j)\cap\widetilde F_j(s)|\gtrsim\varepsilon^C\beta_j\beta_j^*$ by \eqref{E:SG beta bound para} and \eqref{E:F beta bound}. Then
    \[
    \int_{\widetilde\scriptF_j(s)}|P_j^\perp(u)|du=\int_{P_j(\scriptF_j(s))}\int_{P_j^{-1}(u_j)\cap\widetilde\scriptF_j(s)}|P_j^\perp(u)|dudu_j\gtrsim\varepsilon^C\beta_j\beta_{\widetilde m}^*(\beta_j\beta_j^*)^\frac{1}{\widetilde k_j}.
    \]
    The lower bound of \eqref{E:beta j bound} is obtained by further applying \eqref{E:SG beta bound para}.

    To see the upper bound, write $|\Omega|=\varepsilon^{-C_j}(\beta_{\widetilde m}\beta_{\widetilde m}^*)(\beta_j\beta_j^*)^{1/\widetilde k_j}$ for $1\leq j\leq\widetilde m$. Then we deduce the following from \eqref{A''} and the lower bound of \eqref{E:beta j bound}:
    \begin{equation}\label{E:C_j lower bound}
        |\Omega|^N\gtrsim\varepsilon^{C_j^\prime}\displaystyle\prod_{j=1}^{\widetilde m}(\beta_{\widetilde m}\beta_{\widetilde m}^*(\beta_j\beta_j^*)^\frac{1}{\widetilde k_j})^{\widetilde k_j\dbtilde q_j}=\varepsilon^{C_j^\prime}\displaystyle\prod_{j=1}^{\widetilde m}(\beta_j\beta_j^*)^{\widetilde q_j},
    \end{equation}
    where $C_j^\prime=-C_j\widetilde k_j\dbtilde q_j+C_0\prod_{r\neq j}\widetilde k_r\dbtilde q_r$. This contradicts quasiextremality when $C_j$ is too large.
    
    \textbf{Step 2.} We now construct the dual ellipsoid $\scriptC_*$. Let $\widetilde\scriptF(s)=A(\scriptF(s))$. Then $\beta_{\widetilde m}\beta_{\widetilde m}^*\lesssim|\widetilde\scriptF(s)|\lesssim\varepsilon^{-C}\beta_{\widetilde m}\beta_{\widetilde m}^*$. We observe that for any $\rho>0$, one has either
    \begin{enumerate}
        \item $\int_{\widetilde\scriptF(s)}|u|du\gtrsim\rho|\widetilde\scriptF(s)|$, or
        \item $|\widetilde\scriptF(s)\cap B_\rho^n(0)|\gtrsim|\widetilde\scriptF(s)|$,
    \end{enumerate}
    where, $B_\rho^n(0)$ denotes the ball in $\R^n$ of radius $\rho$ and centered at the origin. Write
    \begin{equation}
        \rho=\varepsilon^{-C^\prime}(\beta_{\widetilde m}\beta_{\widetilde m}^*)^\frac{1}{n},
    \end{equation}
    where $a>0$ is a large constant to be determined. The dichotomy essentially boils down to the followings:
    \begin{enumerate}
        \item there exists a refinement $S^\dagger$ of $S$ such that $\int_{\widetilde\scriptF(s)}|u|du\gtrsim\rho^{n+1}$ for any $s\in S^\dagger$; or
        \item there exists a refinement $S^\ddagger$ of $S$ such that $|\widetilde\scriptF(s)\cap B_\rho^n(0)|\sim|\widetilde\scriptF(s)|$ for any $s\in S^\ddagger$.
    \end{enumerate}
    Indeed, for sufficiently large $C^\prime$, the first case cannot happen. Otherwise, we obtain
    \[
    |\Omega|\gtrsim\varepsilon^{-nC^\prime}(\beta_{\widetilde m}\beta_{\widetilde m}^*)^\frac{n+1}{n}\gtrsim\varepsilon^{C-nC^\prime}\displaystyle\prod_{j=1}^M\alpha_j^\frac{q_j}{N},
    \]
    from the upper bound of \eqref{E:beta j bound} with $j=\widetilde m$ and Proposition \ref{P:semi reg rwt}. This contradicts quasiextremality for sufficiently large $C^\prime$.
    
    Write $\scriptC_*=A^{-1}(B_\rho^n(0))=A_*(B)$, where $A_*=\mathbf e\diag(r^*)\mathbf e^T$, $r_j^*=\rho/r_j$ for each $j$. Then $|\scriptF(s)\cap\scriptC_*|\sim|\scriptF(s)|$ and $|\scriptC_*|\lesssim\varepsilon^{-C}|\scriptF(s)|$. Hence,
    \begin{gather}
        \label{E:C* alpha bound}\alpha_j\lesssim\frac{|\scriptC_*|}{|L_j(\scriptC_*)|},\qquad m<j\leq M; \\
        \label{E:C* beta bound}\beta_j^*\lesssim\frac{|\scriptC_*|}{|P_j(\scriptC_*)|}\sim|P_j^\perp(\scriptC_*)|\lesssim\varepsilon^{-C}\beta_j^*,\qquad1\leq j\leq\widetilde m.
    \end{gather}
    By \eqref{E:alpha beta bound}, \eqref{E:C* alpha bound}, \eqref{E:C* beta bound}, and Corollary \ref{C:Lj Pj perp finner quasiextremal}, we can enlarge $C_*$ up to $\varepsilon^{-C}$ loss (and by duality, also enlarge $\scriptC$ up to $\varepsilon^{-C}$ loss) so that it is adapted to $\{L_j\}_{j=1}^M$. We now let $\mathbf e$ denote the orthonormal basis along the principal axes of the new $\scriptC$. Note that $\scriptC_*$ is independent of $s$. For each $s\in S$, consider the refined flow scheme $(e^{s\cdot\mathbf X}(z_0),\scriptF(s),\scriptG_l(s,\cdot),\scriptF_{l+1}(s,\cdot))$ of $(\Omega,\pi_*,9)$. By duality of $\scriptC,\scriptC_*$, we can enlarge $C,C^\prime$ if necessary to obtain the following:
    \begin{equation}
        \forall s\in S\quad u\in \scriptF(s)\qquad\scriptG(s,u)\subset\scriptC,\quad|\scriptC|\lesssim\varepsilon^{-C}|\scriptG(s,u)|.
    \end{equation}
    
    \textbf{Step 3.} Now we claim that the paraball determined by $\scriptC,\scriptC_*$ gives the desired bounds \eqref{E:paraball} and \eqref{E:paraball weak proj est}. Let $\scriptB\coloneqq\scriptB(z_0,\mathbf e,r,r^*,\rho,\{L_j\}_{j=1}^M)$. Recalling \eqref{E:phi}, we have $\phi_{\widetilde m}(\omega)\subset\Omega\cap\scriptB$. By \eqref{E:beta j bound}, we have
    \begin{equation}
        |\Omega\cap\scriptB|\gtrsim\varepsilon^C(\beta_{\widetilde m}\beta_{\widetilde m}^*)^\frac{n}{n+1}\gtrsim\varepsilon^C|\Omega|.
    \end{equation}
    Furthermore, by quasiextremality, lower bound of \eqref{E:beta j bound}, and Proposition \ref{P:semi reg rwt}, we obtain
    \begin{equation}\label{E:paraball est}
        |\scriptB|\sim\rho^{n+1}\lesssim\varepsilon^{-C}\displaystyle\prod_{j=1}^{\widetilde m}(\beta_j\beta_j^*)^\frac{\widetilde q_j}{N}\lesssim\varepsilon^{-C}|\Omega|.
    \end{equation}
    $C$ may vary line by line. 

    It remains to show \eqref{E:paraball weak proj est}. Since $\scriptC,\scriptC_*$ are adapted to the $L_j$ (thus the $P_j$), the bound can be deduced from the construction of $\scriptB$, \eqref{E:C alpha bound}, \eqref{E:C beta bound}, \eqref{E:C* alpha bound}, \eqref{E:C* beta bound}, and $|\scriptB|\lesssim\varepsilon^{-C}|\Omega|$.
\end{proof}

\section{Summing over almost disjoint sets}\label{S:summing disjoint sets}
Our next step is to generalize Proposition \ref{P:regular paraball} to semiregular sets. This will be done by partitioning $\varepsilon$-quasiextremal semiregular sets into $\eta$-regular sets whose $\widetilde\pi_j,\widetilde\pi_j^*$ fibers are drastically different. By Proposition \ref{P:regular paraball}, we can approximate the regular sets with paraballs up to $\eta^{-C}$ loss. We claim that these paraballs must be almost disjoint, so the regular sets are also almost disjoint.

\begin{lemma}\label{L:greedy}
    Let $\varepsilon\in(0.1]$. Let $\Omega\subset\HH^n$. Suppose for every refinement $\Omega^\prime$ of $\Omega$, there is a paraball $\scriptB$ associated with $\{L_j\}_{j=1}^M$ such that $(\Omega^\prime,\scriptB)$ satisfy \eqref{E:paraball} and \eqref{E:paraball proj est}. Then there exist a collection of paraballs $\{\scriptB_k\}_{k\in\scriptK}$ associated with $\{L_j\}_{j=1}^M$ and a constant $C>0$ depending only on $n,\pi_j,p_j$ such that
    \begin{gather}
        \label{E:cover}|\Omega\cap(\displaystyle\bigcup_{k\in\scriptK}\scriptB_k)|\sim|\Omega|, \\
        \label{E:cardinality}\#\scriptK\lesssim\varepsilon^{-C}, \\
        \label{E:projection bound pi}|\pi_j(\scriptB_k)|\leq|\pi_j(\Omega)|,\quad1\leq j\leq M, \\
        \label{E:projection bound pi twiddle}|\widetilde\pi_j(\scriptB_k)|\leq|\widetilde\pi_j(\Omega)|,\quad1\leq j\leq\widetilde m, \\
        \label{E:projection bound pi twiddle *}|\widetilde\pi_j^*(\scriptB_k)|\leq|\widetilde\pi_j^*(\Omega)|,\quad1\leq j\leq\widetilde m.
    \end{gather}
\end{lemma}

\begin{proof}
    Let $\Omega_1=\Omega$. By assumption, we can find a paraball $\scriptB_1$ associated with $\{L_j\}_{j=1}^M$ satisfying \eqref{E:projection bound pi}-\eqref{E:projection bound pi twiddle *}. For $k>1$, suppose we have obtained $(\Omega_r,\scriptB_r)$ satisfying \eqref{E:projection bound pi}-\eqref{E:projection bound pi twiddle *} with $|\Omega_r|>\frac{1}{2}|\Omega|$ for $r<k$, and $\{\Omega_r\cap\scriptB_r\}_{r<k}$ disjoint. Let $\Omega_k=\Omega_{k-1}\backslash\scriptB_{k-1}$. Suppose $|\Omega_k|>\frac{1}{2}|\Omega|$. Then there is a paraball $\scriptB_k$ associated with $\{L_j\}_{j=1}^M$ satisfying \eqref{E:projection bound pi}-\eqref{E:projection bound pi twiddle *}. Since $\{\Omega_k\cap\scriptB_k\}$ are disjoint $\varepsilon^C$-refinements of $\Omega$, we must have $|\Omega_k|\leq\frac{1}{2}|\Omega|$ for some $k\lesssim\varepsilon^{-C}$.
\end{proof}

\begin{lemma}\label{L:regular little interaction}
    Let $\varepsilon\in(0,\frac{1}{2}],A>0$, and let $\Omega,\Omega^\prime\subset\HH^n$. Suppose for any refinements $\widetilde\Omega,\widetilde\Omega^\prime$ of $\Omega,\Omega^\prime$, respectively, there are paraballs $\scriptB,\scriptB^\prime$ associated with $\{L_j\}_{j=1}^M$ such that $(\widetilde\Omega,\scriptB),(\widetilde\Omega^\prime,\scriptB^\prime)$ satisfy \eqref{E:paraball} and \eqref{E:paraball proj est}. Let
    \begin{equation}
        \beta_j=\frac{|\Omega|}{|\widetilde\pi_j(\Omega)|},\quad\beta_j^*=\frac{|\Omega|}{|\widetilde\pi_j^*(\Omega)|};\qquad\beta_j^\prime=\frac{|\Omega^\prime|}{|\widetilde\pi_j(\Omega^\prime)|},\quad(\beta_j^*)^\prime=\frac{|\Omega^\prime|}{|\widetilde\pi_j^*(\Omega^\prime)|}.
    \end{equation}
    If there exists some $1\leq j\leq\widetilde m$ such that
    \begin{equation}\label{E:different beta}
        \max\{\frac{\beta_j}{\beta_j^\prime},\frac{\beta_j^\prime}{\beta_j},\frac{\beta_j^*}{(\beta_j^\prime)^*},\frac{(\beta_j^\prime)^*}{\beta_j^*}\}\geq\varepsilon^{-A},
    \end{equation}
    Then we can choose the refinements $\widetilde\Omega,\widetilde\Omega^\prime$ of $\Omega,\Omega^\prime$, respectively, such that    \begin{equation}\label{E:epsilon little interaction}
        \forall1\leq j\leq M\quad|\pi_j(\widetilde\Omega)\cap\pi_j(\widetilde\Omega^\prime)|\lesssim\varepsilon^{C_A}|\pi_j(\widetilde\Omega\cup\widetilde\Omega^\prime)|,
    \end{equation}
    where $C_A>0$ depends only on $n,\pi_j,p_j,A$, and $\lim_{A\rightarrow\infty}C_A=\infty$.
\end{lemma}

\begin{proof}
    Let $\scriptB,\scriptB^\prime$ be two paraballs associated with $\{L_j\}_{j=1}^M$ such that $(\Omega,\scriptB),(\Omega^\prime,\scriptB^\prime)$ satisfy \eqref{E:paraball} and \eqref{E:paraball proj est}. Let $(\scriptC,\scriptC_*),(\scriptC^\prime,\scriptC_*^\prime)$ be the associated ellipsoids. Then there exists some $1\leq j\leq\widetilde m$ such that
    \[
    \max\left\{\frac{|P_j^\perp(\scriptC)|}{|P_j^\perp(\scriptC^\prime)|},\frac{|P_j^\perp(\scriptC^\prime)|}{|P_j^\perp(\scriptC)|},\frac{|P_j^\perp(\scriptC_*)|}{|P_j^\perp(\scriptC_*^\prime)|},\frac{|P_j^\perp(\scriptC_*^\prime)|}{|P_j^\perp(\scriptC_*)|}\right\}\geq\varepsilon^{-C_A}
    \]
    due to \eqref{E:different beta}, \eqref{E:paraball}, \eqref{E:paraball proj est}, and the structure of $\scriptB$, where $C_A>0$ is a constant depending on $n,\pi_j,p_j,A$, and $\lim_{A\rightarrow\infty}C_A=\infty$. We finish the proof by invoking Lemma \ref{L:greedy} and Property \ref{Prop:paraball}.\ref{Prop:paraball little interaction}.
\end{proof}

\begin{lemma}[Almost disjointness\cite{christ2011quasiextremals}]\label{L:almost disjoint}
    Let $\varepsilon\in(0,\frac{1}{2}]$, and let $d>0$ be a positive integer. Suppose $F\subset\R^d$ is a measurable set, and $\{F_\ell\}_{\ell\in\Lambda}$ are measurable subsets of $F$. If $|F_\ell|\gtrsim\varepsilon|F|$ for each $\ell\in\Lambda$, and $|F_{\ell_1}\cap F_{\ell_2}|\lesssim\varepsilon^3|F|$ whenever $\ell_1\neq\ell_2$, then
    \begin{equation}\label{E:almost disjoint}
        \displaystyle\sum_{\ell\in\Lambda}|F_\ell|\lesssim|F|.
    \end{equation}
\end{lemma}

\begin{proof}
    The assumption forces the set of $F_\ell$ to be finite. By H\"older's inequality,
    \begin{align*}
        & \left(|F|^{-1}\displaystyle\sum_{\ell\in\Lambda}|F_\ell|\right)^2 = |F|^{-2}\left(\int_F\displaystyle\sum_{\ell\in\Lambda}\scriptX_{F_\ell}\right)^2 \\[15pt]
        \leq & |F|^{-1}\int_F\left(\displaystyle\sum_{\ell\in\Lambda}\scriptX_{F_\ell}\right)^2 \sim |F|^{-1}\left(\displaystyle\sum_{\ell\in\Lambda}|F_\ell|+\displaystyle\sum_{\ell_1\neq \ell_2}|F_{\ell_1}\cap F_{\ell_2}|\right).
    \end{align*}
    Suppose $(|F|^{-1}\sum_{\ell\in\Lambda}|F_\ell|)^2\lesssim|F|^{-1}\sum_{\ell_1\neq \ell_2}|F_{\ell_1}\cap F_{\ell_2}|$. Applying H\"older's inequality again yields
    \begin{equation}
        \#\Lambda^2\varepsilon^2\lesssim|F|^{-1}\#\Lambda^2\max_{\ell_1\neq \ell_2}|F_{\ell_1}\cap F_{\ell_2}|\lesssim\varepsilon^3\#\Lambda^2.
    \end{equation}
    Contradiction arises for small $\varepsilon$.
\end{proof}

The remainder of the article is essentially repeating application of the following lemma. We prove it by adapting Section 9 of \cite{christ2011quasiextremals} to the multilinear setting.

\begin{lemma}\label{L:summing without a priori est}
    Let $\varepsilon,\eta_j\in(0,\frac{1}{2}]$ for $1\leq j\leq M$. There exist constants $a_j,c>0$ such that the following property holds. Suppose $p_j\in[1,\infty]$ and $F_\ell,H_j,F_{j,\ell}\geq0$ satisfy $\sum_j1/p_j>1,1\leq\ell\lesssim\prod\eta_j^{-1}$, and
    \begin{gather}
        \label{E:without rwt est}\forall l\qquad F_\ell\lesssim\displaystyle\prod_{j=1}^MF_{j,\ell}^\frac{1}{p_j}, \\
        \label{E:without epsilon}\forall l\qquad F_\ell\sim\varepsilon\displaystyle\prod_{j=1}^MH_j^\frac{1}{p_j}, \\[10pt]
        \label{E:without eta}\forall l\quad\forall1\leq j\leq M\quad F_{j,\ell}\lesssim\eta_jH_j, \\[20pt]
        \label{E:without sum}\forall1\leq j\leq M\quad\sum_\ell F_{j,\ell}\lesssim H_j.
    \end{gather}
    Then
    \begin{equation}\label{E:weak summing quasiextremals}
        \displaystyle\sum_\ell F_\ell\lesssim\min\{\displaystyle\prod_{j=1}^M\eta_j^a,\varepsilon\displaystyle\prod_{j=1}^M\eta_j^{-1}\}\displaystyle\prod_{j=1}^MH_j^\frac{1}{p_j}.
    \end{equation}
\end{lemma}

\begin{proof}
    By H\"older's inequality, \eqref{E:without rwt est}, and \eqref{E:without sum},
    \begin{equation}
        \displaystyle\sum_\ell F_\ell\lesssim\left(\displaystyle\sum_\ell\displaystyle\prod_{j=1}^{M-1}F_{j,\ell}^\frac{r}{p_j}\right)^\frac{1}{r}\left(\displaystyle\sum_\ell F_{M,\ell}\right)^\frac{1}{p_M}\lesssim\left(\displaystyle\sum_\ell\displaystyle\prod_{j=1}^{M-1}F_{j,\ell}^\frac{r}{p_j}\right)^\frac{1}{r}H_M^\frac{1}{p_M},
    \end{equation}
    where $r^{-1}=1-p_M^{-1}<\sum_{j=1}^{M-1}p_j^{-1}\eqqcolon s^{-1}$. Applying H\"oler's inequality again and using \eqref{E:without sum}, one has
    \begin{align*}
        & \left(\displaystyle\sum_\ell\displaystyle\prod_{j=1}^{M-1}F_{j,\ell}^\frac{r}{p_j}\right)^\frac{1}{r}=\left(\displaystyle\sum_\ell\displaystyle\prod_{j=1}^{M-1}F_{j,\ell}^\frac{s}{p_j}\displaystyle\prod_{j=1}^{M-1}F_{j,\ell}^\frac{r-s}{p_j}\right)^\frac{1}{r} \\
        \leq & \displaystyle\prod_{j=1}^{M-1}\left(\displaystyle\sum_\ell F_{j,\ell}\right)^\frac{s}{p_jr}\displaystyle\max_\ell\displaystyle\prod_{j=1}^{M-1}F_{j,\ell}^\frac{r-s}{p_jr}\sim\displaystyle\prod_{j=1}^{M-1}\eta_j^\frac{r-s}{p_jr}H_j^\frac{1}{p_j}.
    \end{align*}
    Replacing $M$ with any other index $j$, we can find a small constant $a>0$ such that
    \begin{equation}
        \displaystyle\sum_\ell F_\ell\lesssim\displaystyle\prod_{j=1}^M\eta_j^aH_j^\frac{1}{p_j}.
    \end{equation}

    On the other hand, we deduce from \eqref{E:without epsilon} and $1\leq l\lesssim\prod\eta_j^{-1}$
    \begin{equation}
        \displaystyle\sum_\ell F_\ell\sim\varepsilon\displaystyle\sum_\ell\displaystyle\prod_{j=1}^MH_j^\frac{1}{p_j}\lesssim\varepsilon\displaystyle\prod_{j=1}^M\eta_j^{-1}H_j^\frac{1}{p_j}.
    \end{equation}    
\end{proof}

\section{Summing quasiextremals}\label{S:summing quasiextremals}
\subsection{Quasiextremal semiregular sets}
\begin{proposition}\label{P:semiregular paraball}
    Let $\varepsilon>0$. Suppose $\Omega\subset\HH^n$ is a $1$-semiregular set with respect to $\pi$ that is $\varepsilon$-quasiextremal. Then there exists a paraball $\scriptB$ associated with $\{L_j\}_{j=1}^M$ such that \eqref{E:paraball} and \eqref{E:paraball proj est} hold true.
\end{proposition}

\begin{proof}
    Without loss of generality, we may assume $p_j<\infty$. For $\ell\in\Z^{\widetilde m}$, define
    \begin{gather}
        \Omega_\ell=\{z\in\Omega:\forall j=1,\dots,\widetilde m\quad2^{\ell_j}\leq|(\widetilde T_j^*)^\Omega(z)|<2^{\ell_j+1}\}.
    \end{gather}
    Then $\Omega_\ell$ is $1$-semiregular with respect to $\pi_*$, and $\Omega=\sqcup_\ell\Omega_\ell$ (up to a null set). Proposition \ref{P:semi reg rwt} yields
    \begin{equation}
        |\Omega_\ell|\lesssim\displaystyle\prod_{j=1}^M|\pi_j(\Omega_\ell)|^\frac{1}{p_j}.
    \end{equation}
    Furthermore, there exists some dyadic number $\eta_0\in(0,0.5]$ such that
    \begin{equation}\label{E:semi reg epsilon}
        |\Omega_\ell|\sim\eta_0\prod_{j=1}^M|\pi_j(\Omega)|^\frac{1}{p_j}\geq\eta_0\displaystyle\prod_{j=1}^M|\pi_j(\Omega_\ell)|^\frac{1}{p_j}.
    \end{equation}
    Then we deduce from Lemmas \ref{L:semiregular refinement} that $\Omega_\ell$ are $C\eta_0$-regular. If $\eta_0\gtrsim\varepsilon^C$, then we can invoke Proposition \ref{P:regular paraball} to approximate $\Omega_\ell$ with paraballs, which also well approximates $\Omega$ up to $\varepsilon$-loss. Suppose we always have $\eta_0\leq\varepsilon^{C_0}$. It suffices to show that this is impossible when $C_0$ is too large.
    
    Let $\Lambda$ denote the set of indices $\ell$. Consider the partition $\Lambda=\sqcup_b\Lambda_b$, where $1\leq b\lesssim(\log(\eta_0^{-1}))^{\widetilde m}$, and for any $b$, every pair of distinct $\ell,\ell^\prime\in\Lambda_b$ enjoys the large gap property $|\ell_j-\ell_j^\prime|\geq A\log(\eta_0^{-1})$ for some $j$ and sufficiently large $A>0$. Applying Proposition \ref{P:regular paraball} and Lemmas \ref{L:regular refinement} and \ref{L:regular little interaction} yields refinements $\Omega_\ell^\prime$ of $\Omega_\ell$ such that
    \begin{equation}
        \forall\ell,\ell^\prime\in\Lambda_b\qquad\ell\neq\ell^\prime\implies|\pi_j(\Omega_\ell^\prime)\cap\pi_j(\Omega_{\ell^\prime}^\prime)|\lesssim\eta_0^{C_A}|\pi_j(\Omega)|.
    \end{equation}
    Assume further a dyadic number $\eta_j\in(0,\frac{1}{2}]$ such that for each $\ell$,
    \begin{equation}\label{E:semi reg eta}
        \forall1\leq j\leq M\quad|\pi_j(\Omega_\ell^\prime)|\sim\eta_j|\pi_j(\Omega)|.
    \end{equation}
    We can assume $\eta_j\gtrsim\eta_0^C$ for all $j$. Otherwise, Proposition \ref{P:semi reg rwt} yields
    \begin{equation}
        |\Omega_\ell^\prime|\lesssim\displaystyle\prod_{j=1}^M|\pi_j(\Omega_\ell^\prime)|^\frac{1}{p_j}\lesssim\eta_0^{\min\frac{C}{p_j}}\displaystyle\prod_{j=1}^M|\pi_j(\Omega)|^\frac{1}{p_j},
    \end{equation}
    contradicting \eqref{E:semi reg epsilon} for sufficiently large $C$. Note that $C_A\geq 3C$ for sufficiently large $A$. Let $\overrightarrow\eta=\{\eta_j\}_{j=0}^M$. Let $\Lambda(b,\overrightarrow\eta)$ denote the set of indices $\ell\in\Lambda_b$ satisfying \eqref{E:semi reg eta}, and let $\Lambda(\overrightarrow\eta)=\sqcup_b\Lambda(b,\overrightarrow\eta)$. Lemma \ref{L:almost disjoint} implies
    \begin{equation}
        \sum_{\ell\in\Lambda(b,\overrightarrow\eta)}|\pi_j(\Omega_\ell^\prime)|\lesssim|\pi_j(\Omega)|,\qquad1\leq j\leq M.
    \end{equation}
    Note that $\#\Lambda(b,\overrightarrow\eta)\lesssim\min\eta_j^{-1}$. Applying Lemma \ref{L:summing without a priori est} with $F_\ell=|\Omega_\ell^\prime|,H_j=|\pi_j(\Omega)|,F_{j,\ell}=|\pi_j(\Omega_\ell^\prime)|$ and recalling \eqref{E:Lp improving} yields
    \begin{equation}
        \displaystyle\sum_{\Lambda(b,\overrightarrow\eta)}|\Omega_\ell|\lesssim\min\{\displaystyle\prod_{j=1}^M\eta_j^a,\eta_0\displaystyle\prod_{j=1}^M\eta_j^{-1}\}\displaystyle\prod_{j=1}^M|\pi_j(\Omega)|^\frac{1}{p_j}.
    \end{equation}
    Summing over $b$ incurs a harmless $(\log(\eta_0^{-1}))^{\widetilde m}$-loss and yields
    \begin{equation}
        \displaystyle\sum_{\Lambda(\overrightarrow\eta)}|\Omega_\ell|\lesssim\min\{\displaystyle\prod_{j=1}^M\eta_j^a,\eta_0^b\}\displaystyle\prod_{j=1}^M|\pi_j(\Omega)|^\frac{1}{p_j},
    \end{equation}
    for some constants $a,b>0$. Summing over the dyadic numbers $\eta_j\in[\varepsilon^C,\frac{1}{2}],\eta_0\leq\varepsilon^{C_0}$ contradicts $\varepsilon$-quasiextremality of $\Omega$ for sufficiently large $C_0$.
\end{proof}

\subsection{The restricted weak-type inequality}
\begin{proposition}\label{P:weak}
Suppose $\mathbf p\in([1,\infty]\cap\Q)^n$ satisfies the conditions \eqref{A}, \eqref{B twiddle}, and \eqref{C}. Then the \textit{restricted weak-type estimate}
\begin{equation} \label{E:RWT ineq}
|\Omega| \lesssim \prod_{j=1}^{M} |\pi_j(\Omega)|^\frac{1}{p_j}
\end{equation}
holds uniformly over all Lebesgue measurable subsets $\Omega\subset\HH^n$.
\end{proposition}

\begin{proof}
    Without loss of generality, we may assume $p_j<\infty$ and $\Omega$ bounded. For $\ell\in\Z^{\widetilde m}$, define
    \begin{equation}
        \Omega_\ell=\{z^\prime\in\Omega:\forall j=1,\dots,\widetilde m\quad2^{\ell_j}\leq|\widetilde T_j^\Omega(z)|<2^{\ell_j+1}\}.
    \end{equation}
    Then $\Omega_\ell$ is $1$-semiregular with respect to $\pi$, and $\Omega=\sqcup\Omega_\ell$ (up to a null set). Proposition \ref{P:semi reg rwt} yields
    \begin{equation}
        |\Omega_\ell|\lesssim\displaystyle\prod_{j=1}^M|\pi_j(\Omega_\ell)|^\frac{1}{p_j}.
    \end{equation}
    Furthermore, there exists some dyadic number $\varepsilon\in(0,0.5]$ such that
    \begin{equation}\label{E:rwt epsilon}
        |\Omega_\ell|\sim\varepsilon\prod_{j=1}^M|\pi_j(\Omega)|^\frac{1}{p_j}\geq\varepsilon\displaystyle\prod_{j=1}^M|\pi_j(\Omega_\ell)|^\frac{1}{p_j}.
    \end{equation}
    
    Let $\Lambda$ denote the set of indices $\ell$. Consider the partition $\Lambda=\sqcup_b\Lambda_b$, where $1\leq b\lesssim(\log(\eta_0^{-1}))^{\widetilde m}$, and for any $b$, every pair of distinct $\ell,\ell^\prime\in\Lambda_b$ enjoys the large gap property $|\ell_j-\ell_j^\prime|\geq A\log(\eta_0^{-1})$ for some $j$ and sufficiently large $A>0$. Then one can invoke Proposition \ref{P:semiregular paraball} and Lemmas \ref{L:semiregular refinement} and \ref{L:regular little interaction} to obtain refinements $\Omega_\ell^\prime$ of $\Omega_\ell$ such that if $\ell\neq\ell^\prime$, then for all $1\leq j\leq M$,
    \begin{equation}
        |\pi_j(\Omega_\ell^\prime)\cap\pi_j(\Omega_{\ell^\prime}^\prime)|\lesssim\varepsilon^{C_A}|\pi_j(\Omega)|,\qquad\lim_{A\rightarrow\infty}C_A=\infty.
    \end{equation}
    
    Assume further a dyadic number $\eta_j\in(0,\frac{1}{2}]$ such that for each $\ell$,
    \begin{equation}\label{E:rwt eta}
        \forall1\leq j\leq M\quad|\pi_j(\Omega_\ell^\prime)|\sim\eta_j|\pi_j(\Omega)|.
    \end{equation}
    We can assume $\eta_j\gtrsim\varepsilon^C$ for all $j$. Otherwise, Proposition \ref{P:semi reg rwt} yields
    \[
    |\Omega_\ell^\prime|\lesssim\displaystyle\prod_{j=1}^M|\pi_j(\Omega_\ell^\prime)|^\frac{1}{p_j}\lesssim\varepsilon^{\min\frac{C}{p_j}}\displaystyle\prod_{j=1}^M|\pi_j(\Omega)|^\frac{1}{p_j},
    \]
    contradicting \eqref{E:rwt epsilon} for sufficiently large $C$. Note that $C_A\geq 3C$ for sufficiently large $A$. Let $\overrightarrow\eta=\{\eta_j\}_{j=1}^M$. Let $\Lambda(b,\varepsilon,\overrightarrow\eta)$ denote the set of indices $\ell\in\Lambda_b$ satisfying \eqref{E:rwt eta}, and let $\Lambda(\varepsilon,\overrightarrow\eta)=\sqcup_b\Lambda(b,\varepsilon,\overrightarrow\eta)$. Lemma \ref{L:almost disjoint} implies
    \[
    \sum_{\ell\in\Lambda(b,\varepsilon,\overrightarrow\eta)}|\pi_j(\Omega_\ell^\prime)|\lesssim|\pi_j(\Omega)|,\qquad1\leq j\leq M.
    \]
    Note that $\#\Lambda(b,\overrightarrow\eta)\lesssim\min\eta_j^{-1}$. Applying Lemma \ref{L:summing without a priori est} with $F_\ell=|\Omega_\ell^\prime|,H_j=|\pi_j(\Omega)|,F_{j,\ell}=|\pi_j(\Omega_\ell^\prime)|$ and recalling \eqref{E:Lp improving} yields
    \begin{equation}
        \displaystyle\sum_{\Lambda(b,\varepsilon,\overrightarrow\eta)}|\Omega_\ell|\lesssim\min\{\displaystyle\prod_{j=1}^M\eta_j^a,\varepsilon\displaystyle\prod_{j=1}^M\eta_j^{-1}\}\displaystyle\prod_{j=1}^M|\pi_j(\Omega)|^\frac{1}{p_j}.
    \end{equation}
    Summing over $b$ incurs a harmless $(\log(\varepsilon^{-1}))^{\widetilde m}$-loss and yields
    \begin{equation}
        \displaystyle\sum_{\Lambda(\varepsilon,\overrightarrow\eta)}|\Omega_\ell|\lesssim\min\{\displaystyle\prod_{j=1}^M\eta_j^a,\varepsilon^b\}\displaystyle\prod_{j=1}^M|\pi_j(\Omega)|^\frac{1}{p_j},
    \end{equation}
    for some positive constants $a,b$. Summing over the dyadic numbers $\eta_j,\varepsilon\in(0,\frac{1}{2}]$ gives \ref{E:RWT ineq}.
\end{proof}

The above proof implies the following theorem that subsumes Proposition \ref{P:regular paraball} and Corollary \ref{P:semiregular paraball}:

\begin{theorem}\label{T:paraball}
    Let $\varepsilon>0$. Suppose $\Omega\subset\HH^n$ is $\varepsilon$-quasiextremal, then there exists a paraball $\scriptB$ associated with $\{L_j\}_{j=1}^M$ such that \eqref{E:paraball} and \eqref{E:paraball proj est} hold true.
\end{theorem}

The following is a variant of Lemma \ref{L:regular little interaction}.
\begin{lemma}\label{L:little interaction}
    Let $\varepsilon\in(0,\frac{1}{2}],A>0$. Suppose $\{\scriptX_{E_j}\}_{j=1}^M,\{\scriptX_{E_j^\prime}\}_{j=1}^M$ are both $\varepsilon$-quasiextremal. Suppose there exists $1\leq j\leq M$ such that
    \begin{equation}\label{E:different volume}
        \max\{\frac{|E_j|}{|E_j^\prime|},\frac{|E_j^\prime|}{|E_j|}\}\geq\varepsilon^{-A}.
    \end{equation}
    Then there exist $\{\scriptX_{\widetilde E_j}\}_{j=1}^M,\{\scriptX_{\widetilde E_j^\prime}\}_{j=1}^M$ such that
    \[
    \scriptM(\{\scriptX_{\widetilde E_j}\}_{j=1}^M)\sim\scriptM(\{\scriptX_{E_j}\}_{j=1}^M),\qquad\scriptM(\{\scriptX_{\widetilde E_j^\prime}\}_{j=1}^M)\sim\scriptM(\{\scriptX_{E_j^\prime}\}_{j=1}^M),
    \]
    and
    \[
    \forall1\leq j\leq M\quad|\widetilde E_j\cap\widetilde E_j^\prime|\lesssim\varepsilon^{C_A}\max\{|E_j|,|E_j^\prime|\},
    \]
    where $C_A>0$ is a constant depending on $n,\pi_j,p_j,A$, and $\lim_{A\rightarrow\infty}C_A=\infty$.
\end{lemma}

\begin{proof}
    Let $\Omega=\cap_{j=1}^M\pi_j^{-1}(E_j),\Omega^\prime=\cap_{j=1}^M\pi_j^{-1}(E_j^\prime)$. By Theorem \ref{T:paraball}, there exist paraballs $\scriptB,\scriptB^\prime$ associated with $\{L_j\}_{j=1}^M$ such that $(\Omega,\scriptB),(\Omega^\prime,\scriptB^\prime)$ satisfy \eqref{E:paraball} and \eqref{E:paraball proj est}. Then
    \begin{equation}
        \varepsilon^C|E_j|\lesssim|\pi_j(\scriptB)|\lesssim|E_j|,\qquad\varepsilon^C|E_j^\prime|\lesssim|\pi_j(\scriptB^\prime)|\lesssim|E_j^\prime|.
    \end{equation}
    Let $(\scriptC,\scriptC_*),(\scriptC^\prime,\scriptC_*^\prime)$ be the associated ellipsoids. We obtain from the structure of paraballs, duality of the ellipsoids, and \eqref{E:different volume} that for some $1\leq j\leq M$,
    \begin{equation}
        \max\left\{\frac{|L_j(\scriptC)|}{|L_j(\scriptC^\prime)|},\frac{|L_j(\scriptC^\prime)|}{|L_j(\scriptC)|},\frac{|L_j(\scriptC_*)|}{|L_j(\scriptC_*^\prime)|},\frac{|L_j(\scriptC_*^\prime)|}{|L_j(\scriptC_*)|},\frac{|\scriptC|}{|\scriptC^\prime|},\frac{|\scriptC^\prime|}{|\scriptC|}\right\}\geq\varepsilon^{-C_A},
    \end{equation}
    where $C_A>0$ is a constant depending on $n,\pi_j,p_j,A$, and $\lim_{A\rightarrow\infty}C_A=\infty$. We finish the proof by invoking Lemmas \ref{L:quasiextremal refinement} and \ref{L:greedy} and Property \ref{Prop:paraball}.\ref{Prop:paraball little interaction}.
\end{proof}

\subsection{The strong-type inequality}\label{S:strong}
For simplicity, we allow a mixed bag of functions and sets in the input of \eqref{E:main ml form}. If a set $E$ is in the input, it is $\scriptX_E$ we are actually feeding to the operator. For example, if $M=2$, then we write $\scriptM(f,E)\coloneqq\scriptM(f,\scriptX_E)$.

First, we claim that it suffices to prove \eqref{E:base case} for \textit{dyadic simple functions} $f_j=\sum_{\ell\in\Z}2^\ell\scriptX_{E_{\ell,j}}$, where $\{E_{\ell,j}\}_\ell$ are disjoint Lebesgue measurable subsets of $\pi_j(\HH^n)$.

\begin{lemma}\label{L:simple}
Suppose \eqref{E:base case} holds uniformly for dyadic simple functions. Then it holds uniformly for compactly supported continuous functions.
\end{lemma}

\begin{proof}
It suffices to consider nonnegative compactly supported continuous functions $f_j$ defined on $\pi_j(\HH^n)$ for $1\leq j\leq M$. Define
\[
\widetilde f_j\coloneqq\displaystyle\sum_{\ell\in\Z}2^\ell\scriptX_{E_{\ell,j}},\qquad E_{\ell,j}\coloneqq\{2^{\ell-1}<f_j\leq2^\ell\}.
\]
Then $\frac{1}{2}\widetilde f_j\leq f_j\leq\widetilde f_j$, and $\|f_j\|_{L^p}\sim\|\widetilde f_j\|_{L^p}$ uniformly over $j$. Then one has
\begin{equation}
    \scriptM(f_1,\dots,f_{M})\leq\scriptM(\widetilde f_1,\dots,\widetilde f_{M})\lesssim\displaystyle\prod_{j=1}^{M}\|\widetilde f_j\|_{L^{p_j}}\sim\displaystyle\prod_{j=1}^{M}\|f_j\|_{L^{p_j}}.
\end{equation}
\end{proof}

\begin{proof}[Proof of Theorem \ref{T:base case}]
    Let $p_j$ be as in Theorem \ref{T:base case}. Then we have $p_j\in[1,\infty]$. By Lemma \ref{L:simple}, it suffices to prove \eqref{E:base case} for dyadic simple functions. By scaling invariance, we may assume $\|f_j\|_{p_j}=1$ for each $j$.
    
    For each $\ell=(\ell_1,\dots,\ell_M)\in\Z^M$, set $\scriptM_\ell=\scriptM(2^{\ell_1}\scriptX_{E_{\ell_1,1}},\dots,2^{\ell_M}\scriptX_{E_{\ell_M,M}})$. By Proposition \ref{P:weak}, we have
    \[
    \scriptM_\ell\lesssim\displaystyle\prod_{j=1}^M(2^{\ell_jp_j}|E_{\ell_j,j}|)^\frac{1}{p_j}\leq1.
    \]
    Let $\varepsilon,\eta_j\in(0,\frac{1}{2}]$ for $1\leq j\leq M$ be dyadic numbers. Write $\overrightarrow\eta=\{\eta_j\}_{j=1}^M$. Let $\Lambda(\varepsilon,\overrightarrow\eta)$ denote the set of $\ell\in\Z^M$ satisfying
    \begin{gather}
        \label{E:epsilon}\scriptM_\ell\sim\varepsilon, \\
        \label{E:eta}\forall1\leq j\leq M\qquad2^{\ell_jp_j}|E_{\ell_j,j}|\sim\eta_j.
    \end{gather}
    Then $\Omega_\ell$ is $\varepsilon$-quasiextremal, and $\#\Lambda(\varepsilon,\overrightarrow\eta)\lesssim\prod\eta_j^{-1}$. Moreover, $\eta_j\gtrsim\varepsilon^C$ for all $j$. Otherwise, Proposition \ref{P:weak} yields
    \[
    \scriptM_\ell\lesssim\displaystyle\prod_{j=1}^M(2^{\ell_jp_j}|E_{\ell_j,j}|)^\frac{1}{p_j}\lesssim\varepsilon^{\min\frac{C}{p_j}},
    \]
    contradicting \eqref{E:epsilon} for sufficiently large $C$.
    
    Consider the partition $\Z^M=\sqcup_b\Lambda_b$, where $1\leq b\lesssim\log(\varepsilon^{-1})$, and within each $\Lambda_b$, every pair of distinct $\ell,\ell^\prime$ enjoys the large gap property $|\ell_j-\ell_j^\prime|\geq A(\log(\varepsilon^{-1}))^M$ for some $j$ and sufficiently large $A>0$. For fixed $b$, we can invoke Lemma \ref{L:little interaction} to obtain $E_j^\ell\subset E_{\ell_j,j}$ for each $\ell\in\Lambda_b,1\leq j\leq M$ such that $\scriptM_\ell^\prime\sim\scriptM_\ell$, where $\scriptM_\ell^\prime=\scriptM(2^{\ell_1}\scriptX_{E_1^\ell},\dots,2^{\ell_M}\scriptX_{E_M^\ell})$. If $\ell\neq\ell^\prime,\ell_j=\ell_j^\prime$ for some $j$, then
    \[
    |E_j^\ell\cap E_j^{\ell^\prime}|\lesssim\varepsilon^{C_A}|E_{\ell_j,j}|,\qquad\lim_{A\rightarrow\infty}C_A=\infty.
    \]
    By the same argument that shows $\eta_j\gtrsim\varepsilon^C$, we have $|E_j^\ell|\gtrsim\varepsilon^C$. Note that $C_A\geq 3C$ for sufficiently large $A$. Lemma \ref{L:almost disjoint} implies $\sum_{\ell:\ell_j=k}|E_j^\ell|\lesssim|E_{k,j}|$ for each $1\leq j\leq M,k\in\Z$. Summing over $k$ yields $\sum_{\ell}2^{\ell_jp_j}|E_j^\ell|\lesssim1$. Applying Lemma \ref{L:summing without a priori est} with $F_{\ell}=\scriptM_\ell,H_j=1,F_{j,\ell}=2^{\ell_jp_j}|E_j^\ell|$ and recalling \eqref{E:Lp improving} yields
    \begin{equation}
        \displaystyle\sum_{\Lambda(b,\varepsilon,\overrightarrow\eta)}\scriptM_\ell\lesssim\min\{\displaystyle\prod_{j=1}^M\eta_j^a,\varepsilon\displaystyle\prod_{j=1}^M\eta_j^{-1}\},
    \end{equation}
    where $\Lambda(b,\varepsilon,\overrightarrow\eta)\coloneqq\Lambda_b\cap\Lambda(\varepsilon,\overrightarrow\eta)$. Summing over $b$ incurs a $(\log(\varepsilon^{-1}))^M$-loss and yields
    \begin{equation}
        \displaystyle\sum_{\Lambda(\varepsilon,\overrightarrow\eta)}\scriptM_\ell\lesssim\min\{\displaystyle\prod_{j=1}^M\eta_j^a,\varepsilon^b\},
    \end{equation}
    for some positive constants $a,b$. We finish the proof by summing over the dyadic numbers $\eta_j,\varepsilon\in(0,\frac{1}{2}]$.    
\end{proof}

\appendix
\section{Counterexamples}\label{A:counter ex}
\begin{exmp}\label{Ex:beta>alpha}
    Recall Subsection \ref{S:proof ex aniso}. Consider
    \[
    \Omega=\bigcup_{n=1}^N\{(x,y,t-\frac{1}{2}x\cdot y)\in\HH^4:x_1\in[0,n^2],x^\prime\in[0,\frac{1}{n}]^3,(y,t)\in(n,n+1)^5\}.
    \]
    Then $|\Omega|\sim\log N$, each $|\pi_j^{-1}(z)\cap\Omega|\sim1$ for $1\leq j\leq n$, while $|\widetilde\pi_1^{-1}(x^\prime, y,t-\frac{1}{2}x^\prime\cdot y^\prime)\cap\Omega|\sim n^2$ for $(x^\prime,y,t)\in[0,\frac{1}{n}]^3\times(n,n+1)^5$ and $|\widetilde\pi_2^{-1}(y,t-\frac{1}{2}x\cdot y)\cap\Omega|\sim\frac{1}{n}$ for $(y,t)\in(n,n+1)^5,$. Therefore, we have $\alpha_j\sim1$, but $\beta_1\sim\log N,\beta_2\sim\frac{\log N}{N}$, making
    \[
    \beta_1\beta_2^2\ll\displaystyle\prod_{j=1}^3\alpha_j.
    \]
\end{exmp}

\begin{exmp}\label{Ex:fibers<<projections}
    Consider the parallelogram (see Figure \ref{Fig:gen finner})
    \[
    \omega=\{(x+y,\frac{x}{N}):x\in[0,N],y\in[0,1]\},
    \]
    and consider $l(x)=x_1,l^\perp(x)=x_2$ for $x\in\R^2$. Then
    \[
    |l(\omega)|=N+1\gg|\frac{|\omega|}{|l^\perp(\omega)|}=1.
    \]
\end{exmp}

\begin{figure}
\centering
\includegraphics[scale=0.5]{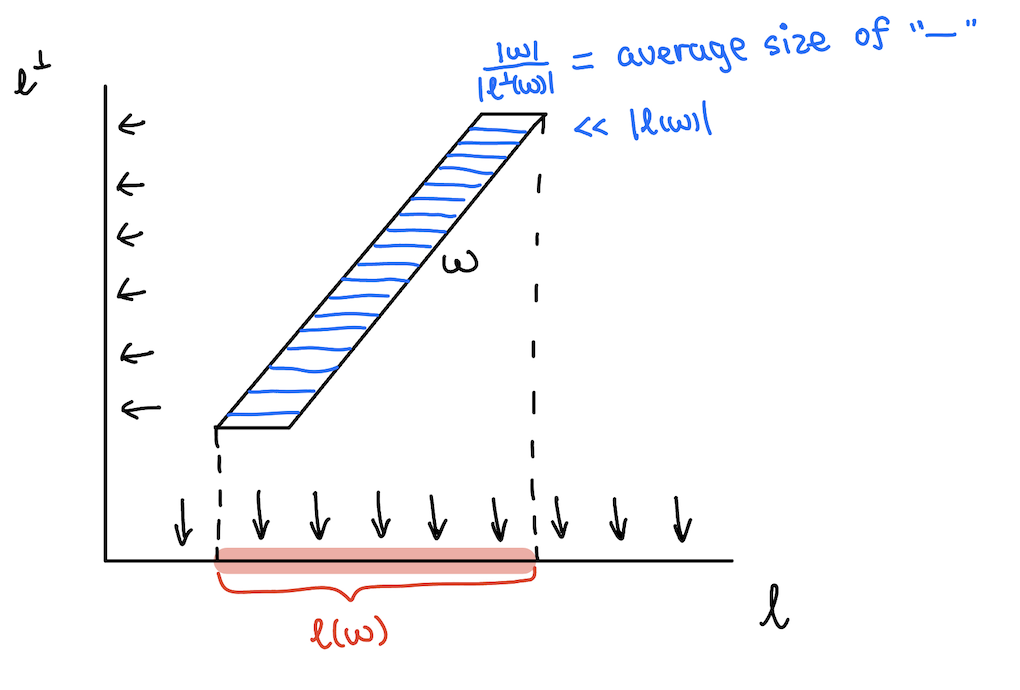}
\caption{}\label{Fig:gen finner}
\end{figure}

\begin{exmp}\label{Ex:counter ex nest gen finner}
    Consider $l_1(x)=(x_3,x_4),l_2(x)=x_4,l_3(x)=x,l_4(x)=x_4$ for $x\in\R^4$, and
    \[
    \omega=\displaystyle\bigcup_{n=1}^N[0,\frac{1}{n}]\times[0,n]^2\times[n,n+1].
    \]
    Then $|\omega|\sim N^2$, and
    \[
    \frac{|\omega|}{|l_1(\omega)|}=\frac{|\omega|}{|l_2(\omega)|}=1,\quad|l_3(\omega)|\sim N^2,\quad|l_4(\omega)|=N.
    \]
    Thus,
    \[
    |\omega|^2\gg\frac{|\omega|}{|l_1(\omega)|}\frac{|\omega|}{|l_2(\omega)|}|l_3(\omega)||l_4(\omega)|.
    \]
\end{exmp}

\begin{exmp}\label{Ex:counter ex reg gen finner}
    Consider $l_1(x)=(x_2,x_3,x_4),l_2(x)=x_4,l_3(x)=(x_2,x_4),l_4(x)=(x_3,x_4)$ for $x\in\R^4$, and
    \[
    \omega=\displaystyle\bigcup_{n=1}^N[0,n]\times[0,\frac{1}{n}]^2\times[n,n+1].
    \]
    Then $|\omega|\sim\log N$, and
    \[
    \frac{|\omega|}{|l_1(\omega)|}\sim\log N,\quad\frac{|\omega|}{|l_2(\omega)|}\sim\frac{\log N}{N},\quad\frac{|\omega|}{|l_3(\omega)|}=\frac{|\omega|}{|l_4(\omega)|}=1.
    \]
    Thus,
    \[
    |\omega|^2\gg\frac{|\omega|}{|l_1(\omega)|}\frac{|\omega|}{|l_2(\omega)|}|l_3(\omega)||l_4(\omega)|.
    \]
\end{exmp}



\end{document}